\documentclass[11pt,reqno]{amsart}
\newcommand{\mysection}[1]{\section{#1}
      \setcounter{equation}{0}}

\newcommand\cbrk{\text{$]$\kern-.15em$]$}} 
\newcommand\opar{\text{\raise.2ex\hbox{${\scriptstyle | }$}\kern-.34em$($} }

\usepackage{color}
\usepackage{amsmath,amsthm,amssymb,amsfonts,enumerate,color,enumerate}
\usepackage[pdftex]{graphicx}

\usepackage{verbatim}
\usepackage[utf8]{inputenc}
\usepackage[spanish,USenglish]{babel} 
\usepackage{color}
\usepackage{tikz}
\usepackage[hidelinks]{hyperref}
\usepackage{mathtools}
\usepackage{bbm}
\usepackage{mathabx}
\allowdisplaybreaks

\DeclareMathOperator*{\esssup}{ess\,sup}
\DeclareMathOperator*{\essup}{ess\,sup}

\newtheorem{theorem}{Theorem}[section]
\newtheorem{lemma}[theorem]{Lemma}

\newtheorem{corollary}[theorem]{Corollary}

\theoremstyle{definition}
\newtheorem{assumption}{Assumption}[section]
\newtheorem{definition}{Definition}[section]

\theoremstyle{remark}
\newtheorem{remark}{Remark}[section]

\newcommand{\F}{\mathcal{{F}}}
\newcommand{\R}{\mathbb{{R}}}

\newcommand\bB{\mathbb{B}}

\newcommand\bL{\mathbb{L}}
\newcommand\bR{\mathbb{R}}

\newcommand\bN{\mathbb{N}}
\newcommand\bM{\mathbb{M}}
\newcommand\bU{\mathbb{U}}
\newcommand\bV{\mathbb{V}}

\newcommand\bS{\mathbb{S}}
\newcommand\bT{\mathbb{T}}

\newcommand\bW{\mathbb{W}}

\newcommand\frc{\mathfrak{c}}

\newcommand\frM{\mathfrak{M}}
\newcommand\frZ{\mathfrak{Z}}
\newcommand\frz{\mathfrak{z}}

\newcommand\cB{\mathcal{B}}

\newcommand\cF{\mathcal{F}}
\newcommand\cG{\mathcal{G}}
\newcommand\cH{\mathcal{H}}

\newcommand\cL{\mathcal{L}}
\newcommand\cM{\mathcal{M}}

\newcommand\cP{\mathcal{P}}
\newcommand\cO{\mathcal{O}}
\newcommand\cQ{\mathcal{Q}}
\newcommand\cR{\mathcal{R}}
\newcommand\cS{\mathcal{S}}

\newcommand\cZ{\mathcal{Z}}

\newcommand{\ep}{{(\varepsilon)}}

\newcommand{\ke}{k_{\varepsilon}}

\newcommand{\E}{\mathbb{E}}

\newcommand{\Nte}{\tilde{N}_{1}}

\newcommand{\vp}{\varphi}

\newcommand{\bRY}{{\mathbb{R}^{d'}}}

\makeatletter
 \newcommand{\sumstar}
 {\operatornamewithlimits{\sum@\kern-.2em\raise1ex\hbox{*}}}
 \makeatother

\begin{document}

\author[A. Davie]{Alexander Davie}
\address{School of Mathematics and Maxwell Institute, 
University of Edinburgh, Scotland, United Kingdom.}
\email{a.davie@ed.ac.uk}

\author[F. Germ]{Fabian Germ}
\address{School of Mathematics and Maxwell Institute, 
University of Edinburgh, Scotland, United Kingdom.}
\email{fgerm@ed.ac.uk}

\author[I. Gy\"ongy]{Istv\'an Gy\"ongy}
\address{School of Mathematics and Maxwell Institute, 
University of Edinburgh, Scotland, United Kingdom.}
\email{i.gyongy@ed.ac.uk}

\keywords{Nonlinear filtering, random measures, L\'evy processes}

\subjclass[2020]{Primary  60G35, 60H15; Secondary 60G57, 60H20}

\begin{abstract} 
A partially observed jump diffusion $Z=(X_t,Y_t)_{t\in[0,T]}$
given by a stochastic differential equation driven by Wiener processes 
and  Poisson martingale measures is considered when the coefficients of the equation satisfy appropriate Lipschitz and growth conditions. Under general conditions it is shown that the conditional density of the unobserved 
component $X_t$ given the observations $(Y_s)_{s\in[0,t]}$ exists 
and belongs to $L_p$ if the conditional density of $X_0$ given $Y_0$ 
exists and belongs to $L_p$. 
\end{abstract}

\title[Filtering density]{On partially observed jump diffusions II. 
The filtering density
}

\maketitle

\tableofcontents

\mysection{Introduction}

As in \cite{GG2022} we consider a partially observed jump diffusion 
$Z=(X_t,Y_t)_{t\in[0,T]}$ satisfying 
the system of stochastic differential equations 
\begin{equation}                                                                          
\begin{split}
    dX_t    &= b(t,Z_t)dt + \sigma(t,Z_t)dW_t + \rho(t,Z_t)dV_t\\
            &+\int_{\frZ_0}\eta(t, Z_{t-},\frz)\,\tilde N_0(d\frz,dt) 
            + \int_{\frZ_1}\xi(t,Z_{t-},\frz)\,\tilde N_{1}(d\frz,dt),\\
    dY_t    &=B(t,Z_t)dt + dV_t + \int_{\frZ_1} \frz\,\tilde N_{1}(d\frz, dt),
    \end{split}                                                                                             \label{system_1}
\end{equation} 
driven by a $d_1+d'$-dimensional 
$\cF_t$-Wiener process $(W_t,V_t)_{t\geq0}$ and independent 
$\F_t$-Poisson martingale measures 
$\tilde N_i(d\frz,dt) = N_i(d\frz,dt)-\nu_i(d\frz)dt$ on  
$\bR_{+}\times\frZ_i$, for $i=0$ and $1$, carried by a complete 
filtered probability space $(\Omega,\cF,(\cF_t)_{t\in[0,T]},P)$,  
where $\nu_0$ and $\nu_1$, the characteristic measures of 
the Poisson random measures $N_0$ and $N_1$, are $\sigma$-finite 
measures on a separable measurable space $(\frZ_0,\cZ_0)$ and 
on 
$\frZ_1=\bR^{d'}\setminus\{0\}$ equipped with the  
$\sigma$-algebra $\cZ_1=\cB(\bR^{d'}\setminus\{0\})$ of its  
Borel sets, respectively. 
The mappings $b=(b^i)$, $B=(B^i)$, 
$\sigma=(\sigma^{ij})$ and $\rho=(\rho^{il})$ 
are Borel functions of $(t,z)=(t,x,y)\in\bR_+\times\bR^{d+d'}$, 
with values in $\bR^d$, $\bR^{d'}$, $\bR^{d\times d_1}$ 
and $\bR^{d\times d'}$, respectively, and $\eta=(\eta^i)$ and $\xi=(\xi^i)$ are 
$\bR^d$-valued $\cB(\bR_+\times\bR^{d+d'})\otimes\cZ_0$-measurable and 
$\bR^d$-valued $\cB(\bR_+\times\bR^{d+d'})\otimes\cZ_1$-measurable functions of  
$(t,z,\frz_0)\in\bR_+\times\bR^{d+d'}\times \frZ_0$ 
and $(t,z,\frz_1)\in\bR_+\times\bR^{d+d'}\times \frZ_1$, 
respectively, where $\cB(\bU)$ denotes the Borel $\sigma$-algebra on $\bU$ 
for topological spaces $\bU$.

In \cite{GG2022} we were interested in the equations for the evolution 
of the conditional distribution $P_t(dx)=P(X_t\in dx|Y_s, s\leq t)$ of the unobserved 
component $X_t$ given the observations $(Y_s)_{s\in[0,T]}$. 
Our aim in the present paper is to show, under fairly general conditions,  
that if the conditional distribution of $X_0$ given $Y_0$ has a density $\pi_0$, 
such that it is almost surely in $L_p$ for some $p\geq2$, then 
$X_t$ for every $t$ has a conditional density $\pi_t$ given $(Y_s)_{t\in[0,t]}$, 
which belongs also to $L_p$, almost surely for all $t$.  In a subsequent paper 
we  investigate the regularity properties of the conditional density. 

The filtering problem has been the subject of intense research 
for the past decades and the literature on it is vast. 
For a brief account on the filtering problem and on the history of the presentation 
of the filtering equations for partially observed diffusion processes 
we refer to \cite{K2011}. Concerning the filtering equations associated to \eqref{system_1} 
we refer the reader to \cite{GG2022} and the references therein.

In the present paper we investigate the existence  
of the conditional density $\pi_t=dP_t/dx$ under general conditions 
on the coefficients of the system \eqref{system_1}. 
We do not assume any non-degeneracy conditions on 
$\sigma$ and $\eta$, i.e., they are allowed 
to vanish. Thus, given the observations, there may not remain any randomness 
to smooth the conditional distribution $P_t(dx)$ of $X_t$, 
i.e., if the initial conditional density $\pi_0$ does not exists, then 
the conditional density $\pi_t$ for $t>0$ may not exist either.   
Therefore assuming that the initial conditional density $\pi_0$ exists, 
we are interested in the smoothness and growth conditions   
which we should require from the coefficients in order to get that $\pi_t$ exists 
for every $t\in[0,T]$ as well.        

For partially observed diffusion processes, i.e., 
when $\xi=\eta=0$ and the observation process $Y$ does not have jumps, 
the existence and the regularity properties of the conditional density $\pi_t$ 
have been extensively studied in the literature. For important results under 
non-degeneracy conditions see, for example,  \cite{K2011} \cite{K1978},  
 \cite{K2010}, \cite{P1982}, and the references therein. 
 Without any non-degeneracy assumptions, in \cite{R1980} the existence of  
$\pi_t$ is proved if $\pi_0\in W^2_p\cap W^2_2$ for some $p\geq 2$,  
the coefficients are bounded, $\sigma\sigma^*+\rho\rho^*$ has uniformly 
bounded derivatives in $x$ up to order 3, and $b$, $B$ have uniformly 
bounded derivatives in $x$ up to second order. Under these conditions 
it is also proved  that $(\pi_t)_{t\in[0,T]}$ is a weakly continuous process with values in 
$W^2_p\cap W^2_2$, and that $\pi_t$ has higher regularity   
if $\pi_0$ and the coefficients are appropriately smoother. 
By a result in a later work, \cite{KX}, one knows that if the coefficients $\sigma$, $\rho$, 
$b$ and $B$ are bounded Lipschitz continuous in $x$, and $\pi_0\in L_2(\bR^d, dx)$, 
then $\pi_t$ exists and remains in $L_2$. The approach to obtain this result is based 
on an $L_2$-estimate for the unnormalised conditional density smoothed by Gaussian 
kernels. 
The same method is also used in \cite{B2014} to prove uniqueness of measure-valued  
solutions for the Zakai equation in the case where the signal is a diffusion process, the observation contains a jump term and the coefficients are time-independent, globally Lipschitz, except for the observation drift term, which contains a time dependence, but is bounded and globally Lipschitz.
The approach from \cite{KX} is extended in \cite{MPX2019} to partially observed jump diffusions 
when the Wiener process  
in the observation process $Y$ is independent of the Wiener process in the unobserved 
process, to prove, in particular, the existence of the conditional density in $L_2$, 
if the initial conditional density exists, belongs to $L_2$, the coefficients 
are bounded Lipschitz functions, the coefficients of the random measures in the unobservable 
process are differentiable in $x$ and satisfy a condition in terms of their Jacobian. 
In \cite{QD2015} and \cite{Q2021} the filtering equations for fairly general filtering models 
with partially observed jump diffusions are obtained and studied,  
but the existence of the conditional 
density (in $L_2$) is proved only in \cite{QD2015}, in the special case when the equation for the unobserved process is driven by a Wiener process and an $\alpha$-stable additive L\'evy process, 
$\rho=0$, the coefficients $b$ and $\sigma$ are bounded functions of $x\in\bR^d$, 
$b$ has bounded first order derivatives, $\sigma$ has bounded derivatives up to second order 
and $B=B(t,x,y)$ is a bounded Lipschitz function in $z=(x,y)$.

The main theorem, Theorem \ref{theorem 1}, of the present paper
reads as follows.  Assume that the coefficients $b$, $\sigma$, $\rho$, 
$B$, $\xi$, $\eta$ and $\rho B$ are Lipschitz continuous in $z=(x,y)\in\bR^{d+d'}$,  
$B$ is bounded, $b$, $\sigma$, $\rho$, $\xi$ and $\eta$ satisfy a linear growth condition, 
$\xi$ and $\eta$ admit uniformly equicontinuous derivatives in $x\in\bR^d$, 
$x+\xi(x)$, $x+\eta(x)$ are bijective mappings in $x\in\bR^d$, and have a Lipschitz continuous 
inverse with Lipschitz constant independent of the other variables.
 Assume, moreover, that  $\E|X_0|^r<\infty$ for 
some $r>2$. Under these conditions, if 
the initial conditional density $\pi_0$ exists for some $p\geq2$, then the conditional density 
$\pi_t$ exists and belongs to $L_p$ for every $t$. Moreover, $(\pi_t)_{t\in[0,T]}$ is weakly 
cadlag as $L_p$-valued process.

To prove our main theorem we use the It\^o formula from \cite{GW2021} 
and adapt an approach from \cite{KX} to estimate  
the $L_p$-norm of the smoothed unnormalised conditional distribution 
for even integers $p\geq2$. Hence we obtain Theorem \ref{theorem 1} 
for even integers $p\geq2$. Then we use an interpolation theorem combined 
with an approximation procedure to get the main theorem for every $p\geq2$.  
In a follow-up paper we show that if  in addition to the above assumptions 
the coefficients have bounded derivatives up to order $m+1$ and $\pi_0$ belongs to 
the Sobolev space $W^m_p$ for some $p\geq 2$ and integer $m\geq1$, then 
$\pi_t$ belongs to $W^m_p$ for $t\in[0,T]$. 

The paper is organised as follows. In section 2 we formulate our main result. In section 3 we recall important results from \cite{GG2022} together with the filtering equations obtained therein. 
In section 4 we prove $L_p$ estimates needed for a priori bounds for the smoothed conditional distribution. In section 5 we obtain an Ito formula for the $L_p$-norm of the smoothed conditional distribution and prove our result for the case $p=2$. 
Section 6 contains existence results for the filtering equation in $L_p$-spaces. 
In the last section we prove our main theorem.

We conclude with some notions and notations 
used throughout the paper. 
For an integer $n\geq0$ the notation $C^n_b(\bR^d)$ means 
the space of real-valued bounded continuous functions on $\bR^d$, 
which have bounded and continuous derivatives up to order $n$. 
(If $n=0$, then $C^0_b(\bR^d)=C_b(\bR^d)$ denotes the space of 
real-valued bounded continuous functions on $\bR^d$). 
We use the notation $C^{\infty}_{0}=C^{\infty}_{0}(\bR^d)$ for the space of 
real-valued compactly supported smooth functions on $\bR^d$.
We denote by  $\bM=\bM(\bR^d)$ the set of finite Borel measures 
on $\bR^d$ and by $\frM=\frM(\bR^d)$ the set of finite signed Borel measures on $\bR^d$. For $\mu\in\frM$ we use the notation 
$$
\mu(\varphi)=\int_{\bR^d}\varphi(x)\,\mu(dx) 
$$
for Borel functions $\varphi$ on $\bR^d$. 
We say that a function $\nu:\Omega\to\bM$ is $\cG$-measurable 
for a $\sigma$-algebra $\cG\subset\cF$, if $\nu(\varphi)$ is a  
$\cG$-measurable random variable for every bounded Borel function 
$\varphi$ on $\bR^d$. 
An $\bM$-valued stochastic process $\nu=(\nu_t)_{t\in[0,T]}$ 
is said to be weakly cadlag if almost surely 
$\nu_t(\varphi)$ is a cadlag function of $t$ for all $\varphi\in C_b(\bR^d)$. An $\frM$-valued process $(\nu_t)_{t\in [0,T]}$ is weakly cadlag, if it is the difference of two $\bM$-valued weakly cadlag processes.
For processes $U=(U_t)_{t\in [0,T]}$ we use the notation
$
\cF_t^{U}
$
for the $P$-completion of the $\sigma$-algebra generated by 
$\{U_s: s\leq t\}$. By an abuse of notation, we often write $\cF_t^U$ when referring to the filtration $(\cF^U_t)_{t\in [0,T]}$, whenever this is clear from the context.  For a  measure space 
$(\frZ,\cZ,\nu)$ and $p\geq1$ we use the notation 
$L_p(\frZ)$ for the $L_p$-space of $\bR^d$-valued 
$\cZ$-measurable mappings defined on $\frZ$.  However, if not otherwise specified, the function spaces are considered to be over $\bR^d$.
We always use without mention the summation convention, 
by which repeated integer valued indices imply a summation. For a multi-index $\alpha=(\alpha_1,\dots,\alpha_d)$ of nonnegative integers $\alpha_i, i=1,\dots,d$, 
a function $\vp$ of $x=(x_1,\dots,x_d)\in\bR^d$ and a nonnegative  integer $k$ we use the notation
$$
D^\alpha\vp(x)=D_1^{\alpha_1}D_2^{\alpha_2}\dots D_d^{\alpha_d}\vp(x),
\quad\text{as well as}
\quad 
|D^k\vp|^2=\sum_{|\gamma|=k}|D^\gamma \vp|^2,
$$ 
where $D_i=\tfrac{\partial}{\partial {x^i}}$ and $|\cdot|$ denotes an appropriate norm. 
We also use the notation $D_{ij}=D_iD_j$.
If we want to stress that the derivative is taken in a variable $x$, we write $D^\alpha_x$. If the norm $|\cdot|$ is not clear from the context, we sometimes use appropriate subscripts, 
as in $|\vp|_{L_p}$ for the $L_p(\bR^d)$-norm of $\vp$. 
For $p\geq 1$ and integers $m\geq 0$ 
the space of functions from $L_p$, 
whose generalized derivatives 
up to order $m$ are also in $L_p$, is denoted by $W^m_p$. 
The norm $|f|_{W^m_p}$ of $f$ in $W^m_p$ is defined
by  
$$
|f|_{W^m_p}^p:=\sum_{k=0}^m \int_{\bR^d}|D^kf(x)|^p\,dx<\infty.
$$
For real-valued functions $f$ and $g$ defined on $\bR^d$ the notation $(f,g)$ means the Lebesgue integral of $fg$ over $\bR^d$  when it is well-defined.
Throughout the paper we work on the finite time interval $[0,T]$, 
where $T>0$ is fixed but arbitrary, as well as 
on a given complete probability space $(\Omega,\cF,P)$ equipped 
with a filtration $(\cF_t)_{t\geq0}$ such that $\cF_0$ 
contains all the $P$-null sets. 
For $p,q\geq 1$ and integers $m\geq 1$ 
we denote by $\bW^m_p= L_p((\Omega,\cF_0,P),W^m_p(\bR^d))$ and 
$\bL_{p,q}\subset  L_p(\Omega,L_q([0,T],L_p(\bR^d)))$ 
the set of $\cF_0\otimes  \cB(\bR^d)$-measurable 
real-valued functions   $f=f(\omega,x)$  
and $\cF_t$-optional $L_p$-valued functions  $g=g_t(\omega,x)$ such that
$$
|f|_{\bW^m_p}^p:=\E |f|_{W^m_p}^p<\infty\quad\text{and}\quad|g|_{\bL_{p,q}}^p:=\E\Big(\int_0^T |g_t|_{L_p}^q dt\Big)^{p/q}<\infty
$$
respectively. If $m=0$ we set $\bL_p=\bW^0_p$.

\mysection{Formulation of the main results}
\label{sec main results}

We fix nonnegative constants $K_0$, $K_1$, $L$, $K$ 
and functions 
$\bar\xi\in L_2(\frZ_1)=L_2(\frZ_1,\cZ_1,\nu_1)$, $\bar\eta\in 
L_2(\frZ_0)=L_2(\frZ_0,\cZ_0,\nu_0)$, used throughout 
the paper, and make the following assumptions.

\begin{assumption}                                                    \label{assumption SDE}
\begin{enumerate}
\item[(i)]For $z_j=(x_j,y_j)\in\mathbb{R}^{d+d'}$ ($j=1,2$), $t\geq0$ and $\frz_i\in\frZ_i$ ($i=0,1$) ,
$$
|b(t, z_1)-b(t,z_2)| + |B(t,z_1)-B(t,z_2)|
    +|\sigma(t,z_1)-\sigma(t, z_2)| 
    + |\rho(t,z_1)-\rho(t,z_2)|\leq L|z_1-z_2|,
$$
$$
|\eta(t,z_1,\frz_0)-\eta(t,z_2,\frz_0)|\leq \bar{\eta}(\frz_0)|z_1-z_2|,
$$
$$
|\xi(t,z_1,\frz_1)-\xi(t,z_2,\frz_1)|\leq \bar{\xi}(\frz_1)|z_1-z_2|.
$$
\item[(ii)] 
For all $z=(x,y)\in\mathbb{R}^{d+d'}$, $t\geq0$ 
and $\frz_i\in \frZ_i$ for $i=0,1$ we have 
$$
|b(t,z)|
 +|\sigma(t,z)|+ |\rho(t,z)|\leq K_0+K_1|z|,
 \quad |B(t,z)|\leq K,
$$
$$
|\eta(t,z,\frz_0)|\leq \bar{\eta}(\frz_0)(K_0+K_1|z|),
\quad   
|\xi(t,z,\frz_1)|
\leq \bar{\xi}(\frz_1)( K_0+K_1|z|), 
$$
$$
\int_{\frZ_1}|\frz|^2\,\nu_1(d\frz)\leq K_0^2.
$$
\item[(iii)] The initial condition $Z_0=(X_0,Y_0)$ is 
an $\cF_0$-measurable random variable 
with values in $\bR^{d+d'}$.
\end{enumerate}
\end{assumption}
\begin{assumption}                                                           \label{assumption p}
The functions $\bar\eta\in L_2(\frZ_0)$ 
and $\bar\xi\in L_2(\frZ_1)$ are bounded in magnitude 
by constants  $K_{\eta}$ 
and $K_{\xi}$, respectively.
\end{assumption}
\begin{assumption}                                                              \label{assumption nu}
For some $r>2$ we have $\E|X_0|^r<\infty$,
and the measure $\nu_1$ satisfies
$$
K_r:=\int_{\frZ_1} |\frz|^{r}\,\nu_1(d\frz)<\infty.
$$
\end{assumption}

By a well-known theorem of It\^o one knows that 
Assumption \ref{assumption SDE} ensures the existence and 
uniqueness of a solution $(X_t,Y_t)_{t\geq0}$ to  \eqref{system_1} 
for any given $\cF_0$-measurable initial 
value $Z_0=(X_0,Y_0)$, and for every $T>0$,
\begin{equation}                                                              \label{bound_Z}
\E\sup_{t\leq T}(|X_t|^q+|Y_t|^q)\leq N(1+\E|X_0|^q+\E|Y_0|^q)
\end{equation}
holds for $q=2$ with a constant $N$ depending only on 
$T$, $K_0$, $K$, $K_1$, 
$|\bar{\xi}|_{L_2}$, $|\bar{\eta}|_{L_2}$ and $d+d'$. 
If in addition to Assumption \ref{assumption SDE} 
we assume Assumptions  \ref{assumption p} 
and \ref{assumption nu}, then it is known, 
see, e.g., \cite{DKS}, that the moment estimate 
\eqref{bound_Z} holds with $q:=r$ 
for every $T>0$, 
where now the constant $N$ depends also on 
$r$, $K_r$ $K_{\xi}$ and $K_{\eta}$.

We also need the following additional assumption.
\begin{assumption}                                                             \label{assumption estimates}
(i) The functions $f_0(t,x,y,\frz_0):=\eta(t,x,y,\frz_0)$ 
and $f_1(t,x,y,\frz_1):=\xi(t, x,y,\frz_1)$ 
are continuously differentiable in $x\in\bR^d$ 
for each $(t,y,\frz_i)\in \bR_+\times\bRY\times \frZ_i$, for $i=0$ and $i=1$, 
respectively, such that 
$$
\lim_{\varepsilon\downarrow0}
\sup_{t\in[0,T]}\sup_{\frz\in\frZ_i}\sup_{|y|\leq R}
\sup_{|x|\leq R, |x'|\leq R, |x-x'|\leq\varepsilon}
|D_xf_i(t,x,y,\frz_i)-D_xf_i(t,x',y,\frz_i)|=0 
$$
for every $R>0$. 

\noindent
(ii) 
There is a constant $\lambda>0$ such that 
for $\theta\in [0,1]$,  $(t,y,\frz_i)\in \bR_+\times\bRY\times \frZ_i$ for $i=0,1$ 
we have  
\begin{equation*}
\lambda |x_1-x_2|\leq |x_1-x_2+\theta(f_i(t,x_1,y,\frz_i)-f_i(t,x_2,y,\frz_i))|
\quad
\text{for $x_1,x_2\in\bR^d$.} 
\end{equation*}
\newline
(iii) The function $\rho B=(\rho^{ik}B^k)$ is Lipschitz in $x\in\bR^d$, 
uniformly in $(t,y)$, i.e., 
$$
|(\rho B)(t,x_1,y)-(\rho B)(t,x_2,y)|\leq L|x_1-x_2|
\quad
\text{for all $x_1,x_2\in\bR^d$ and $(t,y)\in[0,T]\times\bR^{d'}$.}  
$$
\end{assumption}

Recall that $\cF_t^Y$ denotes the completion of the $\sigma$-algebra 
generated by $(Y_s)_{s\leq t}$.
Then the main result of the paper reads as follows. 
\begin{theorem}                                                                                                                \label{theorem 1}
Let Assumptions \ref{assumption SDE}, \ref{assumption p} and  
\ref{assumption estimates} hold. 
If $K_1\neq 0$ in Assumption \ref{assumption SDE}, then let additionally 
Assumption \ref{assumption nu} hold.
Assume the conditional density 
$\pi_0=P(X_0\in dx|\cF_0^Y)/dx$ 
exists and $\E|\pi_0|_{L_p}^p<\infty$ for some $p\geq2$. 
Then almost surely the conditional density 
$\pi_t=P(X_t\in dx|\cF^Y_t)/dx$ 
exists for all $t\in[0,T]$. Moreover,   
$(\pi_t)_{t\in[0,T]}$ is an $L_p$-valued 
weakly cadlag process. 
\end{theorem}
\mysection{The filtering equations}
\label{sec filtering equations}

To describe the evolution of the conditional distribution 
$P_t(dx)=P(X_t\in dx|Y_s,s\leq t)$ 
for $t\in[0,T]$, we introduce 
the random differential operators 
$$
\cL_t=a^{ij}_t(x)D_{ij}+b^i_t(x)D_i, \quad \cM^k_t
=\rho_t^{ik}(x)D_i+B^k_t(x), \quad k=1,2,...,d', 
$$
where 
$$
a^{ij}_t(x):=\tfrac{1}{2}\sum_{k=1}^{d_1}(\sigma^{ik}_t\sigma^{jk}_t)(x)
+\tfrac{1}{2}\sum_{l=1}^{d'}(\rho_t^{il}\rho_t^{jl})(x), 
\quad\sigma_t^{ik}(x):=\sigma^{ik}(t,x,Y_t),\quad
\rho_t^{il}(x):=\rho^{il}(t,x,Y_t), 
$$
$$
b^i_t(x):=b^i(t,x,Y_t),
\quad 
B^k_t(x):=
B^k(t,x,Y_t)
$$
for $\omega\in\Omega$, $t\geq0$, $x=(x^1,...,x^d)\in\bR^d$, 
and $D_i=\partial/\partial x^i$, 
$D_{ij}=\partial^2/(\partial x^i\partial x^j)$ for $i,j=1,2...,d$. 
Moreover for every $t\geq0$ and $\frz \in \frZ_1$ 
we introduce the random operators $I_t^{\xi}$ and $J_t^{\xi}$ defined by 
\begin{equation}                                                                      \label{IJ}                                                       
I_t^{\xi}\varphi(x,\frz)=\varphi(x+\xi_t(x,\frz), \frz)-\varphi(x,\frz), 
\quad
J_t^{\xi}\phi(x, \frz)=I_t^{\xi}\phi(x, \frz)-\sum_{i=1}^d\xi_t^i(x,\frz)D_i\phi(x,\frz)
\end{equation}
for functions $\varphi=\varphi(x,\frz)$ and $\phi=\phi(x,\frz)$ of 
$x\in\bR^d$ and $\frz\in\frZ_1$, 
and furthermore the random operators 
$I_t^{\eta}$ and $J_t^{\eta}$, defined as $I_t^{\xi}$ and $J_t^{\xi}$, respectively, with 
$\eta_t(x,\frz)$ in place of $\xi_t(x,\frz)$, where 
$$
\xi_t(x,\frz_{1}):=\xi(t,x,Y_{t-},\frz_{1}),
\quad
\eta_t(x,\frz_{0}):=\eta(t,x,Y_{t-},\frz_{0})
$$
for $\omega\in\Omega$, $t\geq0$, $x\in\bR^d$ and $\frz_i\in\frZ_i$ for $i=0,1$.

Define the processes
$$
\gamma_t
=\exp\left(-\int_0^tB_s(X_s)\,dV_s-\tfrac{1}{2}\int_0^t|B_s(X_s)|^2\,ds\right), 
\quad t\in[0,T],  
$$
\begin{equation}                                                                \label{tilde Wiener}
\tilde V_t=\int_0^tB_s(X_s)\,ds+V_t, \quad t\in[0,T]. 
\end{equation}
Since by Assumption \ref{assumption SDE} (ii) 
$B$  is bounded in magnitude by a constant, it is well-known 
that $(\gamma_t)_{t\in[0,T]}$ is an $\cF_t$-martingale and hence 
by Girsanov's theorem the measure $Q$ defined by $dQ=\gamma_TdP$ 
is a probability measure equivalent to $P$, and under $Q$ the process 
$(W_t,\tilde V_t)_{t\in[0,T]}$ is a $d_1+d'$-dimensional $\cF_t$-Wiener 
process. 

The equations governing the conditional distribution and the unnormalised 
conditional distributions of $X_t$, given the observations $\{Y_s:s\in[0,t]\}$ 
are given by the following result. 
We denote by $(\cF^Y_t)_{t\in [0,T]}$ the completed filtration generated by 
$(Y_t)_{t\in [0,T]}$. 
\begin{theorem}                                                                   \label{theorem Z1}
Assume that $Z=(X_t,Y_t)_{t\in[0,T]}$ satisfies equation 
\eqref{system_1} and let Assumption  \ref{assumption SDE}(ii) 
hold. 
Assume also $\E|X_0|^2<\infty$ if $K_1\neq 0$ 
in Assumption \ref{assumption SDE}(ii). 
Then there exist measure-valued $\cF^Y_t$-adapted 
weakly cadlag processes 
$(P_t)_{t\in[0,T]}$ and $(\mu_t)_{t\in[0,T]}$ 
such that 
$$
P_t(\varphi)=\mu_t(\varphi)/\mu_t({\bf 1}),
\quad 
\text{for $\omega\in\Omega,\,\, t\in[0,T]$}, 
$$
$$
P_t(\varphi)=\E(\varphi(X_t)|\cF^Y_t),\quad \mu_t(\varphi)
=\E_{Q}(\gamma_t^{-1}\varphi(X_t)|\cF^Y_t) 
\quad\text{(a.s.) for each $t\in[0,T]$}, 
$$
for bounded Borel functions $\varphi$ on $\bR^d$, 
and for every $\varphi\in C^{2}_b(\bR^d)$ almost surely \begin{equation}
\begin{split}
\mu_t(\vp)=&  \mu_0(\vp) +  \int_0^t\mu_{s}(\cL_s\varphi)\,ds
+ \int_0^t \mu_{s}(\cM_s^k\varphi)\,d\tilde V^k_s
+ \int_0^t\int_{\frZ_0}\mu_{s}(J_s^{\eta}\varphi)\,\nu_0(d\frz)ds\\ 
&+ \int_0^t\int_{\frZ_1}\mu_{s}(J_s^{\xi}\varphi)\,\nu_1(d\frz)ds
+\int_0^t\int_{\frZ_1}\mu_{s-}(I_s^{\xi}\varphi)\,\tilde N_1(d\frz,ds), 
\end{split}
                                                                                                                                     \label{eqZ1}
\end{equation}
and
\begin{equation}
\begin{split}
P_t(\vp)=&  P_0(\vp) +  \int_0^tP_{s}(\cL_s\varphi)\,ds
+ \int_0^t \big(P_{s}(\cM_s^k\varphi)-P_{s}(\varphi)P_s(B^k_s)\big)\,d\bar V^k_s\\ 
&+ \int_0^t\int_{\frZ_0}P_{s}(J_s^{\eta}\varphi)\,\nu_0(d\frz)ds
+ \int_0^t\int_{\frZ_1}P_{s}(J_s^{\xi}\varphi)\,\nu_1(d\frz)ds\\
&+\int_0^t\int_{\frZ_1}P_{s-}(I_s^{\xi}\varphi)\,\tilde N_1(d\frz,ds)\\
\end{split}
                                                                                                                                      \label{eqZ2}
\end{equation}
for all $t\in[0,T]$, 
where $(\tilde V_t)_{t\in[0,T]}$ is given in \eqref{tilde Wiener}, 
and the process $(\bar V_t)_{t\in[0,T]}$ is defined by 
$$
d\bar V_t=d\tilde V_t-P_t(B_t)\,dt=dV_t+(B_t(X_t)-P_t(B_t))\,dt, 
\quad \bar V_0=0. 
$$
\end{theorem}
\begin{proof}
This theorem, under more general assumptions, is proved 
in \cite{GG2022}. 
\end{proof}
\begin{remark}
Clearly, equation \eqref{eqZ1} can be rewritten as 
\begin{equation}
\begin{split}
\mu_t(\vp)=&  \mu_0(\vp) +  \int_0^t\mu_{s}(\tilde\cL_s\varphi)\,ds
+ \int_0^t \mu_{s}(\cM_s^k\varphi)\,dV^k_s
+ \int_0^t\int_{\frZ_0}\mu_{s}(J_s^{\eta}\varphi)\,\nu_0(d\frz)ds\\ 
&+ \int_0^t\int_{\frZ_1}\mu_{s}(J_s^{\xi}\varphi)\,\nu_1(d\frz)ds
+\int_0^t\int_{\frZ_1}\mu_{s-}(I_s^{\xi}\varphi)\,\tilde N_1(d\frz,ds), 
\end{split}
                                                                                                                                    \label{eqZ}
\end{equation}
where $\tilde\cL_s=\cL_s+B_s(X_s)\cM_s$. Moreover,  if 
$d\mu_t/dx$ exists for $P\otimes dt$-a.e. $(\omega,t)\in\Omega\times[0,T]$,  
and $u=u_t(x)$ is 
an $\cF_t$-adapted $L_p$-valued weakly cadlag process, for $p>1$, 
such that almost surely $u_t=d\mu_t/dx$ for all $t\in [0,T]$, 
 then for each $\varphi\in C_b^2(\bR^d)$ we have that almost surely 
\begin{equation}
\begin{split}
(u_t,\vp)=&(u_0,\vp) +  \int_0^t(u_{s},\tilde\cL_s\varphi)\,ds
+ \int_0^t(u_{s},\cM_s^k\varphi)\,dV^k_s
+ \int_0^t\int_{\frZ_0}(u_{s},J_s^{\eta}\varphi)\,\nu_0(d\frz)ds\\ 
&+ \int_0^t\int_{\frZ_1}(u_{s},J_s^{\xi}\varphi)\,\nu_1(d\frz)ds
+\int_0^t\int_{\frZ_1}(u_{s-},I_s^{\xi}\varphi)\,\tilde N_1(d\frz,ds).
\end{split}
                                                                                                                                   \label{equZ}
\end{equation}
holds for all $t\in [0,T]$.
\end{remark}
\begin{remark}
\label{remark gamma}
We also recall from \cite{GG2022} that there exists a cadlag $\cF^Y_t$-adapted 
positive process $( ^o\!\gamma_t)_{t\in [0,T]}$, 
 the optional projection of $(\gamma_t)_{t\in [0,T]}$ under $P$ with respect to $(\cF^Y_t)_{t\in [0,T]}$, such that for every $\F^Y_t$-stopping time $\tau\leq T$ we have 
\begin{equation}
\label{o gamma}
\E(\gamma_\tau|\cF^Y_\tau) = {^o\!\gamma}_\tau,\quad\text{almost surely.}
\end{equation}
Since for each $t\in [0,T]$, by known properties of conditional expectations (see i.e. \cite[Thm. 6.1]{RL2018}), almost surely
$$
\mu_t({\bf1}) = \E_Q(\gamma_t^{-1}|\cF^Y_t) 
= 1/\E(\gamma_t|\cF^Y_t) = 1/{^o\!\gamma}_t,
$$
we also have that 
for each $\vp\in C_b^2$ almost surely 
$P_t(\vp) = \mu_t(\vp){^o\!\gamma}_t$ for all $t\in [0,T]$.
\end{remark}

\begin{definition}                                                            
An $\frM$-valued weakly cadlag $\cF_t$-adapted process 
$(\mu_t)_{t\in[0,T]}$ is said to be an $\frM$-solution to the equation 
\begin{align}
d\mu_t=&\tilde\cL_t^{\ast}\mu_tdt+\cM_t^{k\ast}\mu_t\,dV^k_t
+\int_{\frZ_0}J_t^{\eta\ast}\mu_t\,\nu_0(d\frz)dt                                                               \nonumber\\
&+\int_{\frZ_1}J_t^{\xi\ast}\mu_t\,\nu_1(d\frz)dt
+\int_{\frZ_1}I_t^{\xi\ast}\mu_{t-}\,\tilde N_1(d\frz,dt)                                   \label{measureZ}
\end{align}
with initial value $\mu_0$, if for each $\varphi\in C^2_b$ 
almost surely equation \eqref{eqZ} holds for all $t\in[0,T]$. 
If $(\mu_t)_{t\in[0,T]}$ is an $\frM$-solution to equation \eqref{measureZ}, such that it takes values
in $\bM$, then we call it a measure-valued solution. 
\end{definition}

\begin{definition}                                                                \label{def Lp solution}
Let $p\geq1$ and let $\psi$ be an $L_p$-valued 
$\cF_0$-measurable random variable. 
Then we say that an $L_p$-valued $\cF_t$-adapted weakly cadlag 
process $(u_t)_{t\in[0,T]}$ 
is an $L_p$-solution of the equation 
\begin{align}
du_t=&\tilde\cL_t^{\ast}u_tdt+\cM_t^{k\ast}u_t\,dV^k_t
+\int_{\frZ_0}J_t^{\eta\ast}u_t\,\nu_0(d\frz)dt                                                               \nonumber\\
&+\int_{\frZ_1}J_t^{\xi\ast}u_t\,\nu_1(d\frz)dt
+\int_{\frZ_1}I_t^{\xi\ast}u_{t-}\,\tilde N_1(d\frz,dt)                                   \label{equdZ}
\end{align} 
with initial condition $\psi$, 
if for every $\varphi\in C_0^{\infty}$ almost surely \eqref{equZ} holds for all  
$t\in[0,T]$ and $u_0=\psi$ (a.s.). 
\end{definition}

\begin{lemma}                                                                                                                 \label{lemma E sup L1}
Let Assumption \ref{assumption SDE} hold, 
and assume also $\E|X_0|^2<\infty$ if $K_1\neq 0$ 
in Assumptions \ref{assumption SDE}(ii).
Let $(\mu_t)_{t\in[0,T]}$  
be the measure-valued process from Theorem \ref{theorem Z1}. 
Then we have
\begin{equation}                                                                          \label{6.11.2021.1}
\E\sup_{t\in [0,T]}\mu_t({\bf1})\leq N,
\end{equation}
with a constant $N$ depending only on $d$, $K$ and $T$. 
\end{lemma}
{\begin{proof}
Taking ${\bf1}$ instead of 
$\vp$ in the Zakai equation \eqref{eqZ1} 
yields for $(\beta_t)_{t\in [0,T]}:=(B_{t}(X_t))_{t\in [0,T]}$,                                                                              \begin{equation}                                                                          \label{27.10.21.1}
\mu_t({\bf1}) = 1 + \int_0^t\mu_s(\beta_{s}^kB_{s}^k)\,ds 
+ \int_0^t\mu_s (B_{s}^k)\,dV_s^k.
\end{equation}
Taking here $t\wedge\tau_n$ in place of $t$, with the stopping times 
$$
\tau_n:=\inf\{t\geq 0:\mu_t({\bf1})\geq n\},\quad n\geq 1,
$$
and using $|B|\leq K$ after taking expectations on both sides, we get 
$$
\E\mu_{t\wedge \tau_n}({\bf1}) \leq 1 + dK^2\int_0^{t}\E\mu_{s\wedge\tau_n}({\bf1})\,ds.
$$
Using Gronwall's inequality and Fatou's lemma, we obtain for each $n$,
\begin{equation}                                                              \label{9.4.2022}
\sup_{t\in [0,T]}\E\mu_t({\bf1})\leq N,
\end{equation}
with a constant $N=N(d,K,T)$. 
Moreover, by Davis' and Young's inequalities and due to \eqref{9.4.2022} we have 
$$
\E\sup_{t\in[0,T]} \int_0^{t\wedge\tau_n}\mu_s(B_{s}^k)\,dV_s^k
\leq 3\E\Big( \sum_k\int_0^{T\wedge\tau_n} \mu_s(B_{s}^k)^2\,ds\Big)^{1/2}
$$
$$
\leq N\E\Big(\sup_{t\in [0,T]}\mu_{t\wedge\tau_n}({\bf1})\int_0^T\mu_s({\bf1})\,ds\Big)^{1/2}
\leq \tfrac{1}{2}\E\sup_{t\in [0,T]}\mu_{t\wedge\tau_n}({\bf1}) + N'E\int_0^T\mu_s({\bf1})\,ds<\infty, 
$$
with constants $N$ and $N'$ depending only on $d,K,T$. Using this and  \eqref{9.4.2022}, from \eqref{27.10.21.1}
we get 
$$
\E \sup_{t\in[0,T]}\mu_{t\wedge\tau_n}({\bf1}) \leq N,\quad\text{for all $n\geq1$}
$$
with a constant $N(d,K,T)$. By Fatou's lemma we then obtain \eqref{6.11.2021.1}
\end{proof}


\mysection{$L_p$-estimates}                                                                        \label{sec estimates}

Recall that $\bM=\bM(\bR^d)$ denotes the set of 
finite measures on $\cB(\bR^d)$, and $\frM:=\{\mu-\nu:\mu,\nu\in\bM\}$. 
For $\nu\in\frM$ we  
use the notation $|\nu|:=\nu^{+}+\nu^{-}$ for the 
total variation and set $\|\nu\|=|\nu|(\bR^d)$,  
where 
$\nu^{+}\in\bM$ and $\nu^{-}\in\bM$ are the positive and negative parts of $\nu$.
For $\varepsilon>0$ we use the notation $k_{\varepsilon}$ 
for the Gaussian density function on $\bR^d$ 
with mean 0 and 
covariance matrix $\varepsilon I$. For linear functionals $\Phi$, 
acting on a real vector space $V$
containing $\cS=\cS(\bR^d)$, the rapidly decreasing smooth 
functions on $\bR^d$,  the mollification 
 $\Phi^{(\varepsilon)}$ is defined by 
 $$
 \Phi^{(\varepsilon)}(x)=\Phi(\ke(x-\cdot)), \quad x\in\bR^d.
 $$
In particular, when $\Phi=\mu$ is a (signed) measure from $\cS^{\ast}$, 
the dual of $\cS$, or $\Phi=f$ is a function from $\cS^{\ast}$, 
then 
$$
\mu^{(\varepsilon)}(x)=\int_{\bR^d}k_{\varepsilon}(x-y)\,\mu(dy),
\quad
f^{(\varepsilon)}(x)=\int_{\bR^d}k_{\varepsilon}(x-y)f(y)\,dy,\quad x\in\bR^d,   
$$
and 
$$
(L^{\ast}\mu)^{(\varepsilon)}(x)
:=\int_{\bR^d}L_yk_{\varepsilon}(x-y)\mu(dy),\quad x\in\R^d,
$$
$$
(L^{\ast}f)^{(\varepsilon)}(x)
:=\int_{\bR^d}(L_yk_{\varepsilon}(x-y))f(y)\,dy,\quad x\in\R^d,
$$
when $L$ is a linear operator on $V$ such that 
the integrals are well-defined for every $x\in\bR^d$. 
The subscript $y$ in $L_y$ indicates that 
the operator $L$ acts in the $y$-variable of  
the function $\bar k_{\varepsilon}(x,y):= k_{\varepsilon}(x-y)$. 
For example, if $L$ is a differential operator of the form $a^{ij}D_{ij}+b^iD_i+c$,  
where $a^{ij}$, $b^i$ and $c$ are functions defined on $\bR^d$, then 
$$
(L^*\mu)^{(\varepsilon)}(x)=\int_{\bR^d}(a^{ij}(y)\tfrac{\partial^2}{\partial y^i\partial y^j}
+b^i(y)\tfrac{\partial}{\partial y^i}+c(y))k_{\varepsilon}(x-y)\mu(dy). 
$$
We will often use the following well-known 
properties of mollifications with $k_{\varepsilon}$: 
\begin{enumerate}
\item[(i)] $|\varphi^{(\varepsilon)}|_{L_p}
\leq |\varphi|_{L_p}$ for $\varphi\in L_p(\bR^d)$, $p\in[1,\infty)$;
\item[(ii)] $\mu^{(\varepsilon)}(\varphi)
:=\int_{\bR^d}\mu^{(\varepsilon)}(x)\varphi(x)\,dx
=\int_{\bR^d}\varphi^{(\varepsilon)}(x)\mu(dx)
=:\mu(\varphi^{(\varepsilon)})$ for 
$\mu\in \frM$ 
and $\varphi\in L_p(\bR^d)$, $p\geq1$; 
\item[(iii)] $|\mu^{(\delta)}|_{L_p}\leq |\mu^{(\varepsilon)}|_{L_p}$ 
for $0\leq\varepsilon\leq\delta$,  
$\mu\in \frM$ and $p\geq1$. 
This property follows immediately from (i) 
and the ``semigroup property" of the Gaussian kernel,  
\begin{equation}                                                                                               \label{sg}
k_{r+s}(y-z)=\int_{\bR^d}k_{r}(y-x)k_{s}(x-z)\,dx,
\quad \text{$y,z\in\bR^d$ and $r,s\in(0,\infty)$}. 
\end{equation}
\end{enumerate}
The following generalization of (iii) is also useful: 
for integers $p\geq 2$ we have 
\begin{equation}                                                                                       \label{rho}
\rho_{\varepsilon}(y):=\int_{\bR^d}\Pi_{r=1}^pk_{\varepsilon}(x-y_r)\,dx
=c_{p,\varepsilon}e^{-\sum_{1\leq r<s\leq p}|y_r-y_s|^2/(2\varepsilon p)}, 
\quad y=(y_1,....,y_p)\in\bR^{pd},
\end{equation} 
for $\varepsilon>0$, with a constant 
$c_{p,\varepsilon}=c_{p,\varepsilon}(d)=p^{-d/2}(2\pi\varepsilon)^{(1-p)d/2}$. 
This can be  
seen immediately by noticing that for $x,y_k\in\bR^d$ 
and $y=(y_k)_{k=1}^p\in\bR^{pd}$ we have 
$$
\sum_{k=1}^p(x-y_k)^2=p\Big(x-\sum_ky_k/p\Big)^2
+\tfrac{1}{p}\sum_{1\leq k<l\leq p}(y_k-y_l)^2. 
$$
Clearly, for every $r=1,2,...,p$ and $i=1,2,...,d$,  
\begin{equation}                                                                                               \label{partialrho}
\partial_{y^i_r}\rho_{\varepsilon}(y)=\tfrac{1}{\varepsilon p}
\sum_{s=1}^p(y^i_s-y^i_r)\rho_{\varepsilon}(y), 
\quad y=(y_1,...,y_p)\in\bR^d, \quad y_r=(y^1_r...,y^d_r)\in\bR^d. 
\end{equation}
It is easy to see that 
$$
\sum_{r=1}^p\partial_{y_r^j}\rho_{\varepsilon}(y)=0
\quad
\text{for $y\in\bR^{pd}$, $j=1,2,...,d$}. 
$$
We will often use this in the form 
\begin{equation}                                                                                              \label{Drho}
\partial_{y^j_r}\rho_{\varepsilon}(y)
=-\sum_{s\neq r}^p\partial_{y_s^j}\rho_{\varepsilon}(y)
\quad\text{for $r=1,...,p$ and $j=1,2,...,d$}.  
\end{equation}

In order for the left-hand side of the following $L_p$-estimates 
for $\mu\in\frM$ in this section to be well-defined, we require that
\begin{equation}                                                                            \label{second moment mu}
K_1\int_{\bR^d}|x|^2\,|\mu|(dx)<\infty,
\end{equation}
where we use the formal convention that $0\cdot\infty = 0$, 
i.e., if $K_1=0$, then 
condition \eqref{second moment mu} is satisfied.
The following lemma generalises a lemma 
from \cite{KX}.
\begin{lemma}                                                                                       \label{lemma pe1}
Let $p\geq2$ be an integer.  
Let $\sigma=(\sigma^{ik})$ and $b=(b^{i})$ be Borel functions on $\bR^d$ 
with values in $\bR^{d\times m}$ and $\bR^d$, 
respectively, 
such that for some nonnegative constants 
$K_0$, $K_1$ and $L$ we have 
\begin{equation}                                                                                               \label{c1}                                                                                         
|\sigma(x)|+|b(x)|\leq K_0+K_1|x|
\quad |\sigma(x)-\sigma(y)|\leq L|x-y|,
\quad
|b(x)-b(y)|\leq L|x-y|
\end{equation} 
for all $x,y\in\bR^d$. Set $a^{ij}=\sigma^{ik}\sigma^{jk}/2$ for $i,j=1,2,...,d$. 
Let  $\mu\in\frM$ such that it satisfies \eqref{second moment mu}.
Then we have
$$                                                                                                                                                                      
p((\mu^{(\varepsilon)})^{p-1}, ((a^{ij}D_{ij})^*\mu)^{(\varepsilon)})
+\tfrac{p(p-1)}{2}
((\mu^{(\varepsilon)})^{p-2}
((\sigma^{ik}D_i)^*\mu)^{(\varepsilon)}, ((\sigma^{jk}D_j)^*\mu)^{(\varepsilon)})
$$
\begin{equation}                                                                                                           \label{pe1}     
\leq NL^2||\mu|^{(\varepsilon)}|^p_{L_p},                                                                      
\end{equation}
\begin{equation}                                                                                                            \label{pe2}    
((\mu^{(\varepsilon)})^{p-1}, ((b^{i}D_{i})^*\mu)^{(\varepsilon)})
\leq NL^2||\mu|^{(\varepsilon)}|^p_{L_p} 
\end{equation}
with a constant $N=N(d,p)$. 
\end{lemma}

\begin{proof}

Let $A$ and $B$ denote the left-hand side of the inequalities \eqref{pe1} 
and \eqref{pe2}, respectively.  
Note first that using 
\begin{equation}
\label{wellposedness lefthand side}
\sup_{x\in\bR^d}\sum_{k=0}^2|D^k\ke(x)|<\infty,\quad\text{and}
\quad 
\int_{\bR^d}(1+|x|+K_1|x|^2)\,|\mu|(dx)<\infty,
\end{equation}
as well as the conditions on on $\sigma$ and $b$, 
it is easy to verify that $A$ and $B$ are well-defined.
Then by Fubini's theorem and 
the symmetry of the Gaussian kernel 
$$
A=\int_{\bR^{(p+1)d}}f(x,y)\,\mu_p(dy)\,dx
$$
with
$$
f(x,y):=\left(pa^{ij}(y_p)\partial_{y^i_p}\partial_{y^j_p}
+\tfrac{p(p-1)}{2}
\sigma^{ik}(y_{p-1})\sigma^{jk}(y_{p})\partial_{y^i_{p-1}}\partial_{y^j_p}\right)
\Pi_{k=1}^pk_{\varepsilon}(x-y_k), 
$$
where $x\in\bR^d$, $y=(y_k)_{k=1}^p\in\bR^{dp}$, and $y^i_k$ denotes the $i$-th coordinate of $y_k\in\bR^d$ 
for $k=1,...,p$, and $\mu_p(dy):=\mu^{\otimes p}(dy)=\mu(dy_1)...\mu(dy_p)$. 
Hence by Fubini's theorem and symmetry again
\begin{equation}                                                                                 \label{equation A}                                     
A=\int_{\bR^{pd}}
\left(pa^{ij}(y_p)\partial_{y^i_p}\partial_{y^j_p}
+\tfrac{p(p-1)}{2}
\sigma^{ik}(y_{p-1})\sigma^{jk}(y_{p})\partial_{y^i_{p-1}}\partial_{y^j_p}\right)
\rho_{\varepsilon}(y)\,\mu_{p}(dy)
\end{equation}
\begin{equation}                                                                       \label{1}
=\tfrac{1}{2}\sum_{r=1}^p\int_{\bR^{pd}}
\Big(2a^{ij}(y_r)\partial_{y^i_r}\partial_{y^j_r}
+\sum_{s\neq r}
\sigma^{ik}(y_{r})\sigma^{jk}(y_{s})\partial_{y^i_{r}}\partial_{y^j_s}\Big)
\rho_{\varepsilon}(y)\,\mu_{p}(dy),
\end{equation}
where $\rho_{\varepsilon}$ is given in \eqref{rho}. 
Using here \eqref{Drho} and symmetry of expressions in $y_k$ and $y_l$, 
we obtain 
$$
A=-\tfrac{1}{2}\sum_{r=1}^p\sum_{s\neq r}\int_{\bR^{pd}}
\left(2a^{ij}(y_r)\partial_{y^i_r}\partial_{y^j_s}
-
\sigma^{ik}(y_{r})\sigma^{jk}(y_{s})\partial_{y^i_{r}}\partial_{y^j_s}\right)
\rho_{\varepsilon}(y)\,\mu_p(dy)
$$
$$
=-\tfrac{1}{2}\sum_{r=1}^p\sum_{s\neq r}\int_{\bR^{pd}}
\left((a^{ij}(y_r)+a^{ij}(y_s))\partial_{y^i_r}\partial_{y^j_s}
-
\sigma^{ik}(y_{r})\sigma^{jk}(y_{s})\partial_{y^i_{r}}\partial_{y^j_s}\right)
\rho_{\varepsilon}(y)\,\mu_{p}(dy)
$$
$$
=-\tfrac{1}{2}\sum_{r=1}^p\sum^p_{s=1}\int_{\bR^{pd}}
a^{ij}(y_r,y_s)\partial_{y^i_{r}}\partial_{y^j_s}
\rho_{\varepsilon}(y)\,\mu_{p}(dy)
$$
$$
=-\tfrac{1}{2}\sum_{r=1}^p\sum_{s=1}^p\int_{\bR^{pd}}
a^{ij}(y_r,y_s)l^{ij,rs}_{\varepsilon}(y)
\rho_{\varepsilon}(y)\,\mu_{p}(dy), 
$$
where 
\begin{equation}                                                                                               \label{a2}
a^{ij}(x,z)=\frac{1}{2}(\sigma^{ik}(x)-\sigma^{ik}(z))(\sigma^{jk}(x)-\sigma^{jk}(z))  
\quad\text{for $x,z\in\bR^d$}
\end{equation}
 and 
$$
l^{ij,rs}_{\varepsilon}(y)=\rho^{-1}_{\varepsilon}(y)\partial_{y^i_{r}}\partial_{y^j_s}
\rho_{\varepsilon}(y)
=\tfrac{1}{(p\varepsilon)^2}\sum^p_{k=1}\sum_{l=1}^p(y^i_k-y^i_r)(y^j_l-y^j_s)
+\tfrac{\delta_{ij}}{p\varepsilon}.
$$
Making use of the Lipschitz condition on $\sigma$ and 
using for $q=1,2$ that 
\begin{equation}                                                                                                \label{qrho}
\varepsilon^{-q}\sum_{s\neq r}|y_s-y_r|^{2q}\rho_{\varepsilon}
(y)\leq N\rho_{2\varepsilon}(y), \quad y\in\bR^{pd}, \quad q\geq0
\end{equation}
with a constant $N=N(d,p,q)$, 
we have  
$$
\big|\sum_{r=1}^p\sum_{s=1}^p
a^{ij}(y_r,y_s)l^{ij,rs}_{\varepsilon}(y)\big|
\leq \tfrac{NL^2}{(p\varepsilon)^2}\sum_{1\leq r<s\leq p}|y_r-y_s|^4\rho_{\varepsilon}(y)+
\tfrac{NL^2}{p\varepsilon}\sum_{1\leq r<s\leq p}|y_r-y_s|^2\rho_{\varepsilon}(y)
$$
$$
\leq N'L^2\rho_{2\varepsilon}(y)\quad \text{for $y\in\bR^{pd}$}
$$
with constants $N=N(d,p)$ and $N'=N'(d,p)$. Hence 
$$
A\leq N'L^2\int_{\bR^{pd}}\rho_{2\varepsilon}(y)|\mu_{p}|(dy)=N'L^2
\int_{\bR^{pd}}\int_{\bR^d}\Pi_{r=1}^pk_{2\varepsilon}(x-y_r)\,dx|\mu_{p}|(dy) 
$$
$$
=N'L^2\int_{\bR^{d}}\Pi_{r=1}^p\int_{\bR^d}k_{2\varepsilon}(x-y_r)|\mu|(dy_r)\,dx
=N'L^2||\mu|^{(2\varepsilon)}|^p_{L_p}. 
$$
To prove \eqref{pe2} we proceed similarly. By Fubini's theorem and symmetry
$$
pB=\int_{\bR^{pd}}pb^i(y_p)\partial_{y^i_p}\rho_{\varepsilon}(y)\,\mu_{p}(dy)
=\sum_{r=1}^p\int_{\bR^{pd}}b^i(y_r)\partial_{y^i_r}\rho_{\varepsilon}(y)\,\mu_{p}(dy)
$$
$$
=-\sum_{r=1}^p\sum_{s\neq r}\int_{\bR^{pd}}b^i(y_r)\partial_{y^i_s}\rho_{\varepsilon}(y)\,\mu_{p}(dy)
=-\sum_{r=1}^p\sum_{s\neq r}\int_{\bR^{pd}}b^i(y_s)\partial_{y^i_r}\rho_{\varepsilon}(y)\,\mu_{p}(dy). 
$$
Thus 
$$
B=\tfrac{p-1}{p}\sum_{r=1}^p\int_{\bR^{pd}}
b^i(y_r)\partial_{y^i_r}\rho_{\varepsilon}(y)\,\mu_{p}(dy)
-\tfrac{1}{p}\sum_{r=1}^p\sum_{s\neq r}
\int_{\bR^{pd}}b^i(y_s)\partial_{y^i_r}\rho_{\varepsilon}(y)\,\mu_{p}(dy)
$$
$$
=\tfrac{1}{p}\sum_{r=1}^p
\int_{\bR^{pd}}\sum_{s\neq r}(b^i(y_r)-b^i(y_s))\partial_{y^i_r}\rho_{\varepsilon}(y)\,\mu_{p}(dy)
$$
\begin{equation}                                                                                     \label{pb}
=\tfrac{1}{\varepsilon p^2}\sum_{r=1}^p
\int_{\bR^{pd}}\sum_{s\neq r}(b^i(y_r)-b^i(y_s))\sum_{l\neq r}(y^i_l-y^i_r)
\rho_{\varepsilon}(y)\,\mu_{p}(dy). 
\end{equation}
Using the Lipschitz condition on $b$ and the inequality \eqref{qrho}, we obtain 
$$
B\leq 
\tfrac{NL}{\varepsilon}\int_{\bR^{pd}}\sum_{s\neq r}|y_r-y_s|^2\rho_{\varepsilon}(y)
|\mu_p|(dy)
\leq N'L\int_{\bR^{pd}}\rho_{2\varepsilon}(y)|\mu_p|(dy)
=N'L||\mu|^{(2\varepsilon)}|^p_{L_p}
$$
with constants $N=N(p,d)$ and $N'=N(p,d)$, which completes the proof of the lemma.
\end{proof}

\begin{lemma}                                                                                          \label{lemma pe4}
Let $p\geq 2$ be an integer and let $\sigma = (\sigma^{i})$ and $b$ be 
Borel functions on $\bR^d$ with values in $\bR^d$ and $\bR$ respectively. 
Assume furthermore that there exist constants $K\geq 1$, $K_0$ and $L$ such that
$$
|b(x)|\leq K,
\quad |\sigma(x)|\leq K_0+K_1|x|,
\quad |\sigma(x)-\sigma(y)|\leq L|x-y|
,\quad |b\sigma(x)-b\sigma(y)|\leq L|x-y|$$
for all $x,y\in\bR^d$. Let  $\mu\in\frM$ 
such that it satisfies \eqref{second moment mu}. 
Then we have 
\begin{equation}                                                                                                 \label{pe4_1}
\big((\mu^{\ep})^{p-2}(b\mu)^\ep,(b\mu)^\ep \big)
\leq K^2||\mu|^\ep|_{L_p}^p,
\end{equation}
\begin{equation}                                                                                                   \label{pe4_2}
\big((\mu^\ep)^{p-2},((\sigma^iD_i)^*\mu)^\ep(b\mu)^\ep \big)
\leq NKL||\mu|^\ep|_{L_p}^p 
\end{equation}
for every $\varepsilon>0$ with a constant $N=N(d,p)$.
\end{lemma}
\begin{proof}
We note again that by \eqref{wellposedness lefthand side} 
together with the conditions on $\sigma$ and $b$, the left-hand sides of \eqref{pe4_1} and \eqref{pe4_2} are well-defined. 
Rewriting products of integrals as multiple integrals and using  Fubini's theorem 
for the left-hand side of the inequality \eqref{pe4_1} we have  
\begin{equation*}
\int_{\bR^{dp}}b(y_r)b(y_{s})\int_{\bR^d}\Pi_{j=1}^p\ke (x-y_j)\,dx \mu_p(dy)
\end{equation*}
\begin{equation*}
\leq K^2 \int_{\bR^{d(p+1)}}\Pi_{k=1}^p\ke (x-y_k) \,dx|\mu_p|(dy) 
= K^2||\mu|^{\ep}|_{L_p}^p,
\end{equation*}
for any $r,s\in\{1,2,...,p\}$, 
where $y=(y_1,...,y_p)\in\bR^{pd}$, $y_j\in\bR^d$ for $j=1,2,...,p$, 
and the notation $\mu_p(dy)=\mu(dy_1)...\mu(dy_p)$ is used. This proves \eqref{pe4_1}. 

Rewriting products of integrals as multiple integrals, using Fubini's theorem, interchanging 
the order of taking derivatives and integrals, and using equation \eqref{rho},  
for the left-hand side $R$ of the inequality \eqref{pe4_2} we have 
$$
R = \int_{\bR^{dp}}b(y_{k})\sigma^i(y_{r})\partial_{y_{r}^i}
\int_{\bR^d}\Pi_{j=1}^p\ke (x-y_j)\,dx \mu_p(dy)
$$
\begin{equation}                                                                                                         \label{R0}
 = \int_{\bR^{dp}}
 b(y_{k})\sigma^i(y_{r})\partial_{y_{r}^i}\rho_\varepsilon(y)\, \mu_p(dy)
\end{equation}
for any $r, k\in\{1,2,..,p\}$ such that $r\neq k$. Hence 
$$
p(p-1)^2R=\sum_{s=1}^p\sum_{r\neq s}\sum_{k\neq s}
\int_{\bR^{dp}}b(y_k)\sigma^i(y_{s})\partial_{y_s^i}\rho_\varepsilon(y)\, \mu_p(dy) 
$$
$$
=\sum_{s=1}^p\sum_{r\neq s}\sum_{k\neq r}
\int_{\bR^{dp}}
b(y_k)\sigma^i(y_{s})\partial_{y_s^i}\rho_\varepsilon(y)\, \mu_p(dy) 
$$
\begin{equation}                                                                                     \label{R1}
+\sum_{s=1}^p\sum_{r\neq s}
\int_{\bR^{dp}}
(b(y_r)-b(y_s))\sigma^i(y_{s})\partial_{y_s^i}\rho_\varepsilon(y)\, \mu_p(dy), 
\end{equation}
and using \eqref{Drho} from  \eqref{R0} we obtain 
$$
p(p-1)R= -\sum_{r=1}^p\sum_{k\neq r}\sum_{s\neq r}
\int_{\bR^{dp}}b(y_k)\sigma^i(y_{r})\partial_{y_s^i}\rho_\varepsilon(y)\, \mu_p(dy) 
$$
\begin{equation}                                                                                  \label{R2}
=-\sum_{s=1}^p\sum_{r\neq s}\sum_{k\neq r}
\int_{\bR^{dp}}b(y_k)\sigma^i(y_{r})\partial_{y_s^i}\rho_\varepsilon(y)\, \mu_p(dy) 
\end{equation}
Adding up equations \eqref{R1} and \eqref{R2}, 
and taking into account the equation  
$$
(b(y_r)-b(y_s))\sigma^i(y_{s})=b(y_r)\sigma^i(y_r)-b(y_s)\sigma^i(y_{s})
-b(y_r)(\sigma^i(y_{r})-\sigma^i(y_{s}))
$$
we get 
$$
p^2(p-1)R=\sum_{s=1}^p\sum_{r\neq s}\sum^p_{k=1}
\int_{\bR^{dp}}f^i(y_k,y_s,y_r)\partial_{y_s^i}\rho_\varepsilon(y)\, \mu_p(dy) 
$$
\begin{equation}                                                                                                                   \label{R}
+\sum_{s=1}^p\sum_{r\neq s}
\int_{\bR^{dp}}g^i(y_r,y_s)\partial_{y_s^i}\rho_\varepsilon(y)\, \mu_p(dy)
\end{equation}
with functions 
\begin{equation}                                                                                                                        \label{fg}
f^i(x,u,v):=b(x)(\sigma^i(u)-\sigma^i(v)), \quad g^i(u,v)
:=b(u)\sigma^i(u)-b(v)\sigma^i(v)
\end{equation}
defined for $x,u,v\in\bR^d$ for each $i=1,2,...,d$. 
By the boundedness of $|b|$ and the Lipschitz 
condition on $\sigma$ and $b\sigma$ we have 
$$
|f^{i}(x,u,v)|\leq KL|u-v|, \quad |g^i(u,v)|\leq L|u-v|\quad \text{$x,u,v\in\bR^d$, $i=1,2,...,d$}.
$$
Thus, taking into account \eqref{partialrho} and \eqref{qrho}, from \eqref{R} we obtain 
$$
p^2(p-1)R\leq KLN\int_{\bR^{dp}}\rho_{2\varepsilon}(y)|\mu_p|(dy)
=KL N||\mu|^{(2\varepsilon)}|^p_{L_p}
\leq NKL||\mu|^{(\varepsilon)}|^p_{L_p}
$$
with a constant $N=N(d,p )$, that finishes the proof of \eqref{pe4_2}. 
\end{proof}

For $\bR^d$-valued functions  $\xi$ on $\bR^d$ we define the linear operators 
$I^{\xi}$, $J^{\xi}$ and $T^\xi$ by 
$$
T^\xi\vp(x) = \vp(x+\xi(x)),\quad I^{\xi}\varphi(x):=T^\xi\vp(x)-\varphi(x),
$$
\begin{equation}
\label{equ operators TIJ}
J^{\xi}\psi(x):=I^{\xi}\psi(x)-\xi^i(x)D_i\psi(x),\quad x\in\bR^d
\end{equation}
acting on functions $\varphi$ and differentiable functions 
$\psi$ on $\bR^d$. If $\xi$ depends also on some parameters, then 
$I^{\xi}\phi$ and $J^{\xi}\psi$ are defined for each fixed parameter as above. 

\begin{lemma}                                                                                                  \label{lemma pe3}
Let $\xi$ be an $\bR^d$-valued function of $x\in\bR^d$ such that for 
some constants $\lambda>0$, $K_0$ and $L$ 
$$
 |\xi(x)-\xi(y)|\leq L|x-y|
\quad
\text{for all $x,y\in\bR^d$}
$$
and 
\begin{equation}                                                                                                  \label{pc2}
\lambda|x-y|\leq |x-y+\theta(\xi(x)-\xi(y))|
\quad
\text{for all $x,y\in\bR^d$ and $\theta\in[0,1]$}. 
\end{equation}
Let  $\mu\in\frM$ such that it satisfies \eqref{second moment mu},  
let $p\geq2$ be an integer, and for $\varepsilon>0$ set 
$$
C:=
\int_{\bR^d}p(\mu^{(\varepsilon)})^{p-1}(J^{\xi*}\mu)^{(\varepsilon)}
+(\mu^{(\varepsilon)}+(I^{\xi*}\mu)^{(\varepsilon)})^p-(\mu^{(\varepsilon)})^p
-p(\mu^{(\varepsilon)})^{p-1}(I^{\xi*}\mu)^{(\varepsilon)}\,dx,  
$$
where, to ease notation, the argument $x\in\bR^d$ 
is suppressed in the integrand. Then 
\begin{equation}                                                                                               \label{pe3}
|C|\leq N(1+L^2)L^2||\mu|^{(\varepsilon)}|^p_{L_p}
 \quad 
 \text{for all $\varepsilon>0$}, 
\end{equation}
with a constant $N=N(d,p,\lambda)$.   
\end{lemma}

\begin{remark}                                                            \label{remark p3}
Notice that in the special case $p=2$ the 
estimate \eqref{pe3} can be rewritten as 
$$
2(\mu^{(\varepsilon)},(J^{\xi*}\mu)^{(\varepsilon)})
+((I^{\xi*}\mu)^{(\varepsilon)},(I^{\xi*}\mu)^{(\varepsilon)}) 
\leq N(1+L^2)L^2||\mu|^{(\varepsilon)}|^2_{L_2} \quad \text{for all $\varepsilon>0$}. 
$$
\end{remark}
\begin{proof}[Proof of Lemma \ref{lemma pe3}]
Again we note that by \eqref{wellposedness lefthand side}, 
together with the conditions on $\xi$ and that by Taylor's formula
$$
I^\xi\ke(x) = \int_0^1 (D_i\ke)(x-r-\theta\xi(r))\,d\theta\,\xi^i(r),
$$
$$
J^\xi\ke(x) = \int_0^1(1-\theta)(D_{ij}\ke)(x-r-\theta\xi(r))\,d\theta\,\xi^i(r)\xi^j(r),
$$
as well as that $\sup_{x\in\bR^d}\sum_{k=0}^2|D^k\rho_\varepsilon(x)|<\infty$,
it is easy to verify that $C$ is well-defined.
Notice that 
$$
\mu^{(\varepsilon)}+(I^{\xi*}\mu)^{(\varepsilon)}=(T^{\xi*}\mu)^{(\varepsilon)}
$$
and 
$$
p(\mu^{(\varepsilon)})^{p-1}(J^{\xi*}\mu)^{(\varepsilon)}
-p(\mu^{(\varepsilon)})^{p-1}(I^{\xi*}\mu)^{(\varepsilon)}
=-p(\mu^{(\varepsilon)})^{p-1}((\xi^iD_i)^*\mu)^{(\varepsilon)}. 
$$
Hence 
$$
C=\int_{\bR^d}((T^{\xi*}\mu)^{(\varepsilon)})^p-(\mu^{(\varepsilon)})^p
-p(\mu^{(\varepsilon)})^{p-1}((\xi^iD_i)^*\mu)^{(\varepsilon)}\,dx. 
$$
Rewriting here the product of integrals as multiple 
integrals and using the product measure 
$\mu_p(dy):=\mu(dy_1)...\mu(dy_p)$ by Fubini's theorem we get 
$$
((T^{\xi*}\mu)^{(\varepsilon)})^p(x)=\int_{\bR^{pd}}
\Pi_{r=1}^pT^{\xi}_{y_r}\Pi_{r=1}^pk_{\varepsilon}(x-y_r)\,\mu_p(dy), 
$$
$$
(\mu^{(\varepsilon)})^p=\int_{\bR^{pd}}\Pi_{r=1}^pk_{\varepsilon}(x-y_r)\,\mu_p(dy),
$$
\begin{align}
p(\mu^{(\varepsilon)})^{p-1}((\xi^iD_i)^*\mu)^{(\varepsilon)}
=&p\int_{\bR^{pd}}\Pi_{r=1}^{p-1}k_{\varepsilon}(x-y_r)
\xi^i(y_p)\partial_{y^i_p}k_{\varepsilon}(x-y_p)
\mu_p(dy),                                                                                                               \nonumber\\ 
=&p\int_{\bR^{pd}}\xi^i(y_p)\partial_{y^i_p}
\Pi_{r=1}^{p}k_{\varepsilon}(x-y_r)\mu_p(dy)                                                           \nonumber\\ 
=&\int_{\bR^{pd}}\sum_{r=1}^p\xi^i(y_r)\partial_{y^i_r}
\Pi_{r=1}^{p}k_{\varepsilon}(x-y_r)\,\mu_p(dy)        
\end{align}
where the last equation is due to the symmetry of the function 
$\Pi_{r}^pk_{\varepsilon}(x-y_r)$ and the measure $\mu_p(dy)$ in 
$y=(y_1,...,y_p)\in\bR^{pd}$. 
Thus 
$$
C=\int_{\bR^d}\int_{\bR^{pd}}L_y\Pi_{r=1}^pk_{\varepsilon}(x-y_r)\,\mu_p(dy)\,dx
$$
with the operator 
$$
L^\xi_y=\Pi_{r=1}^pT^{\xi}_{y_r}-\mathbb I-\sum_{r=1}^p\xi^i(y_r)\partial_{y^i_r}.  
$$
Using here Fubini's theorem then changing the order of the operator $L_y^{\xi}$ and 
the integration against $dx$ we have  
$$
C=\int_{\bR^{pd}}\int_{\bR^d}L_y^\xi\Pi_{r=1}^pk_{\varepsilon}(x-y_r)\,dx\,\mu_p(dy)
=\int_{\bR^{pd}}L_y^\xi\int_{\bR^d}\Pi_{r=1}^pk_{\varepsilon}(x-y_r)\,dx\,\mu_p(dy)
$$
\begin{equation}
=\int_{\bR^{pd}}L^{\xi}_y\rho_{\varepsilon}(y)\,\mu_p(dy),                       \label{pC}
\end{equation}
where, see \eqref{rho}, 
\begin{equation}                                                                                       \label{rho2}
\rho_{\varepsilon}(y)=
c_{p,\varepsilon}e^{-\sum_{1\leq r<s\leq p}|y_r-y_s|^2/(2\varepsilon p)}
\end{equation}
with $c_{p,\varepsilon}=p^{-d/2}(2\pi\varepsilon)^{(1-p)d/2}$.  

Introduce for $\varepsilon>0$ the function 
$$
\psi_{\varepsilon}(z)=c_{p,\varepsilon}e^{-\sum_{1\leq r<s\leq p}z^2_{rs}/(2p\varepsilon)},
\quad  
z=(z_{rs})_{1\leq r<s\leq p}\in\bR^{p(p-1)d/2}. 
$$
Then clearly, $\rho_{\varepsilon}(y)=\psi_{\varepsilon}(\tilde y)$ with 
$$
\tilde y:=(y_{rs})_{1\leq r<s\leq p}:=(y_r-y_s)_{1\leq r<s\leq p}\in \bR^{p(p-1)d/2},   
$$ 
$$
\Pi_{r=1}^pT^{\xi}_{y_r}\rho_{\varepsilon}(y)=\psi_{\varepsilon}(\tilde y+\tilde \xi(y))
\quad
\text{with $\tilde \xi(y)=(\tilde\xi_{rs}(y))_{1\leq r<s\leq p}$, 
\,\,$\tilde\xi_{rs}(y)=\xi(y_r)-\xi(y_s)$}, 
$$
and by the chain rule 
$$
\sum_{r=1}^p\xi^i(y_r)\partial_{y^i_r}\rho_{\varepsilon}
=\sum_{r=1}^p\xi^i(y_r)\sum_{1\leq k<l\leq p}(\delta_{kr}-\delta_{lr})
(\partial_{z_{kl}^i}\psi_{\varepsilon})(\tilde y)
$$
$$
=\sum_{1\leq k<l\leq p}\sum_{r=1}^p
(\delta_{kr}-\delta_{lr})\xi^i(y_r)
(\partial_{z_{kl}^i}\psi_{\varepsilon})(\tilde y)
=
\sum_{1\leq k<l\leq p}
(\xi^i(y_k)-\xi^i(y_l))
(\partial_{z_{kl}^i}\psi_{\varepsilon})(\tilde y). 
$$
Consequently,
$$
L^{\xi}_y\rho_{\varepsilon}(y)=\psi_{\varepsilon}(\tilde y+\tilde \xi(y))-\psi_{\varepsilon}(\tilde y)
-\sum_{1\leq k<l\leq p}
\tilde\xi^i_{kl}
(y)(\partial_{z_{kl}^i}\psi_{\varepsilon})(\tilde y), 
$$
which by Taylor's formula gives 
$$
L^{\xi}_y\rho_{\varepsilon}(y)
=\int_0^1(1-\theta)\sum_{1\leq k<l\leq p}\sum_{1\leq r<s\leq p}
(\partial_{z_{kl}^i}\partial_{z_{rs}^j}\psi_{\varepsilon})(\tilde y+\theta \tilde\xi(y))
\tilde\xi^i_{kl}(y)\tilde\xi^j_{rs}(y)\,d\theta, 
$$
where the summation convention is used with respect to the repeated indices $i,j=1,2,...,d$. 
Note that 
$$
(\partial_{z_{kl}^i}\partial_{z_{rs}^j}\psi_{\varepsilon})(\tilde y+\theta \tilde\xi(y))
=\psi_{\varepsilon}(\tilde y+\theta \tilde\xi(y))l^{ij,rs,kl}_{\varepsilon}(\tilde y+\theta \tilde\xi(y)) 
$$
with 
$$
l^{ij,rs,kl}_{\varepsilon}(z)
:=\tfrac{1}{(p\varepsilon)^2}z^i_{kl}z^j_{rs}-\tfrac{1}{p\varepsilon}\delta_{rk}\delta_{sl}\delta_{ij}
\quad
\text{for $z=(z_{kl})_{1\leq k<l\leq p}$}. 
$$
Due to the condition \eqref{pc2} there is a constant $\kappa=\kappa(d,\lambda)>1$ 
such that for $\theta\in[0,1]$
\begin{equation}                                                                                               \label{kt}
\kappa^{-1}|x_1-x_2|^2\leq |x_1-x_2+\theta(\xi(x_1)-\xi(x_2))|^2
\quad
\text{for $x_1,x_2\in\bR^d$},  
\end{equation}
which implies 
$$
\psi_{\varepsilon}(\tilde y+\theta \tilde\xi(y))
\leq N\psi_{\varepsilon}(\tilde y/\kappa)
=N\rho_{\kappa\varepsilon}(y), \quad\text{$y\in\bR^{dp}$},  
$$
and together with the Lipschitz condition on $\xi$,
$$
|l^{ij,rs,kl}_{\varepsilon}(\tilde y+\theta \tilde\xi(y))|
\leq 
\tfrac{N}{\varepsilon^2}(1+L^2)
\sum_{1\leq r<s\leq p}|y_r-y_s|^2+\tfrac{N}{\varepsilon}
$$
for all $1\leq r<s\leq p$, $1\leq k<l\leq p$ and $i,j=1,2,...,d$ with a constant $N=N(d,p,\lambda)$. 
Moreover,
$$
|\tilde\xi^i_{kl}(y)\tilde\xi^j_{rs}(y)|\leq NL^2\sum_{1\leq r<s\leq p}|y_s-y_p|^2
\quad\text{for $y=(y_r)_{r=1}^p\in\bR^{pd}$} 
$$
with a constant $N=N(d,p)$. Consequently, taking into account \eqref{qrho} we have 
\begin{align}
|L^{\xi}_y\rho_{\varepsilon}(y)|
\leq &\tfrac{N}{\varepsilon^2}L^2(1+L^2)
\big(\sum_{1\leq r<s\leq p}|y_r-y_s|^2\big)^2
\rho_{\kappa\varepsilon}(y)+
\tfrac{N}{\varepsilon}L^2
\sum_{1\leq r<s\leq p}|y_r-y_s|^2\rho_{\kappa\varepsilon}(y)                                    \nonumber\\
\leq&N'L^2(1+L^2)\rho_{2\kappa\varepsilon}(y)
\quad \text{for $y\in\bR^{dp}$}                                                                                    \nonumber
\end{align}
with constants $N=N(d,p,\lambda)$ and 
$N'=N'(d,p,\lambda)$. Using this we finish 
the proof by noting that \eqref{pC} implies 
$$
|C|\leq N'L^2(1+L^2)\int_{\bR^{pd}}\rho_{2\kappa\varepsilon}(y)\,|\mu_p|(dy)
=N'L^2(1+L^2)||\mu |^{(2\kappa\varepsilon)}|^p_{L_p}
\leq N{L^2}(1+L^2)||\mu|^{(\varepsilon)}|^p_{L_p}. 
$$
\end{proof}

\begin{corollary}                                                                                           \label{corollary J}
Let the conditions of Lemma \ref{lemma pe3} hold. 
Then for every even integer $p\geq2$ 
there is a constant $N=N(d,p,\lambda)$ 
such that for $\varepsilon>0$ and $\mu\in\frM$ we have
\begin{equation}                                                                                           \label{J}
\int_{\bR^d}(\mu^{(\varepsilon)})^{p-1}(x)
(J^{\xi*}\mu)^{(\varepsilon)}(x)\,dx
\leq NL^2(1+L^2)||\mu|^{(\varepsilon)}|^p_{L_p}. 
\end{equation}
\end{corollary}
\begin{proof}
Notice that 
$|a+b|^p-|b|^p-p|a|^{p-2}ab\geq0$ for $p\geq2$ for any $a,b\in\bR$ 
by the convexity of the function $f(a)=|a|^p$, $a\in\bR$.   
Using this with $a=\mu^{(\varepsilon)}$ and $b=(I^{\xi}\mu)^{(\varepsilon)}$ we have 
$$
(\mu^{(\varepsilon)}+(I^{\xi*}\mu)^{(\varepsilon)})^p-(\mu^{(\varepsilon)})^p
-p(\mu^{(\varepsilon)})^{p-1}(I^{\xi*}\mu)^{(\varepsilon)}\geq 0
\quad\text{for $x\in\bR^d$}, 
$$
which shows that \eqref{pe3} implies \eqref{J}, since 
$$
\int_{\bR^d}|\mu^{(\varepsilon)}(x)|^{p-1}|(J^{\xi*}\mu)^{(\varepsilon)}(x)|\,dx<\infty.
$$
\end{proof}

\begin{lemma}                                                                                                  \label{lemma 7.6.1}
Let the conditions of Lemma \ref{lemma pe3} hold. 
Let $p\geq2$ be an even integer.
Then 
there is a constant $N=N(d,p,\lambda)$ such that for $\varepsilon>0$ and 
$\mu\in\frM$ we have
\begin{equation}                                                                                                \label{7.6.2}
\Big|\int_{\bR^d}
(\mu^{(\varepsilon)}+(I^{\xi*}\mu)^{(\varepsilon)})^p-(\mu^{(\varepsilon)})^p
\,dx\Big|
\leq N(1+L)L||\mu|^{(\varepsilon)}|^p_{L_p}, 
\end{equation}
where the argument $x\in\bR^d$ 
is suppressed in the integrand.   
\end{lemma}

\begin{proof} By the same arguments as in the proof of Lemma \ref{lemma pe3} 
we see that the left-hand side of \eqref{7.6.2} is well-defined. Clearly, 
$$
D:=\int_{\bR^d}((T^{\xi*}\mu)^{(\varepsilon)})^p-(\mu^{(\varepsilon)})^p\,dx 
=\int_{\bR^d}\int_{\bR^{pd}}M^{\xi}_y\Pi_{r=1}^pk_{\varepsilon}(x-y_r)\,\mu_p(dy)\,dx
$$
with the operator 
$$
M^\xi_y=\Pi_{r=1}^pT^{\xi}_{y_r}-\mathbb I,  
$$
where $\mu_p(dy)=\Pi_{r=1}^p\mu(dy_r)$, $y=(y_1,...,y_p)\in\bR^{pd}$.
Hence by Fubini's theorem, then changing the order of the operator $M_y^{\xi}$ and 
the integration against $dx$ and by taking into account \eqref{rho} we have  
$$
D=\int_{\bR^{pd}}\int_{\bR^d}M_y^\xi\Pi_{r=1}^pk_{\varepsilon}(x-y_r)\,dx\,\mu_p(dy)
=\int_{\bR^{pd}}M_y^\xi\int_{\bR^d}\Pi_{r=1}^pk_{\varepsilon}(x-y_r)\,dx\,\mu_p(dy)
$$
\begin{equation}
=\int_{\bR^{pd}}M^{\xi}_y\rho_{\varepsilon}(y)\,\mu_p(dy).                       \label{7.6.4}
\end{equation}
As in the proof of Lemma \ref{lemma pe3} we introduce for $\varepsilon>0$ the function 
$$
\psi_{\varepsilon}(z)=c_{p,\varepsilon}e^{-\sum_{1\leq r<s\leq p}z^2_{rs}/(2p\varepsilon)},\quad  
z=(z_{rs})_{1\leq r<s\leq p}\in\bR^{p(p-1)d/2},  
$$
such that $\rho_{\varepsilon}(y)=\psi_{\varepsilon}(\tilde y)$ with 
$
\tilde y:=(y_{rs})_{1\leq r<s\leq p}:=(y_r-y_s)_{1\leq r<s\leq p}\in \bR^{p(p-1)d/2},     
$ 
and 
$$
\Pi_{r=1}^pT^{\xi}_{y_r}\rho_{\varepsilon}(y)
=\psi_{\varepsilon}(\tilde y+\tilde \xi(y))
\quad
\text{with $\tilde \xi(y)=(\tilde\xi_{rs}(y))_{1\leq r<s\leq p}$, \,\,
$\tilde\xi_{rs}(y)=\xi(y_r)-\xi(y_s)$}.  
$$
By Taylor's formula
$$
M^{\xi}_y\rho_{\varepsilon}(y)
=\psi_{\varepsilon}(\tilde y+\tilde \xi(y))-\psi_{\varepsilon}(\tilde y)
=\int_0^1\sum_{1\leq k<l\leq p}
(\partial_{z_{kl}^i}\psi_{\varepsilon})(\tilde y+\theta \tilde\xi(y))
\tilde\xi^i_{kl}(y)\,d\theta, 
$$
where the summation convention is used with respect 
to the repeated indices $i=1,2,...,d$. 
Note that 
$$
(\partial_{z_{kl}^i}\psi_{\varepsilon})(\tilde y+\theta \tilde\xi(y))
=\psi_{\varepsilon}
(\tilde y+\theta \tilde\xi(y))l^{kl,i}_{\varepsilon}(\tilde y+\theta \tilde\xi(y)) 
$$
with 
$$
l^{kl,i}_{\varepsilon}(z)
:=\tfrac{1}{p\varepsilon}z^i_{kl}
\quad
\text{for $z=(z_{kl})_{1\leq k<l\leq p}\in\bR^{p(p-1)d/2}$}. 
$$
By \eqref{kt} we have 
$$
\psi_{\varepsilon}(\tilde y+\theta \tilde\xi(y))
\leq N\psi_{\varepsilon}(\tilde y/\kappa)
=N\rho_{\kappa\varepsilon}(y), \quad\text{$y\in\bR^{dp}$},  
$$
and due to the Lipschitz condition on $\xi$,
$$
|l^{kl,i}_{\varepsilon}(\tilde y+\theta \tilde\xi(y))|
\leq 
\tfrac{N}{\varepsilon}(1+L)|y_k-y_l|
$$
for all $i=1,2,...,d$ with constants $\kappa(d,\lambda)>1$ 
and $N=N(d,p,\lambda)$. 
Moreover, we get 
$$
|\tilde\xi^i_{kl}(y)|\leq NL|y_k-y_l|
\quad\text{for $y=(y_r)_{r=1}^p\in\bR^{pd}$} 
$$
with a constant $N=N(d,p)$. Consequently, 
taking into account \eqref{qrho} we have 
\begin{equation*}
|M^{\xi}_y\rho_{\varepsilon}(y)|
\leq \tfrac{N}{\varepsilon}(1+L)L
\sum_{1\leq k<l\leq p}|y_k-y_l|^2
\rho_{\kappa\varepsilon}(y)\leq N' (1+L)L\rho_{2\kappa\varepsilon}(y)
\quad \text{for $y\in\bR^{dp}$}                                                                                     
\end{equation*}
with constants $N$ and 
$N'$ depending only on $d$, $p$ and $\lambda$. 
Using this, from \eqref{7.6.4} we 
obtain 
$$
|D|\leq N'(1+L)L\int_{\bR^{pd}}\rho_{2\kappa\varepsilon}(y)\,|\mu_p|(dy)
=N'(1+L)L||\mu |^{(2\kappa\varepsilon)}|^p_{L_p}
\leq N'(1+L)L||\mu|^{(\varepsilon)}|^p_{L_p}, 
$$
which completes the proof of the lemma.
\end{proof}


\mysection{The smoothed measures}

We use the notations introduced in Section \ref{sec estimates}. 
Recall that 
an $\bM$-valued stochastic process $(\nu_t)_{t\in[0,T]}$ is 
said to be weakly cadlag if almost surely 
$\nu_t(\varphi)$ is a cadlag function of $t$ for every $\varphi\in C_b(\bR^d)$. 
For such a process $(\nu_t)_{t\in[0,T]}$ there is a set 
$\Omega'\subset\Omega$ of full probability and there is uniquely defined 
(up to indistinguishability)   $\bM$-valued processes $(\nu_{t-})_{t\geq0}$ 
such that for every $\omega\in\Omega'$ 
$$
\nu_{t-}(\varphi)=\lim_{s\uparrow t}\nu_{s}(\varphi)
\quad\text{for all $\varphi\in C_b(\bR^d)$
and $t\in[0,T]$,}
$$
and for each $\omega\in\Omega'$ we have $\nu_{t-}=\nu_t$, 
for all but at most countably many $t\in(0,T]$. 
We say that an $\frM$-valued process 
$(\nu_t)_{t\in[0,T]}$ is weakly cadlag 
if it is the difference of weakly cadlag $\bM$-valued 
processes. An $\frM$-valued process 
$(\nu_t)_{t\geq0}$ is said to be adapted to 
a filtration $(\cG_t)_{t\in[0,T]}$ if $\nu_t(\varphi)$ is a 
$\cG_t$-measurable random variable for every 
$t\in[0,T]$ and bounded Borel function $\varphi$ on $\bR^d$.}

We present first a version of an It\^o formula, Theorem 2.1 from \cite{GW},  
for $L_p$-valued processes. To formulate 
it, let $\psi=(\psi(x))$, $f=(f_t(x))$, $g=(g_t^{j}(x))$ and $h=(h_t(x,\frz))$ be functions 
with values in $\bR$, $\bR$, $\bR^{m}$ and $\bR$, respectively, 
defined on $\Omega\times\bR^d$, $\Omega\times H_T$, $\Omega\times H_T$ 
and $\Omega\times H_T\times\frZ$, respectively, where $H_T:=[0,T]\times\bR^d$ 
and $(\frZ, \cZ,\nu)$ is a measure space with a $\sigma$-finite measure $\nu$ and 
countably generated $\sigma$-algebra $\cZ$. 
Assume that $\psi$ is $\cF_0\otimes\cB(\bR^d)$-measurable, $f$ 
and $g$ are $\cO\otimes\cB(\bR^d)$-measurable and $h$ 
is $\cP\otimes\cB(\bR^d)\otimes\cZ$-measurable, such that almost surely 
\begin{equation}                                                                                        \label{xintegral}
\int_0^T|f_t(x)|\,dt<\infty,\quad\int_0^T\sum_j|g^j_t(x)|^2\,dt<\infty, 
\quad \int_0^T\int_{\frZ}|h_t(x,\frz)|^2\,\nu(d\frz)\,dt<\infty 
\end{equation}
for each $x\in\bR^d$, and for each 
bounded Borel set $\Gamma\subset\bR^d$ almost surely 
$$
\int_{\Gamma}\int_0^T|f_t(x)|\,dt\,dx<\infty,\quad
\int_{\Gamma}\Big(\int_0^T\sum_j|g^j_t(x)|^2\,dt\Big)^{1/2}\,dx<\infty, 
$$
\begin{equation}                                                                                             \label{assumptionFubini}
\int_{\Gamma}
\Big(\int_0^T\int_{\frZ}\sum_j|h_t(x,\frz)|^2\,\nu(d\frz)\,dt\Big)^{1/2}\,dx<\infty.
\end{equation}
Assume, moreover, that for a number $p\in[2,\infty)$ almost surely
$$
\int_{\bR^d}|\psi(x)|^p\,dx<\infty,
\quad
 \int_0^T\int_{\bR^d}|f_t(x)|^p\,dx\,dt<\infty, 
\quad
\int_0^T\int_{\bR^d}\Big(\sum_{j}|g^j_t(x)|^2\Big)^{p/2}\,dx\,dt<\infty,
$$
\begin{equation}                                                                                             \label{pcondition}
\int_0^T\int_{\bR^d}\int_{\frZ}|h_t(x,\frz)|^p\,\nu(d\frz)\,dx\,dt<\infty, 
\,\,
\int_0^T\int_{\bR^d}\Big(\int_{\frZ}|h_t(x,\frz)|^2\,\nu(d\frz)\Big)^{p/2}\,dx\,dt <\infty.  
\end{equation}

\begin{theorem}                                                                                             \label{theorem x}
Let conditions \eqref{xintegral}, \eqref{assumptionFubini} and 
\eqref{pcondition} with a number $p\geq2$ hold. Assume there is 
an $\cO\otimes\cB(\bR^d)$-measurable 
real-valued function $v$ on $\Omega\times H_T$ such that 
almost surely 
\begin{equation}                                                                                             \label{vcondition}
\int_{\bR^d}|v_t(x)|^p\,dx<\infty\quad\text{for all $t\in[0,T]$}
\end{equation}
and for every $x\in\bR^d$ almost surely 
\begin{equation}                                                                                             \label{xdifferential}
v_t(x)=\psi(x)+\int_0^tf_s(x)\,ds+\int_0^tg^j_s(x)\,dw^j_s
+\int_0^t\int_{\frZ}h_s(x,\frz)\,\tilde\pi(d\frz,ds)
\end{equation}
holds for all $t\in[0,T]$, where 
$(w_t)_{t\geq0}$ is an $m$-dimensional  $\cF_t$-Wiener process, 
$\pi(d\frz,ds)$ is an $\cF_t$-Poisson  
measure with characteristic measure $\nu$, and  
$\tilde{\pi}(d\frz,ds)=\pi(d\frz,ds)-\nu(d\frz)ds$ is the compensated martingale measure. 
Then almost surely $(v_t)_{t\in[0,T]}$ is an $L_p$-valued 
$\cF_t$-adapted cadlag process and almost surely 
\begin{align}
|v_t|^p_{L_p}
&= |\psi|_{L_p}^p
+p\int_0^t\int_{\mathbb{R}^d}|v_s|^{p-2} v_sg^{j}_s\,dx\, dw^j_s                    \nonumber\\
& 
+\tfrac{p}{2}\int_0^t\int_{\mathbb{R}^d}\big( 2|v_s|^{p-2}v_sf_s
+(p-1)|v_s|^{p-2}\sum_{j}|g^j_s|^2\big)\,dx\,ds                                                                \nonumber\\
& 
+p\int_0^t\int_{\frZ}\int_{\mathbb{R}^d}|v_{s-}|^{p-2}v_{s-}h_s
\,dx\,\tilde{\pi}(d\frz,ds)                                                                                              \nonumber\\
&
+\int_0^t\int_{\frZ}\int_{\mathbb{R}^d}
(|v_{s-}+h_s|^p-|v_{s-}|^p-p|v_{s-}|^{p-2}v_{s-}h_s)
\,dx\,\pi(d\frz,ds)                                                                                                         \label{pIto}
\end{align}
for all $t\in[0,T]$, where   $v_{s-}$ means the left-hand limit 
in $L_p$ at $s$ of $v$. 
\end{theorem}

\begin{proof} 
By a truncation and stopping time argument it is not difficult to see that without loss 
of generality we may assume that the random variables in  \eqref{pcondition} have finite 
expectation. From \eqref{xdifferential} we get that
for each $\varphi\in C^{\infty}_0$  almost surely 
\begin{equation}                                                                                                      \label{1.15.2.22}                                                                                   
(v_t,\varphi)=(\psi,\varphi)+\int_0^t(f_s,\varphi)\,ds+\int_0^t(g^j_s,\varphi)\,dw^j_s
+\int_0^t\int_{\frZ}(h_s(\frz),\varphi)\,\tilde\pi(d\frz,ds) 
\end{equation}
holds for all $t\in[0,T]$. This we can see 
if we multiply both sides of equation \eqref{xdifferential} 
with $\varphi$ and then,  
making use of our measurability conditions 
and the conditions \eqref{xintegral} and \eqref{assumptionFubini}, 
we integrate over $\bR^d$ with respect to $dx$ 
and use deterministic and stochastic 
Fubini theorems from \cite{K-ItoWentzell} 
and \cite{GW} to change the order of integrations.  
Due to \eqref{1.15.2.22}, the measurability conditions 
on $\psi$, $f$, $g$, $h$ and $v$ and to their integrability conditions, 
\eqref{pcondition} and \eqref{vcondition}, 
by virtue of Theorem 2.1 from \cite{GW} 
there is an $L_p$-valued $\cF_t$-adapted 
cadlag process $(\bar v_t)_{t\in[0,T]}$ such that 
for each $\varphi\in C^{\infty}_0$ 
almost surely \eqref{1.15.2.22} holds with $\bar v$ in place of $v$, 
and almost surely 
\eqref{pIto} holds with $\bar v$ in place of $v$. Thus 
for each $\varphi\in C_0^{\infty}$ almost surely
$(v_t,\varphi)=(\bar v_t,\varphi)$ for all $t\in[0,T]$, 
which implies that almost surely $v=\bar v$ 
as $L_p$-valued processes, and finishes the proof of the theorem. 
\end{proof} 

\begin{lemma}                                                                                   \label{lemma ItoLp}  
Let Assumption \ref{assumption SDE}  hold. Assume                                    
$(\mu_t)_{t\in[0,T]}$ is an $\frM$-solution to equation \eqref{measureZ}
If $K_1\neq0$ in Assumption \ref{assumption SDE} (ii) then assume also 
\begin{equation}                                                                               \label{moment2}
\essup_{t\in[0,T]}\int_{\bR^d}|y|^2\,|\mu_t|(dy)<\infty\,\,(\rm{a.s.)}.
\end{equation}
Then for each $x\in\bR^d$ and $\varepsilon>0$,
\begin{equation}                                                                                         \label{eZakai}                                                                  
\begin{split}
 \mu^\ep_t(x)&= \mu^\ep_0(x) +  \int_0^t(\tilde{\cL}^*_s\mu_s)^{(\varepsilon)}(x)\,ds
+\int_0^t(\cM^{j*}_s\mu_s)^{(\varepsilon)}(x)\,dV^j_s \\
    &+ \int_0^t\int_{\frZ_0}(J^{\eta\ast}_s\mu_s)^{(\varepsilon)}(x)\,\nu_0(d\frz)\,ds
    + \int_0^t\int_{\frZ_1}(J^{\xi\ast}_s\mu_s)^{(\varepsilon)}(x)\,\nu_1(d\frz)\,ds\\
    &+ \int_0^t\int_{\frZ_1}(I^{\xi\ast}_s\mu_{s-})^{(\varepsilon)}(x)\,\Nte(d\frz,ds)
\end{split}
\end{equation}
holds almost surely for all $t\in[0,T]$. Moreover, 
for each $\varepsilon>0$ and $p\geq2$ 
$$
|\mu^\ep_t|_{L_p}^p 
= |\mu^\ep_0|_{L_p}^p 
+ p\int_0^t \big(|\mu^\ep_s|^{p-2}\mu^\ep_s,(\mathcal{M}_s^{k*}\mu_s)^\ep \big)\,dV^k_s
 + p\int_0^t \big(|\mu^\ep_s|^{p-2}\mu^\ep_s,(\tilde{\cL}^*_s\mu_s)^\ep \big)\,ds
$$
$$
+ \tfrac{p(p-1)}{2}\sum_k
\int_0^t\big(|\mu_s^\ep|^{p-2}, |(\mathcal{M}_s^{k*}\mu_s)^\ep|^2 \big)\,ds
+ p\int_0^t\int_{\frZ_0} 
\big(|\mu^\ep_s|^{p-2}\mu^\ep_s,(J^{\eta*}_s\mu_s)^\ep \big)\,\nu_0(d\frz)ds
$$
$$
+ p\int_0^t\int_{\frZ_1} 
\big(|\mu^\ep_s|^{p-2}\mu^\ep_s,(J^{\xi*}_s\mu_s)^\ep \big)\,\nu_1(d\frz)ds 
+p\int_0^t\int_{\frZ_1} 
\big(|\mu^\ep_{s-}|^{p-2}\mu^\ep_{s-},(I^{\xi*}_s\mu_{s-})^\ep \big)\,\tilde N_1(d\frz,ds)
 $$
 \begin{align}                                                                        \label{Ito formula mu epsilon p norm}
 +\int_0^t\int_{\frZ_1}\int_{\bR^d}
 \Big\{
 &\big|\mu_{s-}^\ep + (I_s^{\xi*}\mu_{s-})^\ep\big|^p 
- |\mu_{s-}^\ep|^p - p|\mu_{s-}^\ep|^{p-2}\mu^\ep_{s-}(I_s^{\xi*}\mu_{s-})^\ep
\Big\}\,dxN_1(d\frz,ds)
 \end{align}
holds almost surely for all $t\in[0,T]$.

\end{lemma}

\begin{proof}
Let $\psi\in C^{\infty}_0(\bR^d)$ such that $\psi(0)=1$, 
and for integers $r\geq1$ define $\psi_r$ by dilation,  
$\psi_r(x)=\psi(x/r)$, $x\in\bR^d$.
Then substituting 
$k_{\varepsilon}(x-\cdot)\psi_r(\cdot)\in C^{\infty}_0$ 
in place of $\varphi$ in \eqref{eqZ},  for each $x\in\bR^d$  we get 
$$
\mu_t(\ke(x-\cdot)\psi_r)=  \mu_0(\ke(x-\cdot)\psi_r) 
+  \int_0^t\mu_{s}(\tilde{\cL_s}(\ke(x-\cdot)\psi_r))\,ds
$$
$$
+ \int_0^t \mu_{s}(\cM^k_s(\ke(x-\cdot)\psi_r))\,dV^k_s
+ \int_0^t\int_{\frZ_0}\mu_{s}(J_s^{\eta}(\ke(x-\cdot)\psi_r))\,\nu_0(d\frz)ds 
$$
\begin{equation}                                                                                                  \label{erZakai}
+ \int_0^t\int_{\frZ_1}\mu_{s}(J_s^{\xi}(\ke(x-\cdot)\psi_r))\,\nu_1(d\frz)ds
+\int_0^t\int_{\frZ_1}\mu_{s-}(I_s^{\xi}(\ke(x-\cdot)\psi_r))\,\tilde N_1(d\frz,ds),
\end{equation}
almost surely for all $t\in[0,T]$. 
Clearly, 
$\lim_{r\to\infty}\ke(x-y)\psi_r(y)=\ke(x-y)$ and there is a constant $N$, 
independent of $r$, such that 
$$
|\ke(x-y)\psi_r(y)|\leq N\quad\text{for all $x,y\in\bR^d$}. 
$$ 
Hence almost surely 
\begin{equation}
\lim_{r\to\infty}\mu_t(\ke(x-\cdot)\psi_r)=\mu_t(\ke(x-\cdot))
\quad\text{for each $x\in\bR^d$ and $t\in[0,T]$}.
\end{equation}
It is easy to see that for every $\omega\in\Omega$, 
$x,y\in\bR^d$, $s\in[0,T]$ and $\frz_i\in\frZ_i$ $(i=0,1)$  
we have 
\begin{equation}                                                                        \label{limit}
\lim_{r\to\infty}A_s(\ke(x-y)\psi_r(y))=A_s(\ke(x-y))
\end{equation}
with $\tilde{\cL}$, $\cM^k$, $J^{\eta}$, $J^{\xi}$ and $I^{\xi}$ 
in place of $A$. Clearly, 
$$
\sup_{x\in\bR^d}|\psi_r(x)|=\sup_{x\in\bR^d}|\psi(x)|<\infty, 
\quad
\sup_{x\in\bR^d}|D\psi_r(x)|=r^{-1}\sup_{x\in\bR^d}|D\psi(x)|<\infty,
$$
$$ 
\sup_{x\in\bR^d}|D^2\psi_r(x)|=r^{-2}\sup_{x\in\bR^d}|D^2\psi(x)|<\infty, 
$$
and there is a constant $N$ depending only on $d$ 
and $\varepsilon$ such that for all $x,y\in\bR^d$
\begin{equation}                                                                              \label{bound0}
|k_{\varepsilon}(x-y)|
+|Dk_{\varepsilon}(x-y)|
+|D^2k_{\varepsilon}(x-y)|\leq N.  
\end{equation}
Hence, due to Assumption \ref{assumption SDE}, we have 
a constant $N=N(\varepsilon,  d,K,K_0,K_1)$ such that 
\begin{equation}                                                                                   \label{bound1}
|\tilde{\cL_s}(\ke(x-y)\psi_r(y))|
\leq N(K^2_0+K_1^2|y|^2+K^2_1|Y_s|^2),
\end{equation}
\begin{equation}                                                                                        \label{bound2}
\sum_k|\cM^k_s(\ke(x-y)\psi_r(y))|^2\leq 
N(K^2_0+K^2_1|y|^2+K^2_1|Y_s|^2) 
\end{equation}
for $x,y\in\bR^d$, $s\in[0,T]$, $r\geq1$ and $\omega\in\Omega$. 
Similarly, applying Taylor's formula to
$$
 A_s(\ke(x-y)\psi_r(y))\quad\text{with 
$J^{\eta}$, $J^{\xi}$ and $I^{\xi}$ in place of $A$}, 
$$
we have a constant $N=N(\varepsilon, d, K_0,K_1)$ such that 
\begin{align}                                                                                                 \label{bound3}
|J_s^{\eta}(\ke(x-y)\psi_r(y))|
&\leq\sup_{v\in\bR^d}|D^2_v(\ke(x-v)\psi_r(v))||\eta_{s}(y,\frz_0)|^2
\leq N|\eta_{s}(y,\frz_0)|^2, \\
 \label{bound4}
|J_s^{\xi}(\ke(x-y)\psi_r(y))|
&\leq\sup_{v\in\bR^d}|D^2_v(\ke(x-v)\psi_r(v))||\xi_{s}(y,\frz_1)|^2
\leq N|\xi_{s}(y,\frz_1)|^2
\end{align}
and 
\begin{equation}                                                                                           \label{bound5}
|I^{\xi}_{s}(\ke(x-y)\psi_r(y))|^2
\leq 
\sup_{v\in\bR^d}|D_v(\ke(x-v)\psi_r(v))|^2|\xi_{s}(y,\frz_1)|^2
\leq N|\xi_{s}(y,\frz_1)|^2, 
\end{equation}
respectively, for all $x,y\in\bR^d$, $s\in[0,T]$, 
$\frz_i\in\frZ_i$, $i=0,1$ and $\omega\in\Omega$. 
Using \eqref{limit} (with $A:=\cL$), \eqref{bound1} and  \eqref{moment2}, 
by Lebesgue's theorem on dominated convergence 
we get for each $x\in\bR^d$ 
$$
\lim_{r\to\infty}\int_0^t\mu_{s}(\tilde{\cL}_s(\ke(x-\cdot)\psi_r))\,ds
=\int_0^t\mu_{s}(\tilde{\cL}_s\ke(x-\cdot))\,ds
\quad\text{almost surely,}
$$
uniformly in $t\in [0,T]$. Using Jensen's inequality, 
\eqref{limit} (with $A:=\cM$), \eqref{bound2} and  \eqref{moment2}, 
by Lebesgue's theorem on dominated convergence 
we obtain
$$
\limsup_{r\to\infty}\int_0^T
\sum_k|\mu_{s}(\cM^k_s(\ke(x-\cdot)\psi_r))-\mu_{s}(\cM^k_s(\ke(x-\cdot))|^2\,ds
$$
$$
\leq \esssup_{s\in[0,T]}\|\mu_s\|\limsup_{r\to\infty}\int_0^T\int_{\bR^d}
\sum_k|\cM^k_s(\ke(x-\cdot)\psi_r))-\cM^k_s(\ke(x-\cdot))|^2\,|\mu_s|(dy)\,ds=0
$$
almost surely,  which implies that for $r\to\infty$ 
$$
\int_0^t\mu_{s}(\cM^k_s(\ke(x-\cdot)\psi_r))\,dV^k_s
\to \int_0^t\mu_{s}(\cM^k_s\ke(x-\cdot))\,dV^k_s
$$
in probability, for each $x\in\bR^d$, uniformly in $t\in [0,T]$. 
Since by Assumption \ref{assumption SDE}(ii)  
and \eqref{moment2} 
$$
\int_0^T\int_{\bR^d}\int_{\frZ_0}|\eta_s(y,\frz_0)|^2\nu_0(d\frz_0)\,|\mu_s|(dy)\,ds
\leq 2K^2_0|\bar\eta|_{L_2}^2\int_0^T\|\mu_s\|\,ds
$$
$$
+2K^2_1|\bar\eta|_{L_2}^2
\int_0^T\int_{\bR^d}|y|^2\,|\mu_s|(dy)\,ds
+2K^2_1|\bar\eta|_{L_2}^2
\int_0^T\int_{\bR^d}|Y_s|^2\,|\mu_s|(dy)\,ds<\infty\quad\text{(a.s.)}, 
$$
from \eqref{limit} (with $A:=J^{\eta}$) and \eqref{bound3} 
by Lebesgue's theorem on dominated convergence 
we get 
$$
\lim_{r\to\infty}\int_0^t\int_{\frZ_0}\mu_{s}(J_s^{\eta}(\ke(x-\cdot)\psi_r))\,\nu_0(d\frz)ds 
=\int_0^t\int_{\frZ_0}\mu_{s}(J_s^{\eta}\ke(x-\cdot))\,\nu_0(d\frz)ds
\quad\text{(a.s.)},  
$$
uniformly in $t\in[0,T]$. In the same way we obtain this with $J^{\xi}$, 
$\nu_1$ and $\frZ_1$ in place of $J^{\eta}$,  
$\nu_0$ and $\frZ_0$, respectively. Similarly, using first Jensen's inequality 
and Fubini's theorem we have 
$$
\limsup_{r\to\infty}\int_0^T\int_{\frZ_1}
|\mu_{s}(I_s^{\xi}(\ke(x-\cdot)\psi_r))-\mu_{s}(I_s^{\xi}(\ke(x-\cdot))|^2\,\nu_1(d\frz)\,ds
$$
$$
\leq  \esssup_{s\in[0,T]}\|\mu_s\|\limsup_{r\to\infty}\int_0^T\int_{\bR^d}\int_{\frZ_1}
|I_s^{\xi}(\ke(x-y)\psi_r(y)))-I_s^{\xi}(\ke(x-y))|^2\,\nu_1(d\frz)\,|\mu_s|(dy)\,ds=0\,\,\text{(a.s.)},  
$$
which implies that for $r\to\infty$ for each $x\in\bR^d$ we have 
$$
\int_0^t\int_{\frZ_1}
\mu_{s}(I_{s-}^{\xi}(\ke(x-\cdot)\psi_r))\,\tilde N_1(d\frz,ds)\to \int_0^t\int_{\frZ_1}
\mu_{s}(I_{s-}^{\xi}(\ke(x-\cdot)))\,\tilde N_1(d\frz,ds) 
$$
in probability, uniformly in  $t\in[0,T]$. Consequently, letting $r\to\infty$ 
in equation \eqref{erZakai}, we obtain 
that \eqref{eZakai} holds almost surely for each $t\in[0,T]$.

To prove \eqref{Ito formula mu epsilon p norm} we are going to verify that 
$$
f_t(x):=(\tilde{\cL}^*_t\mu_t)^{(\varepsilon)}(x)
+\int_{\frZ_0}(J^{\eta *}_t\mu_s)^{(\varepsilon)}(x)\,\nu_0(d\frz_0)
+\int_{\frZ_1}(J^{\xi *}_t\mu_s)^{(\varepsilon)}(x)\,\nu_1(d\frz_1), 
$$
$$
g^j_t(x):=(\cM^{j*}_t\mu_t)^{(\varepsilon)}(x), 
\quad
h_t(x,\frz):=(I^{\xi*}_t\mu_{t-})^{(\varepsilon)}(x), 
\quad
v_t(x):=\mu^{(\varepsilon)}_t(x), 
$$
($\omega\in\Omega$, $t\in[0,T]$, $x\in\bR^d$, 
$\frz\in\frZ_1$, $j=1,2,...,d'$)  
satisfy the conditions of Theorem \ref{theorem x} 
with the $\cF_t$-Wiener process $w:=V$ and 
$\cF_t$-Poisson martingale measure 
$\tilde\pi:=\tilde N_1$, carried by the probability space $(\Omega, \cF,Q)$ 
equipped with the 
filtration $(\cF_t)_{t\geq0}$. To see that $f$, $g$, $h$ 
satisfy the required measurability properties first we claim that  
for bounded 
$\cO\otimes\cB(\bR^d)\otimes\cB(\bR^{d})\otimes\cZ_0$-measurable functions 
$A=A_t(x,y,\frz)$ and bounded 
$\cP\otimes\cB(\bR^d)\otimes\cB(\bR^{d})\otimes\cZ_0$-measurable  
$A=A_t(x,y,\frz)$,  the functions 
\begin{equation}                                                                                       \label{1.20.2.22}
\int_{\bR^d}A_t(x,y,\frz)\mu_t(dy)
\quad
\text{and}\
\quad 
\int_{\bR^d}B_t(x,y,\frz)\mu_{t-}(dy)
\end{equation}
are $\cO\otimes\cB(\bR^d)\otimes\cZ_0$-  
and  $\cP\otimes\cB(\bR^d)\otimes\cZ_0$-measurable,  
in $(\omega,t,x,\frz)\in\Omega\times[0,T]\times\bR^d\times\frZ_0$, 
respectively.  Indeed, this is 
obvious if $A_t(x,y,\frz)=\alpha_t\varphi(x)\phi(y)\kappa(\frz)$ 
and 
$B_t(x,y,\frz)=\beta_t\varphi(x)\phi(y)\kappa(\frz)$
with $\varphi,\phi\in C_b(\bR^d)$, bounded $\cZ_0$-measurable function 
$\kappa$ on $\frZ_0$,  
and bounded $\cO$-measurable function $\alpha$ and 
bounded $\cP$-measurable $\beta$ 
on $\Omega\times[0,T]$. Thus our claim follows 
by a standard application of the monotone class lemma for functions. 
Hence one can easily see 
that our claim remains valid if we replace the boundedness condition 
with the existence of the integrals 
in \eqref{1.20.2.22}. Using this and taking into account 
\eqref{bound1} and \eqref{bound2}
and the estimates obtained by Taylor's formula, 
\begin{equation}                                                                                                 \label{estimate3}
|J_s^{\eta}\ke(x-y)|\leq N|\eta_{s}(y,\frz_0)|^2, 
\quad                                                                        
|J_s^{\xi}\ke(x-y)|\leq N|\xi_{s}(y,\frz_1)|^2, 
\end{equation}
\begin{equation}                                                                                           \label{estimate5}
|I^{\xi}_{s}\ke(x-y)|^2\leq N|\xi_{s}(y,\frz_1)|^2 
\end{equation}
for $x,y\in\bR^d$, $s\in[0,T]$, 
$\frz_i\in\frZ_i$ and $\omega\in\Omega$, 
where $N=N(\varepsilon,  d)$,  
it is not difficult to show  that  
$(\tilde{\cL}^*_t\mu_t)^{(\varepsilon)}(x)$, 
$(\cM^{j\ast}\mu_t)^{(\varepsilon)}(x)$ 
are $\cO\otimes\cB(\bR^d)$-measurable in $(\omega,t)$, 
$(J_t^{\eta\ast}\mu_t)^{(\varepsilon)}(x)$ and 
$(J_t^{\xi\ast}\mu_t)^{(\varepsilon)}(x)$ 
are $\cP\otimes\cB(\bR^d)\otimes\cZ_0$- and 
$\cP\otimes\cB(\bR^d)\otimes\cZ_1$-measurable 
in $(\omega,t,x,\frz_0)$ and 
$(\omega,t,x,\frz_1)$, respectively, and 
$(I_t^{\xi\ast}\mu_{t-})^{(\varepsilon)}(x)$ is  
$\cP\otimes\cB(\bR^d)\otimes\cZ_1$-measurable  
in $(\omega,t,x,\frz_1)$. Finally, integrating 
$(J_t^{\eta\ast}\mu_t)^{(\varepsilon)}(x)$ 
and $(J_t^{\xi\ast}\mu_t)^{(\varepsilon)}(x)$ 
over $\frZ_0$ and $\frZ_1$, respectively, 
by Fubini's theorem we get that $f$ is 
$\cO\otimes\cB(\bR^d)$-measurable. 
Using the estimates \eqref{bound1}, \eqref{bound2} together with 
\eqref{estimate3} and \eqref{estimate5} 
it is easy to see that 
due to $\esssup_{t\in[0,T]}|\mu_t|(\bR^d)<\infty$ (a.s.) 
and \eqref{moment2} the conditions 
\eqref{xintegral}, \eqref{assumptionFubini} hold. 
By Minkowski's inequality for every $x\in\bR^d$, $t\in[0,T]$ 
and $\omega\in\Omega$ we have 
$$
|\mu_t^{(\varepsilon)}|^p_{L_p}
=\int_{\bR^d}|\int_{\bR^d}k_{\varepsilon}(x-y)\mu_t(dy)|^p\,dx
\leq |k_{\varepsilon}|^p_{L_p}|\mu_t|^p(\bR^d)<\infty, 
$$
which shows that condition \eqref{vcondition} holds. 
To complete the proof of the lemma 
it remains to show that almost surely 
$$
A:=\int_0^T\int_{\bR^d}|(\tilde{\cL}^*_s\mu_s)^{(\varepsilon)}(x)|^p\,dxds<\infty, 
\quad
B:=\int_0^T\int_{\bR^d}
\big(\sum_k|(\cM^{k*}_s\mu_s)^{(\varepsilon)}(x)|^2\big)^{p/2}\,dxds<\infty, 
$$
$$
C_{\eta}:=\int_0^T
\int_{\bR^d}\Big|\int_{\frZ_0}
(J^{\eta *}_s\mu_s)^{(\varepsilon)}(x,\frz)\,\nu_0(d\frz)\Big|^pdxds<\infty, 
$$
$$
C_{\xi}:=\int_0^T
\int_{\bR^d}\big|\int_{\frZ_1}
(J^{\xi *}_s\mu_s)^{(\varepsilon)}(x,\frz)\,\nu_1(d\frz)\big|^pdxds<\infty, 
$$
$$
G:=\int_0^T\int_{\bR^d}
\int_{\frZ_1}|(I^{\xi *}_s\mu_s)^{(\varepsilon)}(x,\frz)|^p\,\nu_1(d\frz)dxds<\infty, 
$$
$$
H:=\int_0^T
\int_{\bR^d}
\Big(\int_{\frZ_1}
|(I^{\xi *}_s\mu_s)^{(\varepsilon)}(x,\frz)|^2\,\nu_1(d\frz)\Big)^{p/2}dxds<\infty.  
$$
To this end note first that with a constant $N=N(\varepsilon,d)$ 
\begin{equation}                                                                                     \label{e}
|k_{\varepsilon}(x-y)|
+|Dk_{\varepsilon}(x-y)|
+|D^2k_{\varepsilon}(x-y)|\leq Nk_{2\varepsilon}(x-y)\quad\text{for all $x,y\in\bR^d$}.  
\end{equation}
Thus, using Minkowski's inequality and Assumption \ref{assumption SDE}(ii), 
we have a constant 
$N$, depending on $\varepsilon$, $d$, $K_0$, $K$ and $K_1$, 
such that almost surely
$$
A\leq \int_0^T\Big(\int_{\bR^d}
\Big(\int_{\bR^d}|\tilde{\cL}_s k_{\varepsilon}(x-y)|^pdx\Big)^{1/p}|\mu_s|(dy)\Big)^pds
$$
$$
\leq N |k_{2\varepsilon}|^p_{L_p}
\int_0^T\Big(\int_{\bR^d}(K^2_0+K_1^2|y|^2+
K_1^2|Y_s|^2)|\mu_s|(dy)\Big)^p ds. 
$$
Hence taking into account $\esssup_{s\in[0,T]}\|\mu_s\|<\infty$ (a.s.),
\eqref{moment2} (if $K_1\neq0$), as well as the cadlagness of $(Y_t)_{t\in [0,T]}$, 
we get $A<\infty$ (a.s.). In the same way we have $B<\infty$ (a.s.). 
By Taylor's formula and \eqref{e} 
for each $x\in\bR^d$ we have 
$$
|J^{\eta}_yk_{\varepsilon}(x-y)|
\leq \int_0^1|D^2_yk_{\varepsilon}|(x-y-\theta\eta(y,\frz))|\eta(y,\frz)|^2\,d\theta, 
$$
$$
\leq N\int_0^1k_{2\varepsilon}(x-y-\theta\eta(y,\frz))\,d\theta\,|\eta(y,\frz)|^2,  
$$
for all $y\in\bR^d$, $s\in[0,T]$, $\frz\in\frZ_0$ and $\omega\in\Omega$.  
Here, and often later on, the variable $s$ is suppressed, and the subscript 
$y$ in $J^{\eta}_y$ indicates that the operator 
$J^{\eta}$ acts in the variable $y$. 
Hence Minkowski's inequality gives 
$$
\Big(\int_{\bR^d}|J^{\eta}_yk_{\varepsilon}(x-y)|^p\,dx\Big)^{1/p}
\leq N|k_{2\varepsilon}|_{L_p}|\eta(y,\frz)|^2
$$
with a constant $N=N(d,\varepsilon)$. 
Thus by the Minkowski inequality and Fubini's theorem,  
$$
C_{\eta}\leq \int_0^T
\left(\int_{\frZ_0}
\int_{\bR^d}
\left(
\int_{\bR^d}|J^{\eta}_yk_{\varepsilon}(x-y)|^p\,dx
\right)^{1/p}\,|\mu_s|(dy)
\,\nu_0(d\frz)
\right)^pds
$$
$$
\leq N^p|k_{2\varepsilon}|^p_{L_p}
\int_0^T
\Big(
\int_{\frZ_0}
\int_{\bR^d}|\eta_s(y,\frz)|^{2}\,|\mu_s|(dy)
\,\nu_0(d\frz)
\Big)^pds
$$
$$
\leq 2^pN^p|\bar\eta|^{2p}_{L_2}|k_{2\varepsilon}|^p_{L_p}
\int_0^T
\Big(
\int_{\bR^d}(K_0^2+K_1^2|y|^{2}+K_1^2|Y_s|^{2})\,|\mu_s|(dy)
\Big)^p\,ds<\infty \,\,\text{(a.s.)}.  
$$
In the same way we get $C_{\xi}<\infty$ (a.s.). 
By Taylor's formula and \eqref{e} 
for each $x\in\bR^d$ we have 
$$
|I^{\xi}_yk_{\varepsilon}(x-y)|
\leq 
\int_0^1
|D_yk_{\varepsilon}|(x-y-\theta\xi(y,\frz))|\xi(y,\frz)|\,d\theta, 
$$
\begin{equation}                                                                                          \label{1.6.2.22}
\leq N\int_0^1k_{2\varepsilon}(x-y-\theta\xi(y,\frz))\,d\theta\,|\xi(y,\frz)|,  
\end{equation}
for all $y\in\bR^d$, $s\in[0,T]$, $\frz\in\frZ_0$ and $\omega\in\Omega$,  
with a constant $N=N(d,p,\varepsilon)$. 
Hence similarly to above we 
obtain 
$$
G\leq NK_{\xi}^{p-2}|\bar\xi|^2_{L_2}|k_{2\varepsilon}|^p_{L_p}
\int_0^T\Big(
\int_{\bR^d}(K_0+K_1|y|+K_1|Y_s|)\,|\mu_s|(dy)
\Big)^p\,ds<\infty \,\,\text{(a.s.)}.   
$$
with a constant $N=N(d,p,\varepsilon)$. 
By Minkowski's inequality, 
taking into account \eqref{1.6.2.22} 
and Assumption \ref{assumption p} we have 
$$
H\leq \int_0^T
\left(
\int_{\frZ_1}\Big(\int_{\bR^d}|(I_t^{\xi*}\mu_{t})^{\ep}|^p\,dx\Big)^{2/p}\,\nu_1(d\frz)
\right)^{p/2}\,dt
$$
$$
\leq\int_0^T
\left(
\int_{\frZ_1}\Big(\int_{\bR^d}
\Big(
\int_{\bR^d}|I_t^\xi k_{\varepsilon}(x-y)|^p\,dx\Big)^{1/p}|\mu_t|(dy)
\Big)^{2}\,\nu_1(d\frz)\right)^{p/2}\,dt
$$
$$
\leq N |\bar\xi|^p_{L_2}|k_{2\varepsilon}|^p_{L_p}
\int_0^T
\Big(
\int_{\bR^d}(K_0+K_1|y|+K_1|Y_t|)\,|\mu_t|(dy)
\Big)^{p}\,dt<\infty
$$
almost surely, with a constant $N=N(d,p,\varepsilon)$. 
\end{proof}
\begin{lemma}                                                                                   \label{lemma Ito_u_Lp}  
Let Assumption \ref{assumption SDE} hold. Assume                                    
$(u_t)_{t\in[0,T]}$ is an $L_p$-solution of equation \eqref{equdZ} 
for a given $p\geq2$ 
such that $\esssup_{t\in [0,T]}|u_t|_{L_1}<\infty$ (a.s.), and if  
$K_1\neq0$ in Assumption \ref{assumption SDE} (ii), then 
\begin{equation}                                                                               \label{umoment2}
\essup_{t\in[0,T]}\int_{\bR^d}|y|^2\,|u_t|(dy)<\infty\,\,(\rm{a.s.)}.
\end{equation}
Then for each $x\in\bR^d$ and $\varepsilon>0$,
\begin{equation}                                                                                         \label{euZakai}                                                                  
\begin{split}
 u^\ep_t(x)&= u^\ep_0(x) +  \int_0^t(\tilde{\cL}^*_su_s)^{(\varepsilon)}(x)\,ds
+\int_0^t(\cM^{j\ast}_su_s)^{(\varepsilon)}(x)\,dV^j_s \\
    &+ \int_0^t\int_{\frZ_0}(J^{\eta\ast}_su_s)^{(\varepsilon)}(x)\,\nu_0(d\frz)\,ds
    + \int_0^t\int_{\frZ_1}(J^{\xi\ast}_su_s)^{(\varepsilon)}(x)\,\nu_1(d\frz)\,ds\\
    &+ \int_0^t\int_{\frZ_1}(I^{\eta *}_su_{s-})^{(\varepsilon)}(x)\,\Nte(d\frz,ds)
\end{split}
\end{equation}
holds almost surely for all $t\in[0,T]$. Moreover, 
for each $\varepsilon>0$ and $p\geq2$ 
$$
|u^\ep_t|_{L_p}^p 
= |u^\ep_0|_{L_p}^p 
+ p\int_0^t \big(|u^\ep_s|^{p-2}u^\ep_s,(\mathcal{M}_s^{k\ast}u_s)^\ep \big)\,dV^k_s
 + p\int_0^t \big(|u^\ep_s|^{p-2}u^\ep_s,(\tilde{\cL}^*_s u_s)^\ep \big)\,ds
$$
$$
+ \tfrac{p(p-1)}{2}\sum_k\int_0^t\big(|u_s^\ep|^{p-2}, |(\mathcal{M}_s^{k\ast}u_s)^\ep|^2 \big)\,ds
+ p\int_0^t\int_{\frZ_0} \big(|u^\ep_s|^{p-2}u^\ep_s,(J^{\eta\ast}_su_s)^\ep \big)\,\nu_0(d\frz)ds
$$
$$
+ p\int_0^t\int_{\frZ_1} 
\big(|u^\ep_s|^{p-2}u^\ep_s,(J^{\xi\ast}_su_s)^\ep 
\big)\,\nu_1(d\frz)ds 
+p\int_0^t\int_{\frZ_1} 
\big(
|u^\ep_{s-}|^{p-2}u^\ep_{s-},(I^{\xi\ast}u_{s-})^\ep
 \big)\,\tilde N_1(d\frz,ds)
 $$
 \begin{align}                                                                        \label{Ito formula u epsilon p norm}
 +\int_0^t\int_{\frZ_1}\int_{\bR^d}
 \Big\{&\big|u_{s-}^\ep + (I_s^{\xi\ast}u_{s-})^\ep\big|^p 
- |u_{s-}^\ep|^p - p|u_{s-}^\ep|^{p-2}u^\ep_{s-}(I_s^{\xi\ast}u_{s-})^\ep
\Big\}\,dxN_1(d\frz,ds)
 \end{align}
holds almost surely for all $t\in[0,T]$, where $u_{s-}$ denotes the left limit in $L_p$. 
\end{lemma}
\begin{proof}
Notice that equations \eqref{euZakai} and \eqref{Ito formula u epsilon p norm} 
can be formally obtained from equations \eqref{eZakai} 
and \eqref{Ito formula mu epsilon p norm}, respectively, by substituting 
$u_t(x)dx$ and $u_{t-}(x)dx$ in place of $\mu_t(dx)$ and $\mu_{t-}(dx)$, 
respectively. Note, however, that $u_t(x)dx$, defines a signed measure only 
for $P\otimes dt$-almost every $(\omega,t)\in\Omega\times[0,T]$. Thus this lemma 
does not follow directly from Lemma \ref{lemma ItoLp}. We can copy, however, 
the proof of Lemma \ref{lemma ItoLp} by replacing $\mu_t(dx)$ and $\mu_{t-}(dx)$ with 
$u_t(x)dx$ and $u_{t-}(x)dx$, respectively. We need also take into account that 
since $(u_{t})_{t\in[0,T]}$ is an $L_p$-valued weakly cadlag process, 
we have have a set $\Omega'$ 
of full probability such that $u_{t-}(\omega)=u_t(\omega)$ for all but countably many $t\in[0,T]$, 
and $\sup_{t\in[0,T]}|u_t(\omega)|_{L_p}<\infty$ for $\omega\in\Omega'$.  
\end{proof}

\begin{lemma}                                                             \label{lemma supemu}
Let Assumptions \ref{assumption SDE},  \ref{assumption p}
and \ref{assumption estimates}  hold. 
Let $(\mu _t)_{t\in[0,T]}$ be a measure-valued solution to 
\eqref{measureZ}. If $K_1\neq 0$ in Assumption \ref{assumption SDE}, 
then assume additionally \eqref{moment2}. 
Then for $\varepsilon>0$ and even integers $p\geq2$ we have   
\begin{equation}                                                                                                     \label{esupbecsles}
\E\sup_{t\in[0,T]}|\mu^{(\varepsilon)}_t|^p_{L_p}
\leq N\E|\mu^{(\varepsilon)}_0|^p_{L_p}
\end{equation}
with a constant 
$N=N(p,d,T,K, K_{\xi}, K_{\eta},L,\lambda, |\bar\xi|_{L_2}, |\bar\eta|_{L_2})$.
\end{lemma}

\begin{proof}
We may assume that $\E |\mu_0^{(\varepsilon)}|_{L_p}^p<\infty$. 
Define  
$$
Q_p(b, \sigma, \rho, \beta,\mu, k_{\varepsilon})
=p\big(|\mu^\ep|^{p-2}\mu^\ep,(\tilde\cL^*\mu)^\ep \big)                                                
+\tfrac{p(p-1)}{2}\sum_k\big(|\mu^\ep|^{p-2}, |(\mathcal{M}^{k*}\mu)^\ep|^2 \big),  
$$
$$    
\cQ_p^{(0)}(\eta(\frz_0),\mu, k_{\varepsilon})
=p\big(|\mu^{\ep}|^{p-2}\mu^{\ep},(J^{\eta(\frz_0)*}\mu)^{\ep}\big),                          
$$
\begin{equation}                                                                         \label{Q}
\cQ_p^{(1)}(\xi(\frz_1), \mu, k_{\varepsilon})
=
p(|\mu^{\ep}|^{p-2}\mu^{\ep}, (J^{\xi(\frz_1)*}\mu)^{\ep}),
\end{equation}
$$                                                    
\cR_p(\xi(\frz_1),\mu,k_{\varepsilon})=|\mu^{\ep} + (I^{\xi(\frz_1)*}\mu)^{\ep}|^p_{L_p} 
- |\mu^{\ep}|^p_{L_p} - p(|\mu^{\ep}|^{p-2}\mu^{\ep},(I^{\xi(\frz_1)*}\mu)^{\ep}),                   
$$
for $\mu\in\bM$, $\beta\in\bR^{d'}$, functions $b$, $\sigma$ and $\rho$ on $\bR^d$, 
with values in $\bR^d$, $\bR^{d\times d_1}$ and $\bR^{d\times d'}$, respectively, 
and $\bR^d$-valued functions $\eta(\frz_0)$ and $\xi(\frz_1)$ on $\bR^d$ for each $\frz_i\in\frZ_i$, $i=0,1$, 
where 
$$ 
\tilde\cL=\tfrac{1}{2}(\sigma^{il}\sigma^{jl}+\rho^{ik}\rho^{jk})D_{ij}+
\beta^l\rho^{il}D_i+\beta^lB^l,   
\quad
\cM^k=\rho^{ik}D_i+B^k, \quad k=1,2,...,d'. 
$$
Recall that $(f,g)$ denotes the integral of the product of Lebesgue measurable functions 
$f$ and $g$ over $\bR^d$ against the Lebesgue measure on $\bR^d$. 
By Lemma \ref{lemma ItoLp}
$$
d|\mu_t^{(\varepsilon)}|^p_{L_p}               
=\cQ_p(b_t, \sigma_t, \rho_t, \beta_t,\mu_t, k_{\varepsilon})\,dt                                                                  
 +\int_{\frZ_0}\cQ_p^{(0)}(\eta_t(\frz),\mu_t, k_{\varepsilon}) \,\nu_0(d\frz)\,dt 
 $$
 \begin{equation}                                                                 \label{7.23.1}                                                                                       
 +\int_{\frZ_1}\cQ_p^{(1)}(\xi_t(\frz), \mu_t, k_{\varepsilon})\,\nu_1(d\frz)\,dt
+\int_{\frZ_1}\cR_p(\xi_{t}(\frz), \mu_{t-}, k_{\varepsilon})\,N_1(d\frz,dt)
+d\zeta_1(t)+d\zeta_2(t),                        
\end{equation}
where $\beta_t=B_t(X_t)$  and 
\begin{equation}                                                            \label{30.9.21.7}
\zeta_1(t)=p\int_0^t \big(|\mu^\ep_s|^{p-2}
\mu^\ep_s,(\mathcal{M}_s^{k *}\mu_s)^\ep \big)\,dV^k_s, 
\end{equation}
$$
\zeta_2(t)=p\int_0^t\int_{\frZ_1} 
\big(|\mu^\ep_{s}|^{p-2}\mu^\ep_{s},(I_s^{\xi *}\mu_{s})^\ep \big)\,\Nte(d\frz,ds)\quad t\in[0,T]   
$$
are local martingales under $P$. 
We write 
\begin{equation}                                                                 \label{7.23.2}
\int_{\frZ_1}\cR_p(\xi_{t}(\frz_1), \mu_{t-}, k_{\varepsilon})\,N_1(d\frz,dt)
=\int_{\frZ_1}\cR_p(\xi_{t}(\frz_1), \mu_{t-}, k_{\varepsilon})\,\nu_1(d\frz)dt+d\zeta_3(t)
\end{equation}
with 
$$
\zeta_3(t)=\int_0^t\int_{\frZ_1}\cR_p(\xi_s(\frz), \mu_{s-}, k_{\varepsilon})\,N_1(d\frz,ds)
-\int_0^t\int_{\frZ_1}\cR_p(\xi_s(\frz), \mu_{s-}, k_{\varepsilon})\,\nu_1(d\frz)ds, 
$$
which we can justify if we show 
\begin{equation}
\label{6.10.21.3}
A:=\int_0^T
\int_{\frZ_1}
|\cR_p(\xi_s(\frz), \mu_s, k_{\varepsilon})|\,
\nu_1(d\frz)\,ds<\infty \,\text{(a.s.)}.  
\end{equation}
To this end observe that by 
Taylor's formula 
\begin{equation}                                                                                    \label{7.14.4}                                                                                                  
0\leq\cR_p(\xi_t(\frz), \mu_t, k_{\varepsilon})
\leq N\int_{\bR^d}
|\mu^{(\varepsilon)}_t|^{p-2}|
(I^{\xi(\frz)*}_t\mu_t)^{(\varepsilon)}|^{2}
+|(I^{\xi(\frz)*}_t\mu_t)^{(\varepsilon)}|^{p}\,dx 
\end{equation}
with a constant $N=N(d,p)$. Hence 
$$
\int_{\frZ_1}\cR_p(\xi_t(\frz), \mu_t, k_{\varepsilon})\,\nu_1(d\frz)
\leq N\int_{\bR^d}
|\mu^{(\varepsilon)}_t|^{p-2}|
(I^{\xi(\frz)*}_t\mu_t)^{(\varepsilon)}|_{L_2(\frZ_1)}^{2}
+|(I^{\xi(\frz)*}_t\mu_t)^{(\varepsilon)}|^{p}_{L_p(\frZ_1)}\,dx
$$
$$
\leq N'\big(
|\mu_t^{(\varepsilon)}|^p_{L_p}+A_1(t)+A_2(t)\big)
$$
with
\begin{equation}
\label{6.10.21.2}
A_1(t)=\int_{\bR^d}|(I^{\xi(\frz) *}_t\mu_t)^{(\varepsilon)}|^{p}_{L_2(\frZ_1)}\,dx,
\quad
A_2(t)=
\int_{\bR^d}|(I^{\xi(\frz) *}_t\mu_t)^{(\varepsilon)}|^{p}_{L_p(\frZ_1)}\,dx
\end{equation}
and constants $N$ and $N'$ depending only on $d$ and $p$. 
By Minkowski's inequality
\begin{equation}                                                                              \label{7.14.1}
|\mu_t^{(\varepsilon)}|^p_{L_p}
=\int_{\bR^d}\Big|\int_{\bR^d}k_{\varepsilon}(x-y)\,\mu_t(dy)\Big|^p\,dx
\leq
\Big|
\int_{\bR^d}|k_{\varepsilon}|_{L_p}\,\mu_t(dy)
\Big|^p
\leq |k_{\varepsilon}|_{L_p}^p\,\mu_t^p({\bf1}), 
\end{equation}
$$
A_1(t)=\int_{\bR^d}\Big|\int_{\frZ_1}
\big|
\int_{\bR^d}(k_{\varepsilon}(x-y-\xi_t(y,\frz))-k_{\varepsilon}(x-y))\,\mu_t(dy)\big|^2\nu_1(d\frz)
\Big|^{p/2}dx
$$
$$
\leq \Big|\int_{\frZ_1}
\Big|
\int_{\bR^d}(k_{\varepsilon}(\cdot-y-\xi_t(y,\frz))-k_{\varepsilon}(\cdot-y))\,\mu_t(dy)
\Big|^{2}_{L_p}
\nu_1(d\frz)\Big|^{p/2}
$$
$$
\leq \Big|
\int_{\frZ_1}
\Big|
\int_{\bR^d}|Dk_{\varepsilon}|_{L_p}|\xi_t(y,\frz)|\,\mu_t(dy)
\Big|^{2}
\nu_1(d\frz)\Big|^{p/2}
$$
\begin{equation}                                                                              \label{7.14.2}
\leq 
|Dk_{\varepsilon}|^p_{L_p}
|\bar\xi|_{L_2(\frZ_1)}^p
\Big(
\int_{\bR^d}(K_0+K_1|y|+K_1|Y_t|)\,\mu_t(dy) 
\Big)^p,
\end{equation}
and similarly, using Assumption \ref{assumption p},
$$
A_2(t)= \int_{\bR^d}
\int_{\frZ_1}
\Big|
\int_{\bR^d}(k_{\varepsilon}(x-y-\xi_{t}(y,\frz))-k_{\varepsilon}(x-y))\,\mu_t(dy)
\Big|^{p}
\nu_1(d\frz)dx
$$
$$
\leq\int_{\frZ_1}
\Big|\int_{\bR^d}
\big|k_{\varepsilon}(\cdot-y-\xi_t(y,\frz))-k_{\varepsilon}(\cdot-y))\big|_{L_p}\mu_t(dy)
\Big|^{p}
\nu_1(d\frz)
$$
\begin{equation}                                                                                                   \label{7.14.3}
\leq K^{p-2}_{\xi}|Dk_{\varepsilon}|^p_{L_p}
|\bar\xi|^{2}_{L_2(\frZ_1)}
\Big(
\int_{\bR^d}(K_0+K_1|y|+K_1|Y_t|)\,\mu_t(dy) 
\Big)^p. 
\end{equation}                                                                            
By \eqref{7.14.4}--\eqref{7.14.3} we have a constant 
$N=N(K_{\xi},p,d,\varepsilon,|\bar\xi|_{L_2{(\frZ_1)}})$ such that 
$$
A\leq N\int_0^T\mu_t^{p}({\bf1})\,dt
+N\int_0^T
\Big(\int_{\bR^d}(K_0+K_1|y|+K_1|Y_t|)\,\mu_t(dy) 
\Big)^{p}dt<\infty \,\text{(a.s.)}.   
$$
Next we claim  that, with the operator $T^{\xi}$ defined in \eqref{equ operators TIJ}, 
\begin{equation}                                                                                                      \label{7.19.1}
\zeta_2(t)+\zeta_3(t)
=\int_0^t\int_{\frZ_1}
|(T^{\xi *}_s\mu_s)^{(\varepsilon)}|^p_{L_p}-|\mu_s^{(\varepsilon)}|^p_{L_p}
\tilde N_1(d\frz,ds)=:\zeta(t)\quad\text{for $t\in[0,T]$}. 
\end{equation}
To see that the stochastic integral $\zeta(t)$  is well-defined as an It\^o integral 
note that by Lemma \ref{lemma 7.6.1} and \eqref{7.14.1}
\begin{equation}                                                                               \label{7.27.1}                                                          
\int_0^T\int_{\frZ_1}||(T_s^{\xi *}\mu_s)^{(\varepsilon)}|^p_{L_p}
-|\mu_s^{(\varepsilon)}|^p_{L_p}|^2\,\nu_1(d\frz)ds
\leq N|\bar\xi|^2_{L_2(\frZ_1)}
\int_{0}^T|\mu_s^{(\varepsilon)}|^{2p}_{L_p}\,ds
\end{equation}
$$
\leq N|\bar\xi|^2_{L_2(\frZ_1)}|k_{\varepsilon}|_{L_p}^{2p}
\int_0^T\mu_s^{2p}({\bf1})\,ds<\infty\,\text{(a.s.)}
$$
with a constant $N=N(d,p,\lambda, K_{\xi})$. Since $\frZ_1$ is $\sigma$-finite, 
there is an increasing sequence $(\frZ_{1n})_{n=1}^{\infty}$,  
$\frZ_{1n}\in\cZ_1$,  
such that $\nu_1(\frZ_{1n})<\infty$ for every $n$ and $\cup_{n=1}^{\infty}\frZ_{1n}=\frZ_1$. 
Then it is easy to see that 
$$
\bar\zeta_{2n}(t)=p\int_0^t\int_{\frZ_1}{\bf1}_{\frZ_{1n}}(\frz)
\big(|\mu^\ep_{s}|^{p-2}\mu^\ep_{s},(I_s^{\xi *}\mu_{s})^\ep \big)\,N_{1}(d\frz,ds), 
$$
$$
\hat\zeta_{2n}(t)=p\int_0^t\int_{\frZ_1}{\bf1}_{\frZ_{1n}}(\frz)
\big(|\mu^\ep_{s}|^{p-2}\mu^\ep_{s},(I_s^{\xi *}\mu_{s})^\ep \big)\,\nu_1(d\frz)ds, 
$$
$$
\bar\zeta_{3n}(t)=\int_0^t\int_{\frZ_1}
{\bf1}_{\frZ_{1n}}(\frz)\cR_p(\xi_s(\frz), \mu_{s-}, k_{\varepsilon})\,N_1(d\frz,ds),  
$$
$$
\hat\zeta_{3n}(t)=\int_0^t\int_{\frZ_1}{\bf1}_{\frZ_{1n}}(\frz)
\cR_p(\xi_s(\frz), \mu_{s-}, k_{\varepsilon})\,\nu_1(d\frz)ds
$$
are well-defined, and 
$$
\zeta_2(t)=\lim_{n\to\infty}(\bar\zeta_{2n}(t)-\hat\zeta_{2n}(t)), 
\quad
\zeta_3(t)=\lim_{n\to\infty}\bar\zeta_{3n}(t)-\lim_{n\to\infty}\hat\zeta_{3n}(t), 
$$
where the limits are understood in probability. Hence 
$$
\zeta_2(t)+\zeta_3(t)=\lim_{n\to\infty}\Big(\bar\zeta_{2n}(t)
+\bar\zeta_{3n}(t)-\big(\hat\zeta_{2n}(t)+\hat\zeta_{3n}(t)\big)\Big)
$$
$$
=\lim_{n\to\infty}\Big(\int_0^t\int_{\frZ_1}{\bf1}_{\frZ_{1n}}(\frz)
(|(T_s^{\xi *}\mu_s)^{(\varepsilon)}|^p_{L_p}-|\mu_s^{(\varepsilon)})|^p_{L_p})
\tilde N_1(d\frz,ds)\Big)=\zeta(t), 
$$
which completes the proof of \eqref{7.19.1}. Consequently, from \eqref{7.23.1}-\eqref{7.23.2}
we have
$$
d|\mu_t^{(\varepsilon)}|^p_{L_p}               
=\cQ_p(b_t, \sigma_t, \rho_t, \beta_t,\mu_t, k_{\varepsilon})\,dt                                                                  
 +\int_{\frZ_0}\cQ_p^{(0)}(\eta_t(\frz_0),\mu_t, k_{\varepsilon}) \,\nu_0(d\frz)\,dt 
 $$
 \begin{equation}                                                                                  \label{7.23.3}                                                                                                                                                                              
 +\int_{\frZ_1}\cQ_p^{(1)}(\xi_t(\frz_1), \mu_t, k_{\varepsilon})
 +\cR_p(\xi_{t}(\frz_1), \mu_{t}, k_{\varepsilon})\,\nu_1(d\frz)\,dt
+d\zeta_1(t)+d\zeta(t).                        
\end{equation}
By  Lemma \ref{lemma pe1} and Lemma \ref{lemma pe4} we have 
\begin{equation}                                                                    \label{7.27.3}
Q(b_s, \sigma_s, \rho_s, \beta_s,\mu_s, k_{\varepsilon})
\leq N(L^2+K^2)|\mu_s^\ep|_{L_p}^p
\end{equation}
with a constant $N=N(d,p)$, 
and by Lemma \ref{lemma pe3} and Corollary \ref{corollary J}, 
using that $\bar\xi\leq K_{\xi}$ and $\bar\eta\leq K_{\eta}$, we have 
\begin{equation}                                                                    \label{7.27.4}
\cQ^{(0)}(\eta_s(\frz), \mu_s, k_{\varepsilon})
\leq N{\bar\eta}^2(\frz)|\mu_s^\ep|_{L_p}^p, 
\quad
(\cQ^{(1)}+\cR_p)(\xi_s(\frz), \mu_s, k_{\varepsilon})
\leq N{\bar\xi}^2(\frz)|\mu_s^\ep|_{L_p}^p  
\end{equation}
with a constant $N=N(K_{\xi},K_{\eta},d,p,\lambda)$. Thus from \eqref{7.23.3} for 
$c^{\varepsilon}_t:=|\mu^{(\varepsilon)}_t|^p_{L_p}$ we obtain that almost surely 
\begin{equation}                                                                                            \label{Ctp}
c_t^{\varepsilon}
\leq |\mu^{(\varepsilon)}_0|^p_{L_p}+N\int_0^tc_s^\varepsilon\,ds+m^{\varepsilon}_t
\quad
\text{for all $t\in[0,T]$}
\end{equation}
with a constant $N=N(T,p,d,K,K_{\xi}, K_{\eta}, L,\lambda,|\bar\xi|_{L_2},|\bar\eta|_{L_2})$ 
and the local martingale $m^{\varepsilon}=\zeta_1+\zeta$.  
For integers $n\geq1$ set $\tau_n=\bar\tau_n\wedge\tilde\tau_n$, where 
$(\tilde\tau_n)_{n=1}^{\infty}$ is a localising sequence of stopping times for 
$m^{\varepsilon}$
and 
$$
\bar\tau_n=\bar\tau_n(\varepsilon)=\inf\Big\{t\in[0,T]:\int_0^tc_s^\varepsilon\,ds\geq n\Big\}. 
$$
Then from  \eqref{Ctp} we get 
$$
\E c^{\varepsilon}_{t\wedge\tau_n}
\leq \E|\mu^{(\varepsilon)}_0|_{L_p}^p+N\int_0^t\E c^{\varepsilon}_{s\wedge\tau_n}\,ds<\infty
\quad
\text{for $t\in[0,T]$ and integers $n\geq1$}. 
$$
Hence by Gronwall's lemma 
$$
\E c^{\varepsilon}_{t\wedge\tau_n}\leq N\E|\mu^{(\varepsilon)}_0|_{L_p}^p
\quad
\text{for $t\in[0,T]$ and integers $n\geq1$}
$$
with a constant $N=N(T,p,d,K,,K_{\xi}, K_{\eta}, L,\lambda,|\bar\xi|_{L_2},|\bar\eta|_{L_2})$. 
Letting here $n\to\infty$, by Fatou's lemma we obtain 
\begin{equation}                                                                                           \label{7.27.2}                                                                                         
\sup_{t\in[0,T]}\E|\mu^{(\varepsilon)}_t|^p_{L_p}\leq N\E|\mu^{(\varepsilon)}_0|^p_{L_p}. 
\end{equation}
Hence we follow a standard way to prove \eqref{esupbecsles}. 
Clearly, from \eqref{Ctp}, taking into account \eqref{7.27.2}, we have a constant 
$N=N(T,p,d,K,K_{\xi}, K_{\eta}, L,\lambda,|\bar\xi|_{L_2},|\bar\eta|_{L_2})$ 
such that  
\begin{equation}                                                                                           \label{2.1.3.22}
\E\sup_{t\leq T}c^{\varepsilon}_{t\wedge\tau}\leq N\E|\mu_0^\ep|^p_{L_p}
+
\E\sup_{t\leq T}|\zeta_1(t\wedge\tau)|+\E\sup_{t\leq T}|\zeta(t\wedge\tau)|
\end{equation}
for every stopping time $\tau$. By estimates in Lemmas \ref{lemma pe1} and 
\ref{lemma 7.6.1} for the Doob-Meyer processes $\langle \zeta_1\rangle$ and 
$\langle \zeta\rangle$ of $\zeta_1$ and $\zeta$ we have 
\begin{align}
\langle\zeta_1\rangle(t)&=p^2\int_0^t \big|(|\mu^\ep_s|^{p-2}
\mu^\ep_s,(\mathcal{M}_s^{k\ast}\mu_s)^\ep \big)|^2\,ds
\leq
N_1\int_0^{t} |\mu_s^\ep|_{L_p}^{2p}\,ds<\infty,                                                  \nonumber\\
\langle\zeta\rangle(t)&=\int_0^t\int_{\frZ_1}||(T_s^{\xi *}\mu_s)^{(\varepsilon)}|^p_{L_p}
-|\mu_s^{(\varepsilon)}|^p_{L_p}|^2\nu_1(d\frz)ds
\leq N_2\int_0^{t} |\mu_s^\ep|_{L_p}^{2p}\,ds<\infty                                             \label{1.1.3.22}
\end{align}
almost surely for all $t\in[0,T]$, 
with constants $N_1=N_1(d,p,L)$ and $N_2=N_2(d,p,\lambda, K_\xi, |\bar\xi|_{L_2(\frZ_1)})$. 
Using the Davis inequality, by \eqref{1.1.3.22} and \eqref{7.27.2}
we get 
$$
\E\sup_{t\leq T}|\zeta_1(t\wedge\tau)|+\E\sup_{t\leq T}|\zeta(t\wedge\tau)|
\leq 3\E\langle\zeta_1\rangle^{1/2}(t\wedge\tau)+3\E\langle\zeta\rangle^{1/2}(t\wedge\tau)
$$
$$
\leq N'\E\Big(\int_0^{T} |\mu_{s\wedge\tau}^\ep|_{L_p}^{2p}\,ds \Big)^{1/2}
\leq N'\E\Big(\sup_{t\leq T}|\mu_{s}^\ep|_{L_p}^{p}
\int_0^{T} |\mu_{s\wedge\tau}^\ep|_{L_p}^{p}\,ds \Big)^{1/2}
$$
\begin{equation}
\label{30.9.21.3}
\leq \frac{1}{2}\E\sup_{t\leq T}|\mu_{t\wedge\tau}^\ep|_{L_p}^p 
+ N^{''}\E\int_0^T |\mu_{s}^\ep|_{L_p}^{p}\,ds
\leq \tfrac{1}{2}\E\sup_{t\leq T}|\mu_{t\wedge\tau}^\ep|_{L_p}^p
+N^{'''}\E|\mu_0^\ep|^p_{L_p}, 
\end{equation}
which by \eqref{2.1.3.22} gives
$$
\E\sup_{t\leq T}c^{\varepsilon}_{t\wedge\tau}\leq N\E|\mu_0^\ep|^p_{L_p}
+
\tfrac{1}{2}\E\sup_{t\leq T}c^{\varepsilon}_{t\wedge\tau}  
$$
with constants $N,N',N'',N'''$ depending 
on $T$, $p$, $d$, $K$, $K_{\xi}$, $K_\eta$, $L$, $\lambda$, 
$|\bar\xi|_{L_2}$ and $|\bar\eta|_{L_2}$
Substituting here the stopping time 
$$
\rho_n=\inf\{t\in[0,T]: \langle\zeta_1\rangle(t)+ \langle\zeta\rangle(t)\geq n\}
$$
in place of $\tau$, from \eqref{2.1.3.22} by virtue of the Davis inequality we have 
$$
\E\sup_{t\leq T}c^{\varepsilon}_{t\wedge\rho_n}\leq N\E|\mu_0^\ep|^p_{L_p}
+
\tfrac{1}{2}\E\sup_{t\leq T}c^{\varepsilon}_{t\wedge\rho_n}<\infty
$$
for every integer $n\geq1$. Hence 
$$
\E\sup_{t\leq T\wedge\rho_n}|\mu_t^\ep|_{L_p}^p
\leq 
2N\E |\mu_0^{(\varepsilon)}|^p_{L_p}, 
$$
and letting here $n\to\infty$ by Fatou's lemma we finish the proof of 
\eqref{esupbecsles}. 
\end{proof}

\begin{lemma}                                                             \label{lemma usup}
Let Assumptions \ref{assumption SDE},  
\ref{assumption p}
and \ref{assumption estimates}  hold. Let $(u _t)_{t\in[0,T]}$ be 
an $L_p$-solution to \eqref{equdZ} 
for an even integer $p\geq2$ such that 
$\essup_{t\in[0,T]}|u_t|_{L_1}<\infty$ (a.s.). 
If $K_1\neq 0$ in Assumption \ref{assumption SDE}, 
then assume additionally \eqref{moment2}. 
Then we have 
\begin{equation}                                                                                                     \label{eusupbecsles}
\E\sup_{t\in[0,T]}|u_t|^p_{L_p}\leq N\E|u_0|^p_{L_p}
\end{equation}
with a constant 
$N=N(p,d,T,K,K_{\xi}, K_{\eta}, L, \lambda, |\bar\xi|_{L_2}, |\bar\eta|_{L_2})$.
\end{lemma}

\begin{proof}
We may assume $\E|u_0|^p_{L_p}<\infty$. 
By Lemma \ref{lemma Ito_u_Lp} for every $\varepsilon>0$ 
equation \eqref{Ito formula u epsilon p norm} 
holds almost surely for all $t\in[0,T]$. 
Hence following the proof of Lemma \ref{lemma supemu} 
with $u^{(\varepsilon)}_t(x)$, 
$u_t(x)dx$, $u_{t-}(x)dx$, $|u_t(x)|dx$ in place of $\mu^{(\varepsilon)}_t(x)$, 
$\mu_t(dx)$, $\mu_{t-}(dx)dx$ and $|\mu_t|(dx)$, respectively, 
and taking into account that 
almost surely $u_t=u_{t-}$ for all but countable many $t\in[0,T]$,  
we obtain the counterpart 
of \eqref{Ctp}, 
\begin{align}                                                                                      
|u^{(\varepsilon)}_t|^p_{L_p}
&\leq |u^{(\varepsilon)}_0|^p_{L_p}
+N\int_0^t||u_s|^{(\varepsilon)}|^p_{L_p}\,ds+m^{\varepsilon}_t    \nonumber\\
&\leq |u_0|^p_{L_p}
+N\int_0^t|u_s|^p_{L_p}\,ds+m^{\varepsilon}_t                               
\quad
\text{almost surely for all $t\in[0,T]$}                                      \label{1.27.2.22}
\end{align}
with a constant 
$N=N(T,p,d,K,K_{\xi}, K_{\eta}, L,\lambda,|\bar\xi|_{L_2},|\bar\eta|_{L_2})$ 
and a (cadlag) local martingale 
$m^{\varepsilon}_t=\zeta_1^{\varepsilon}(t)+\zeta^{\varepsilon}(t)$, $t\in[0,T]$,  
where 
$$
\zeta_1^{\varepsilon}(t)=p\int_0^t \big(|u^\ep_s|^{p-2}
u^\ep_s,(\cM^{k\ast}_su_s)^{\varepsilon}\big)\,dV^k_s, 
$$
$$
\zeta^{\varepsilon}(t):=\int_0^t\int_{\frZ_1}
|(T^{\xi\ast}u_s)^{(\varepsilon)}|^p_{L_p}-|u_s^{(\varepsilon)}|^p_{L_p}
\tilde N_1(d\frz,ds).   
$$
Since $(u_t)_{t\in[0,T]}$ 
is a weakly cadlag $\cF_t$-adapted process, 
we have $\sup_{t\in[0,T]}|u_t|_{L_p}<\infty$ (a.s.),  and hence 
$$
\int_0^t|u_s|^r_{L_p}\,ds, \quad t\in[0,T]
$$
is a continuous $\cF_t$-adapted process for every $r>0$. 
For $\varepsilon>0$ 
and integers $n\geq1$, $k\geq1$ define the stopping times 
$\tau^{\varepsilon}_{n,k}:=\bar\tau_{n}\wedge\tilde\tau^{\varepsilon}_{k}$, 
where 
$$
\bar\tau_n:=\inf\Big\{t\in[0,T]:\int_0^t|u_s|^p_{L_p}\,ds\geq n\Big\} 
$$
for integers $n\geq1$, and 
$(\tilde \tau_k^{\varepsilon})_{k=1}^{\infty}$ is a localizing sequence for the 
local martingale $m^{\varepsilon}$. Thus from \eqref{1.27.2.22} for 
$c^{\varepsilon}_t:=|u^{\varepsilon}_t|_{L_p}^p$ 
and $c_t:=|u_t|_{L_p}^p$ we get 
$$
\E c^{\varepsilon}_{t\wedge\tau^{\varepsilon}_{n,k}}
\leq \E c_0+N\E\int_0^tc_{s\wedge\bar\tau_n}\,ds<\infty 
$$
for every $t\in[0,T]$. Letting here first $k\to\infty$ 
and then $\varepsilon \to0$ by Fatou's lemma we obtain 
$$
\E c_{t\wedge\bar\tau_n}
\leq\E c_0+N\E\int_0^tc_{s\wedge\bar\tau_n}\,ds<\infty,  \quad t\in[0,T], 
$$
which by Gronwall's lemma gives 
$$
\E c_{t\wedge\bar\tau_n}
\leq e^{NT}\E|u_0|^p_{L_p}\quad\text{for $t\in[0,T]$}. 
$$
Letting now $n\to\infty$ by Fatou's lemma we have 
\begin{equation}                                                                                      \label{3.3.3.22}
\sup_{t\in[0,T]}\E|u_t|_{L_p}^p\leq e^{NT}\E|u_0|^p_{L_p}. 
\end{equation}
Hence we are going to prove \eqref{eusupbecsles} in an already familiar way. 
Analogously to \eqref{2.1.3.22}, due to Lemmas \ref{lemma pe1} 
and \ref{lemma 7.6.1},
for the Doob-Meyer processes of $\zeta_1^{\varepsilon}$ 
and $\zeta^{\varepsilon}$ 
we have with constants $N_1=N_1(d,p,L)$ 
and $N_2=N_2(d,p,\lambda, K_\xi, |\bar\xi|_{L_2(\frZ_1)})$, 
\begin{align}
\langle\zeta^{\varepsilon}_1\rangle(t)
&=p^2\int_0^t \big|(|u^\ep_s|^{p-2}
u^\ep_s,(\mathcal{M}_s^{k\ast}u_s)^\ep \big)|^2\,ds                                   \nonumber\\
&\leq N_1\int_0^{t} ||u_s|^\ep|_{L_p}^{2p}\,ds                                                   
\leq N_1\int_0^{t} |u_s|_{L_p}^{2p}\,ds,                                                                \nonumber\\
\langle\zeta^{\varepsilon}\rangle(t)                                              
&=\int_0^t\int_{\frZ_1}||(T_s^{\xi *}u_s)^{(\varepsilon)}|^p_{L_p}
-|u_s^{(\varepsilon)}|^p_{L_p}|^2\nu_1(\frz)ds                                            \nonumber\\
&\leq N_2\int_0^{t} ||u_s|^\ep|_{L_p}^{2p}\,ds
\leq N_2\int_0^{t} |u_s|_{L_p}^{2p}\,ds.                                           \label{1.3.3.22}
\end{align} 
We define the stopping time 
$\rho_{n,k}^{\varepsilon}=\tilde\tau^{\varepsilon}_k\wedge\rho_n$, where 
$$
\rho_n=\inf\Big\{t\in[0,T]:\int_0^t|u_s|^{2p}_{L_p}\,ds\geq n\Big\} 
\quad
\text{for every integer $n\geq1$},  
$$
and $(\tilde\tau^{\varepsilon}_k)_{k=1}^{\infty}$ denotes, 
as before, a localizing sequence 
of stopping times for $m^{\varepsilon}$. 
Notice that from \eqref{1.27.2.22}, 
due to  \eqref{3.3.3.22} and \eqref{1.3.3.22}, by using 
the Davis inequality we have 
\begin{align}                                                                                           
\E\sup_{t\leq T}c^{\varepsilon}_{t\wedge\rho^{\varepsilon}_{n,k}}
&\leq N'\E|u_0|^p_{L_p}+
\E\sup_{t\leq T}|\zeta_1(t\wedge\rho^{\varepsilon}_{n,k})|
+\E\sup_{t\leq T}|\zeta(t\wedge\rho^{\varepsilon}_{n,k})|               \nonumber\\
&\leq N\E|u_0|^p_{L_p}
+N\E
\Big(
\int_0^{T\wedge\rho_n}|u_t|^{2p}_{L_p}\,dt
\Big)^{1/2}<\infty,                                                                                 \nonumber
\end{align}
where $N'$ and $N$ are constants, depending only on 
$p$, $d$, $T$, $K$, $K_{\xi}$, $K_{\eta}$ $L$ $\lambda$, 
$|\bar\xi|_{L_2}$ 
and $|\bar\eta|_{L_2}$. Letting here first $k\to\infty$ 
and then $\varepsilon\to0$ by 
Fatou's lemma we obtain 
\begin{equation}                                                                  \label{1.4.3.22}
\E\sup_{t\leq T}c_{t\wedge\rho_{n}}
\leq N\E|u_0|^p_{L_p}
+N\E
\Big(
\int_0^{T\wedge\rho_n}|u_t|^{2p}_{L_p}\,dt
\Big)^{1/2}<\infty
\quad  
\text{for every $n$.}
\end{equation}
Hence, in the same standard way as before, 
by Young's inequality and \eqref{3.3.3.22} we have 
\begin{align*}                                                                 
\E\sup_{t\leq T}c_{t\wedge\rho_{n}}
&\leq N\E|u_0|^p_{L_p}
+N\E\Big(
\sup_{t\leq T}c_{t\wedge\rho_{n}}\int_0^{T}|u_t|^{p}_{L_p}\,dt
\Big)^{1/2}                                                                                              \nonumber\\
&\leq N\E|u_0|^p_{L_p}
+\tfrac{1}{2}\E\sup_{t\leq T}c_{t\wedge\rho_{n}}
+N^2\E\int_0^{T}|u_t|^{p}_{L_p}\,dt                                                                  \nonumber\\
&\leq N'\E|u_0|^p_{L_p}
+\tfrac{1}{2}\E\sup_{t\leq T}c_{t\wedge\rho_{n}}<\infty, 
\end{align*}
with a constant 
$N'=N'(T,p,d,K,K_{\xi}, K_{\eta}, L,\lambda,|\bar\xi|_{L_2},|\bar\eta|_{L_2})$, 
which gives 
$$
\E\sup_{t\leq T}c_{t\wedge\rho_{n}}\leq 2N'\E|u_0|^p_{L_p}. 
$$
Letting here $n\to\infty$ by Fatou's lemma 
we finish the proof of  \eqref{eusupbecsles}.
\end{proof}

To formulate the next lemma let $(S,\cS)$ 
denote a measurable space, and let 
$\cH\subset\cF\otimes\cS$ be a $\sigma$-algebra.
\begin{lemma}                                                                               \label{lemma density}
Let $\mu=(\mu_s)_{s\in S}$ be an $\bM$-valued function 
on $\Omega\times S$ 
such that $\mu_s(\varphi)$ is an $\cH$-measurable 
random variable for every bounded Borel function $\varphi$ on $\bR^d$,  
and $\E\mu_s({\bf1})<\infty$ for every $s\in S$.  
Let $p>1$ and assume that for a positive sequence 
$\varepsilon_n\to0$  we have
$$
\limsup_{\varepsilon_n\to0}\E|\mu_s^{(\varepsilon_n)}|^p_{L_p}=:N_s^p<\infty 
\quad\text{for every $s\in S$}. 
$$
Then for every $s\in S$ the density $d\mu_s/dx$ exists almost surely, and there is an 
$L_p(\bR^d)$-valued $\cH$-measurable mapping $u$ on $\Omega\times S$ such that 
for each $s$ we have $u_s=d\mu_s/dx$ (a.s.).
Moreover, $\lim_{n\to\infty}|\mu_s^{(\varepsilon_n)}-u_s|_{L_p}=0$ (a.s.)    
and $\E|u_s|_{L_p}^p\leq N_s^p$ for each $s\in S$. 
 \end{lemma}

\begin{proof}
Fix $s\in S$. Since $(\mu_s^{(\varepsilon_n)})_{n=1}^{\infty}$ is a bounded sequence 
in $\bL_p:=L_p((\Omega,\cF,P), L_p(\bR^d))$ 
from any subsequence of it one can choose a subsequence, 
$\mu_s^{(\varepsilon_{n'})}$, which  converges weakly 
in $\bL_p$ to some $\bar u_s\in\bL_p$. Thus 
for every $\varphi\in C^{\infty}_0(\bR^d)$ and $G\in \cF$ we have 
$$
\E\int_{\bR^d}\mu_s^{(\varepsilon_{n'})}(x){\bf1}_G\varphi(x)\,dx\to \E\int_{\bR^d}\bar u_s(x){\bf1}_G\varphi(x)\,dx   
\quad\text{as $n'\to\infty$}. 
$$
On the other hand, since 
$$
\E\int_{\bR^d}\int_{\bR^d}k_{\varepsilon_n}(x-y){\bf1}_G|\varphi(x)|\,\mu_s(dy)dx
\leq |\mu_s^{(\varepsilon_{n})}|_{\bL_p}|\varphi|_{L_q}<\infty\quad \text{with $q=p/(p-1)$}, 
$$
we can use Fubini's theorem, and then, due to $\E\mu_s({\bf1})<\infty$, 
we can use Lebesgue's theorem 
on dominated convergence to get 
$$
\E\int_{\bR^d}\mu^{(\varepsilon_{n'})}_s(x){\bf1}_G\varphi(x)\,dx
=\E\int_{\bR^d}{\bf1}_G\varphi^{(\varepsilon_{n'})}(x)\,\mu_s(dx)
\to \E\int_{\bR^d}{\bf1}_G\varphi(x)\,\mu_s(dx). 
$$
Consequently, 
\begin{equation}                                                                                 \label{weak limit}                                                                                                                                                                                  
\E{\bf1}_G\int_{\bR^d}\varphi(x)\,\mu_s(dx)
=\E{\bf1}_G\int_{\bR^d}\varphi(x)\bar u_s(x)\,dx
\quad
\text{for any $G\in \cF$ and $\varphi\in C_0^{\infty}$}, 
\end{equation}
which implies that $d\mu_s/dx$ almost surely exists in $L_p$ and equals $\bar u_s$. 
Notice, that $\bar u_s$, as an element of $\bL_p$, 
is independent of the chosen subsequences, i.e., if $\tilde u_s$ is the weak limit 
in $\bL_p$ of some subsequence of a subsequence of $\mu_s^{(\varepsilon_n)}$, then 
by \eqref{weak limit} we have
$$
\E{\bf1}_G\int_{\bR^d}\varphi(x)\bar u_s(x)\,dx
=\E{\bf1}_G\int_{\bR^d}\varphi(x)\tilde u_s(x)\,dx
\quad
\text{for any $G\in\cF$ and $\varphi\in C_0^{\infty}$}, 
$$
which means $\bar u_s=\tilde u_s$ in $\bL_p$. Consequently, the whole sequence  
$\mu^{(\varepsilon_n)}_s$ converges weakly to $\bar u_s$ in $\bL_p$ for every $s$, 
and for each $s$ almost surely $\bar u_s=d\mu_s/dx\in L_p$. 
Hence $\mu_s^{(\varepsilon_n)}=\bar u_s^{(\varepsilon_n)}\in L_p$ (a.s.), 
and thus by a well-known property of mollifications, 
$\lim_{n\to\infty}|\mu^{(\varepsilon_n)}_s-\bar u_s|_{L_p}=0$ (a.s.). Set 
$$
A:=\{(\omega,s)\in\Omega\times S: \text{$\mu_s^{(\varepsilon_n)}$ 
is convergent in $L_p$ as $n\to\infty$}\},  
$$
and let $u_s$ denote the limit of ${\bf1}_A\mu_s^{(\varepsilon_n)}$ in $L_p$. 
Then, since $(\mu_s^{(\varepsilon_n)})_{s\in S}$ is an $L_p$-valued 
$\cH$-measurable function 
of $(\omega,s)$ for every $n$, the function $u=(u_s)_{s\in S}$
 is also an $L_p$-valued $\cH$-measurable 
function, and clearly, $u_s=d\mu_s/dx$ (a.s.) 
and $\E|u_s|^p_{L_p}\leq N_s^p$ for each $s$. 
\end{proof}

\begin{lemma}                                                                                                     \label{lemma eu1}
Let Assumptions \ref{assumption SDE},  \ref{assumption p}
and  \ref{assumption estimates} hold. Let $\mu=(\mu_t)_{t\in[0,T]}$ be 
a measure-valued solution to \eqref{measureZ}.
If $K_1\neq 0$ in Assumption \ref{assumption SDE}, 
then assume additionally 
\eqref{moment2}.
Assume $u_0= d\mu_0/dx$ exists almost surely and 
$\E|u_0|^p_{L_p}<\infty$ for some even $p\geq2$. 
Then the following statements hold. 
\begin{enumerate}
\item[(i)]
For each $t\in[0,T]$ the density $d\mu_t/dx$ 
exists almost surely, and there is an $L_p$-valued $\cF_t$-adapted 
 weakly cadlag process 
$(u_t)_{t\in[0,T]}$ such that almost surely $u_t=d\mu_t/dx$ 
for every $t\in[0,T]$ and 
$\E\sup_{t\in[0,T]}|u_t|^p_{L_p}<\infty$. 
\item[(ii)] 
If $\mu'=(\mu'_t)_{t\in[0,T]}$ satisfies the same conditions 
(with the same even integer $p$) 
as $\mu$, then for $u_t=d\mu_t/dx$ and $u'_t=d\mu'_t/dx$ we have 
\begin{equation}                                                                                     \label{7.5.1}
\E\sup_{t\in[0,T]}|u_t-u'_t|^p_{L_p}
\leq N\E|u_0-u_0'|^p_{L_p}
\quad
\text{for $t\in[0,T],$}
\end{equation}
with a constant $N$ depending only on $d$, $p$, $K$, $K_{\xi}$, $K_{\eta}$, $L$, 
$\lambda$, $T$, $|\bar\eta|_{L_2(\frZ_1)}$  
and $|\bar\xi|_{L_2(\frZ_0)}$.
\end{enumerate}
\end{lemma}
\begin{proof}
By Lemma \ref{lemma supemu} we have
$$ 
\E\sup_{t\in[0,T]}|\mu_t^{(\varepsilon)}|^p_{L_p}
\leq N\E|\mu_0^{(\varepsilon)}|^p_{L_p}<\infty
\quad\text{for every $t\in[0,T]$ and $\varepsilon>0$}
$$
with a constant 
$N=N(d,p,K, T, 
K_{\xi},K_{\eta}, L,\lambda,|\bar\eta|_{L_2(\frZ_1)},|\bar\xi|_{L_2(\frZ_0)})$. 
Moreover, by Lemma \ref{lemma density}, 
there is an $L_p$-valued $\cF_t$-adapted 
$\cF\otimes\cB([0,T])$-measurable process 
$(\bar u_t)_{t\in[0,T]}$ such that $\bar u_t=d\mu_t/dx$ (a.s.) for every $t\in[0,T]$.
To prove (i) let $A$ be a countable dense subset of $[0,T]$, 
such that $T\in A$. 
Then 
$$
\E\sup_{t\in A}|\bar u_t|_{L_p}^p
=\E\sup_{t\in A}
\liminf_{n\to\infty}|\mu_t^{(\varepsilon_n)}|_{L_p}^p
\leq \E\liminf_{n\to\infty}\sup_{t\in A}|\mu_t^{(\varepsilon_n)}|_{L_p}^p
$$
\begin{equation}                                                                                 \label{26.07.2}
\leq \liminf_{n\to \infty} \E \sup_{t\in A}|\mu_t^{(\varepsilon_n)}|_{L_p}^p
\leq N\E|d\mu_0/dx|^p_{L_p}<\infty 
\end{equation} 
for a sequence $\varepsilon_n\downarrow0$,
and there is a set $\Omega'\in\cF_0$ of full probability such that 
$$
\sup_{t\in A}|\bar u_t(\omega)|_{L_p}<\infty, 
\quad
d\mu_t/dx=\bar u_t 
\quad
 \text{for every $\omega\in\Omega'$ and $t\in A$}, 
$$
and
$
\mu_t(\varphi) 
$
is a cadlag function in $t\in[0,T]$ for $\omega\in\Omega'$ and 
$\varphi\in C^{\infty}_0(\bR^d)$. Hence, if $t\in[0,T]$ and 
$\omega\in\Omega'$,  
then there is a sequence $t_n=t_n(\omega)\in A$ 
such that $t_n\downarrow t$ and $\bar u_{t_n}(\omega)$ converges weakly 
in $L_p$ to an element, denoted by $u_t(\omega)$. Note that 
since $\bar u_{t_n}(\omega)$ is $dx$-everywhere nonnegative  
for every $n$, the function $u_t(\omega)$ is also $dx$-almost everywhere 
nonnegative. Moreover, by property of a weak limit 
we have 
\begin{equation*}
|u_t(\omega)|_{L_p}\leq \liminf_{n\to\infty}|\bar u_{t_n}(\omega)|_{L_p}
\leq \sup_{s\in A}|\bar u_{s}(\omega)|_{L_p},  
\end{equation*}
which gives 
\begin{equation}                                                           \label{sup}
\sup_{t\in[0,T]}|u_t(\omega)|_{L_p}
\leq \sup_{s\in A}|\bar u_{s}(\omega)|_{L_p}<\infty 
\quad
\text{for $\omega\in\Omega'$}. 
\end{equation}
Notice that                                                                    
\begin{equation}                                                            \label{umu}
(u_t(\omega),\varphi)
=\lim_{n\to\infty}\mu_{t_n}(\omega,\varphi)=\mu_{t}(\omega,\varphi)
\quad
\text{for $\omega\in\Omega'$, $t\in[0,T]$ and $\varphi\in C_0^{\infty}$}, 
\end{equation}
which shows that $u_t(\omega)$ does not depend on the 
sequence $t_n$. In particular, for $\omega\in\Omega'$ we have 
$\bar u_t(\omega)=u_t(\omega)$ for $t\in A$. Moreover, it shows that  
 $(u_t(\omega),\varphi)$ 
is a cadlag function of $t\in[0,T]$ for every $\varphi\in C_0^{\infty}$. 
Hence, due to \eqref{sup},  since $C_0^{\infty}$ is dense in $L_q$,
it follows that $u_t(\omega)$ is a weakly cadlag $L_p$-valued 
function of $t\in[0,T]$ for each $\omega\in\Omega'$. 
Moreover, from \eqref{umu}, by the monotone class lemma it follows that $u_t=d\mu_t/dx$ for 
every $\omega\in\Omega'$ and $t\in[0,T]$. 
Define $u_t(\omega)=0$ for $\omega\notin\Omega'$ and $t\in[0,T]$. 
Then $(u_t)_{t\in[0,T]}$ is an $L_p$-valued weakly cadlag function in $t\in[0,T]$ 
for every $\omega\in\Omega$, and since due to 
\eqref{umu} almost surely $(u_t,\varphi)=\mu_t(\varphi)$ for 
$\varphi\in C_0^{\infty}$, it follows that $u_t$ is an $\cF_t$-measurable 
$L_p$-valued random variable for every $t\in[0,T]$. Moreover, by virtue 
of \eqref{26.07.2} and \eqref{sup} we have 
$\E\sup_{t\in[0,T]}|u_t|^p_{L_p}<\infty$. 
To prove (ii), notice that by (i) the process $\bar u_t:=u_t-u'_t$, $t\in[0,T]$, is an 
$L_p$-solution to equation \eqref{equdZ} such that 
$\esssup_{t\in[0,T]}|\bar u_t|_{L_1}<\infty$ (a.s.). 
Thus we have \eqref{7.5.1} by Lemma \ref{lemma usup}.  
\end{proof}
\begin{definition}
Let $p>1$ and let $\psi$ be an $L_p$-valued $\cF_0$-measurable 
random variable. Then we say that an $L_p$-valued $\cF_t$-optional 
process $v=(v_t)_{t\in[0,T]}$ is a $\bV_p$-solution to \eqref{equdZ} 
with initial value $\psi$ if for each $\varphi\in C_0^{\infty}$ 
\begin{equation}
\begin{split}
(v_t,\vp)=&(\psi,\vp) +  \int_0^t(v_{s},\tilde\cL_s\varphi)\,ds
+ \int_0^t(v_{s},\cM_s^k\varphi)\,dV^k_s
+ \int_0^t\int_{\frZ_0}(v_{s},J_s^{\eta}\varphi)\,\nu_0(d\frz)ds\\ 
&+ \int_0^t\int_{\frZ_1}(v_{s},J_s^{\xi}\varphi)\,\nu_1(d\frz)ds
+\int_0^t\int_{\frZ_1}(v_{s},I_s^{\xi}\varphi)\,\tilde N_1(d\frz,ds) 
\end{split}
                                                                                                                \label{eqvZ}
\end{equation}
for $P\otimes dt$-a.e.  
$(\omega,t)\in\Omega\times[0,T]$. 
\end{definition}

\begin{lemma}                                                                                     \label{lemma weakly cadlag}
Let Assumption \ref{assumption SDE} (ii) holds.  Let 
$(v_t)_{t\in[0,T]}$ be a $\bV_p$-solution for a $p>1$ such that 
$\esssup_{t\in[0,T]}|v_t|_{L_p}<\infty$ (a.s.), and there is an $L_p$-valued 
random variable $g$ such that  for each $\varphi\in C_0^{\infty}$ equation 
\eqref{eqvZ} for $t:=T$ holds  almost surely with $g$ in place of $v_T$. 
Then there exists an $L_p$-solution 
$u=(u_t)_{t\in[0,T]}$ to equation \eqref{equdZ} such that $u_0=\psi$ 
and $u=v$, $P\otimes dt$-almost everywhere. 
\end{lemma}

\begin{proof}
Let $\Phi\in C_0^\infty$ be a countable dense set in $L_q$ for $q=p/(p-1)$. 
Then there is a set $\Omega'\in\Omega$ of full probability 
and for every $\omega\in\Omega'$ 
there is a set $\bT_\omega\subset[0,T]$ of 
full Lebesgue measure in $[0,T]$, such that 
$\sup_{t\in\bT_{\omega}}|v_t(\omega)|_{L_p}<\infty$  
for $\omega\in\Omega'$,   
and for all $\varphi\in\Phi$  
equation \eqref{eqvZ} holds for all $\omega\in\Omega'$ 
and $t\in\bT_\omega$. 
We may also assume that for each $\varphi\in\Phi$ 
and $\omega\in\Omega'$ 
equation  \eqref{equZ} holds for $t=T$ with $g$ in place of $v_T$. 
Since the right-hand side of equation \eqref{eqvZ}, which 
we denote by $F_t(\varphi)$ for short, is 
almost surely a cadlag function of $t$, we may assume, 
that for $\omega\in\Omega'$ 
it is cadlag for all $\varphi\in\Phi$. 
Since $\bT_{\omega}$ is dense in $[0,T]$ and 
$\sup_{t\in\bT_{\omega}}|v_t(\omega)|_{L_p}<\infty$ for 
$\omega\in\Omega'$, for each $\omega\in\Omega'$ 
and $t\in[0,T)$ we have a sequence 
$t_n=t_n(\omega)\in\bT_{\omega}$  such that 
$t_n\downarrow t$ and $v_{t_n}\to \bar v_t$ weakly in $L_p$ 
for some element $\bar v_t=\bar v_t(\omega)\in L_p$. Hence 
\begin{equation}                                                                                \label{modification}
(\bar v_t(\omega),\varphi)=\lim_{n\to\infty}(v_{t_n}(\omega),\varphi)
=\lim_{n\to\infty}F_{t_n(\omega)}(\omega,\varphi)
=F_{t}(\omega,\varphi)\quad\text{for all $\varphi\in \Phi$}, 
\end{equation}
which implies that for every sequence $t_n=t_n(\omega)\in\bT_{\omega}$ 
such that $t_n\downarrow t$ the sequence $v_{t_n(\omega)}(\omega)$ 
converges weakly to $\bar v_t(\omega)$ in $L_p$. In particular, 
$\bar v_t(\omega)=v_t(\omega)$ for $\omega\in\Omega'$ 
and $t\in\bT_{\omega}$. For $\omega\in\Omega'$ 
we define $u_t(\omega):=\bar v_t(\omega)$ for $t\in[0,T)$ 
and $u_T(\omega):=g(\omega)$, and for $\omega\in\Omega\setminus\Omega'$ 
we set $u_t(\omega)=0$ for all $t\in[0,T]$. 
Then due to \eqref{modification} and that almost surely 
$(u_T,\varphi)=F_T(\varphi)$ 
for all $\varphi\in\Phi$, the process $u=(u_t)_{t\in[0,T]}$ is 
an $L_p$-valued $\cF_t$-adapted  
weakly cadlag process such that almost surely \eqref{eqvZ} 
holds for all $\varphi\in C_0^{\infty}$. 
Clearly, $u=v$ $P\otimes dt$ (a.e.).  Thus we also have that almost surely 
$$
\int_0^t\int_{\frZ_1} (u_{s-},I_s^\xi\vp)\,\Nte(d\frz,ds) =  \int_0^t\int_{\frZ_1} (u_{s},I_s^\xi\vp)\,\Nte(d\frz,ds)
$$
for all $t\in[0,T]$ and hence $u$ satisfies \eqref{eqvZ}, with $u_s$ replaced by $u_{s-}$ in the last term on the right-hand side, almost surely for all $\vp\in C_0^\infty$ for all $t\in [0,T]$, i.e., $u$ is an $L_p$ solution to \eqref{equdZ}.
\end{proof}

\section{Solvability of the filtering equations in $L_p$-spaces}

To show the solvability of the linear filtering equation \eqref{equdZ}, the Zakai equation,
with any $\cF_0$-measurable $L_p$-valued initial condition,  
we want to apply the existence and uniqueness theorem for 
stochastic integro-differential equations proved in \cite{GW2020}. 
With this purpose in mind first we assume that the coefficients 
$\sigma$, $b$, $\rho$, $B$, $\xi$, $\eta$ are smooth in $x\in\bR^d$, 
and under this additional assumption we are going 
to determine the form of the  ``adjoint" 
operators $\tilde\cL^{\ast}$, $\cM^{k\ast}$, $J^{\eta\ast}$, $J^{\xi\ast}$ 
and $I^{\xi\ast}$ as operators acting directly on $C^{\infty}_0$ such that 
$$
\int_{\bR^d} A^{\ast}\varphi(x)\phi(x)\,dx=
\int_{\bR^d} \varphi(x)A\phi(x)\,dx\quad \text{for all $\varphi, \phi\in C_0^{\infty}$}, 
$$ 
for $\tilde\cL$, $\cM$, $J^{\xi}$, $J^{\eta}$ and $I^{\xi}$ in place of $A$. 
The form of $\tilde\cL^{\ast}$ and $\cM^{k\ast}$ 
is immediately obvious by integrating by parts. 
To find the form of the other operators (defined in \eqref{equ operators TIJ}), let 
$\zeta:\bR^d\to\bR^d$ 
such that 
$$
\tau(x):=\tau^{\zeta}(x):=x+\zeta(x),\quad x\in\bR^d, 
$$
is a $C^1$-diffeomorphism on $\bR^d$. 
Then observe that for $\varphi$, $\phi\in C^{\infty}_0$ 
we have 
$$
(\phi, T^{\zeta}\varphi)
= \int_{\bR^d} \phi(\tau^{ -1}(x))|\det D\tau^{-1}(x)|\vp(x)\,dx
=(|\det D\tau^{-1}|\, T^{\zeta^{\ast}}\phi,\varphi)
$$
with 
\begin{equation}                                                                                 \label{starfunction}
\zeta^{\ast}(x):=-x+\tau^{-1}(x)=-\zeta(\tau^{-1}(x)), 
\quad
T^{\zeta^{\ast}}\phi(x)=\phi(x+\zeta^{\ast}(x)).
\end{equation}
Similarly, 
\begin{align*} 
(\phi, I^{\zeta}\varphi) 
&= \int_{\bR^d} 
\big(
\phi(\tau^{ -1}(x))|\det D\tau^{-1}(x)|-\phi(x)
\big)\vp(x)\,dx                                                                                                  \\
&=\int_{\bR^d}
\big(
\phi(\tau^{-1}(x))|\det D\tau^{-1}(x)|-\phi(\tau^{-1}(x))+\phi(\tau^{-1}(x))-\phi(x)
\big)
\vp(x)\,dx                                                                                                                 \\
&=(\frc T^{\zeta^{\ast}}\phi,\vp)+(I^{\zeta^{\ast}}\phi, \vp),  
\end{align*}
where
$$
\frc(x)=|\det D\tau^{-1}(x)|-1, 
$$
and 
\begin{align*}
(\phi,J^{\zeta}\vp) &= (\phi,I^\zeta\vp) - (\phi,\zeta^iD_i\vp)                                        \\
&=  (I^{\zeta^{\ast}}\phi, \varphi)
+ (\frc T^{\zeta^{\ast}}\phi,\varphi) 
+ (\zeta^iD_i\phi,\varphi) + ((D_i\zeta^i)\phi,\varphi) \\
&= (J^{\zeta^{\ast}}\phi, \varphi) + (\frc I^{\zeta^{\ast}}\phi,\varphi)
+ ((\frc + D_i\zeta^i)\phi,\varphi) + (({\zeta^{\ast}}^i+\zeta^i)D_i\phi,\varphi) \\
&=(J^{\zeta^{\ast}}\phi, \varphi) + (\frc I^{\zeta^{\ast}}\phi,\varphi)
+ ((\bar\frc 
+D_i{\zeta^{\ast}}^i +D_i\zeta^i)\phi,\varphi) 
+ (({\zeta^{\ast}}^i+\zeta^i)D_i\phi,\varphi),  
\end{align*}
where
$$
\bar\frc=|\det D\tau^{-1}(x)|-1-D_i{\zeta^{\ast}} ^i. 
$$
Consequently, $T^{\zeta\ast}$, $I^{\zeta\ast}$ and $J^{\zeta\ast}$, 
the formal adjoint of  $T^{\zeta}$, $I^{\zeta}$ and $J^{\zeta}$, 
can be written in the form 
\begin{equation*}
T^{\zeta\ast}=|\det D\tau^{-1}|\, T^{\zeta^{\ast}},  
\end{equation*}
\begin{equation}                                                                       \label{*}
I^{\zeta\ast}=I^{\zeta^{\ast}}+\frc T^{\zeta^{\ast}},
\quad
J^{\zeta\ast}=J^{\zeta^{\ast}}+\frc I^{\zeta^{\ast}}
+(\zeta^{\ast i}+\zeta^i)D_i +\bar\frc +D_i(\zeta^{\ast i}+\zeta^i).  
\end{equation}
\begin{lemma}                                                                     \label{lemma 1.14.3.22}
Let $\zeta$ be an $\bR^d$-valued function on $\bR^d$ such that 
for an integer $m\geq1$ it is continuously differentiable up to order $m$, 
and 
\begin{equation}                                                              \label{diffeomorphism}
\inf_{x\in\bR^d}|\det(I+D\zeta(x))|=:\lambda>0, 
\quad 
\max_{1\leq k\leq m}\sup_{x\in\bR^d}|D^k\zeta(x)|=:M_m<\infty. 
\end{equation}
Then the following statements hold.
\begin{enumerate}
\item[(i)] The function $\tau=x+\zeta(x)$, $x\in\bR^d$, 
is a $C^m$-diffeomorphism, such that 
\begin{equation}                                                                  \label{Dinverse}
\inf_{x\in\bR^d}|\det D\tau^{-1}(x)|\geq\lambda',
\quad 
\max_{1\leq k\leq m}\sup_{x\in\bR^d}|D^k\tau^{-1}|\leq M'_{m}<\infty, 
\end{equation}
with constants $\lambda'=\lambda'(d, M_1)>0$  
and $M'_m=M'_m(d,\lambda,M_m)$. 
\item[(ii)] The function $\zeta^{\ast}(x)=-x+\tau^{-1}(x)$, $x\in\bR^d$,  
is continuously differentiable up to order $m$, such that 
\begin{align}
\sup_{\bR^d}|\zeta^{\ast}|&=\sup_{\bR^d}|\zeta|,                                   \label{1.12.3.22} \\
\sup_{\bR^d}|D^k\zeta^{\ast}|
&\leq M^{\ast}_m\max_{1\leq j\leq k}\sup_{\bR^d}|D^j\zeta|, 
\quad\text{for $k=1,2,...,m$},                                                                  \label{2.12.3.22}\\                                                               
\inf_{\theta\in[0,1]}\inf_{\bR^d}|\det(I+\theta D\zeta^{\ast})|
&\geq \lambda'\inf_{\theta\in[0,1]}\inf_{\bR^d}|\det(I+\theta D\zeta)|,              \label{3.12.3.22}                                         
\end{align}
with a constant $M_m^{\ast}=M_m^{\ast}(d,\lambda,M_m)$ and with 
$\lambda'$ from  \eqref{Dinverse}.
\item[(iii)] For the functions 
$\frc=\det(I+D\zeta^{\ast})-1$, 
$\bar\frc=\frc-D_i\zeta^{\ast i}$ and $\zeta+\zeta^{\ast}$ 
we have 
\begin{equation}                                                                                        \label{4.12.3.22}
\sup_{x\in\bR^d}|D^k\frc(x)|
\leq N\max_{1\leq j\leq k+1}\sup_{\bR^d}|D^j\zeta|,        
\quad 
\sup_{x\in\bR^d}|D^k\bar\frc(x)|
\leq N\max_{1\leq j\leq k+1}\sup_{\bR^d}|D^j\zeta|^2
\end{equation}                                                                                            \label{5.12.3.22}
\begin{equation}                                                                                         \label{6.12.3.22}
\sup_{\bR^d}|D^{k}(\zeta+\zeta^{\ast})|
\leq N \max_{1\leq j\leq k+1}\sup_{\bR^d}|D^j\zeta|^2.                             
\end{equation}
for $1\leq k\leq m-1$ with a constant $N=N(d,\lambda, m, M_m)$. 
\end{enumerate}
\end{lemma}
\begin{proof}
To prove (i) note that \eqref{diffeomorphism} implies that $\tau$ is 
a $C^m$-diffeomorphism and the estimates in \eqref{Dinverse} are proved 
in \cite{DGW} (see Lemma 3.3 therein). By (i) the mapping $\tau$ 
is a $C^m$-diffeomorphism. From $\tau(x)=x+\zeta(x)$, by substituting 
$\tau^{-1}(x)$ in place of $x$ we obtain $\zeta^{\ast}(x)=-\zeta(\tau^{-1}(x))$. 
Hence \eqref{1.12.3.22} follows immediately, and due to the second estimates 
in \eqref{Dinverse}, the estimate in 
\eqref{2.12.3.22} also follows.  Notice that
$$
x+\theta\zeta^{\ast}(x)=\tau^{-1}(x)+\zeta(\tau^{-1}(x))-\theta\zeta(\tau^{-1}(x))
=\tau^{-1}(x)+(1-\theta)\zeta(\tau^{-1}(x)).  
$$
Hence, by the first inequality in \eqref{Dinverse}, 
\begin{align*}
|\det(I+\theta D\zeta^{\ast})|&=|\det(I+(1-\theta)D\zeta(\tau^{-1}))||\det D\tau^{-1}|  \\
&\geq\lambda' |\det(I+(1-\theta)D\zeta(\tau^{-1}))|,                                              
\end{align*}
which implies  \eqref{3.12.3.22}. To prove the inequalities in \eqref{4.12.3.22} 
notice that for the function $F(A)=\det A$, considered as the function 
of the entries $A^{ij}$ of $d\times d$ real matrices $A$, 
we have 
$$
\frac{\partial}{\partial A^{ij}}\det A\big |_{A=I}=\delta_{ij}, \quad i,j=1,2,...,d.
$$
Thus 
$$
\frac{\partial}{\partial\theta}\det(I+\theta D\zeta^{\ast})\big|_{\theta=0}
=\delta_{ij}D_j\zeta^{\ast i}=D_i\zeta^{\ast i}, 
$$
and by Taylor's formula we get 
$$
\frc=\det(I+D\zeta^{\ast})-\det I
=\int_0^1
\tfrac{\partial}{\partial A^{ij}}F(I+\theta D\zeta^{\ast})\,d\theta D_i\zeta^{\ast j}  
$$
and
$$
\bar\frc=\det(I+D\zeta^{\ast})-\det I-D_i\zeta^{\ast i}
=\int_0^1(1-\theta)\tfrac{\partial^2}{\partial A^{ij}\partial A^{kl}}
F(I+\theta D\zeta^{\ast})\,d\theta D_i\zeta^{\ast j}D_k\zeta^{\ast l}.
$$
Hence using the estimates in \eqref{2.12.3.22} 
we get \eqref{4.12.3.22}. Note that  
$$
\zeta+\zeta^{\ast}=\zeta(\tau^{-1}-\theta\zeta^{\ast})\big|^{\theta=1}_{\theta=0}
=\zeta^{\ast i}\int_0^1(D_i\zeta)(\tau^{-1}-\theta\zeta^{\ast})\,d\theta. 
$$
Hence by the second estimate in \eqref{diffeomorphism} and \eqref{2.12.3.22} 
we obtain \eqref{6.12.3.22}. 
\end{proof}

In this section for $\varepsilon>0$ and functions $v$ on $\bR^d$ 
we use the notation $v^{(\varepsilon)}$ for the convolution of $v$ 
with $\kappa_{\varepsilon}(\cdot)=\varepsilon^{-d}\kappa(\cdot/\varepsilon)$, where $\kappa$ 
is a fixed nonnegative $C_0^{\infty}$ function of unit integral such that 
$\kappa(x)=0$ for $|x|\geq1$ and $\kappa(-x)=\kappa(x)$ 
for $x\in\bR^d$. 
\begin{lemma}                                                                                \label{lemma ediff}
Let $\tau$ be an $\bR^d$-valued function on $\bR^d$ 
with uniformly continuous derivative  
$D\tau$ on $\bR^d$ such that with positive constants $\lambda$ and $K$ 
$$
\lambda\leq |\det D\tau| \quad\text{and}\quad |D\tau|\leq M\quad \text{on $\bR^d$}.  
$$
Then 
$$
\tfrac{1}{2}\lambda\leq |\det D\tau^{(\varepsilon)}| \quad\text{on $\bR^d$} 
$$
for $\varepsilon\in(0,\varepsilon_0)$ for $\varepsilon_0>0$ satisfying 
$\delta(\varepsilon_0)\leq \lambda/(2d!dM^{d-1})$, where 
$\delta=\delta(\varepsilon)$ is the modulus of continuity of $D\tau$. 
\end{lemma}
\begin{proof}
Clearly,
$$
\sup_{x\in\bR^d}|D_{j}\tau^i-D_{j}\tau^{i(\varepsilon)}|\leq\delta(\varepsilon)
\quad
\text{for $\varepsilon>0$, $i,j=1,2,...,d$}.  
$$
Hence 
$$
\sup_{x\in \bR^d}|\Pi_{i=1}^dD_{j_i}\tau^i-\Pi_{i=1}^dD_{j_i}\tau^{i(\varepsilon)}|
\leq \sum_{i=1}^dM^{d-1}\sup_{\bR^d}|D_{j_i}\tau^i-D_{j_i}\tau^{i(\varepsilon)}|
\leq dM^{d-1}\delta(\varepsilon)
\quad
\text{for $\varepsilon>0$},   
$$
for every permutation $(j_1,...,j_d)$ of $1,2,...,d$. 
Therefore 
$$
\sup_{x\in\bR^d}|\det D\tau-\det D\tau^{(\varepsilon)}|
\leq d!\,dM^{d-1}\delta(\varepsilon)
\quad
\text{for $\varepsilon>0$}. 
$$
Consequently, choosing $\varepsilon_0>0$ such that 
$\delta(\varepsilon_0)\leq \lambda/(2d!dM^{d-1})$, 
for $\varepsilon\in(0,\varepsilon_0)$ 
we have 
$$
|\det D\tau^{(\varepsilon)}|
\geq 
|\det D\tau|-|\det D\tau-\det D\tau^{(\varepsilon)}|
\geq \lambda/2
\quad
\text{on $\bR^d$}. 
$$
\end{proof}
\begin{corollary}                                                                   \label{corollary 1.13.3.22}
Let $\zeta$ be an $\bR^d$-valued function on $\bR^d$ such that 
$D\zeta$ is a uniformly continuous function on $\bR^d$ and 
\begin{equation}                                                                     \label{3.13.3.22}
0<\lambda\leq\inf_{\bR^d}\det(I+D\zeta), 
\quad
\sup_{\bR^d}|D\zeta|\leq M<\infty  
\end{equation}
with some positive constants $\lambda$ and $M$. 
Let $\varepsilon_0>0$ such that $\delta(\varepsilon_0)
\leq \lambda/(2d!dM^{d-1})$.  
Then for every $\varepsilon\in(0,\varepsilon_0)$ 
the first inequality in \eqref{3.13.3.22} holds 
for $\zeta^{(\varepsilon)}$ in place of $\zeta$ with $\lambda/2$ 
in place of $\lambda$. Moreover, 
$\sup_{\bR^d}|D^k\zeta^{(\varepsilon)}|\leq M_k$ for every 
integer $k$ with a constant $M_k=M_k(d, M,\varepsilon)$, where $M_1=M$. 
Hence Lemma \ref{lemma 1.14.3.22} holds 
with $\zeta^{(\varepsilon)}$ in place of $\zeta$, 
for $\varepsilon\in(0,\varepsilon_0)$ 
for every integer $m\geq1$. 
\end{corollary}

Consider for $\varepsilon\in(0,1)$ the equation 
\begin{align}                                                                       
du_t^{\varepsilon}=&\tilde\cL_t^{\varepsilon\ast}u_t^{\varepsilon}\,dt
+\cM_t^{\varepsilon k\ast}u_t^{\varepsilon}\,dV^k_t
+\int_{\frZ_0}J_t^{\eta^{\varepsilon}\ast}u_t^{\varepsilon}\,\nu_0(d\frz)dt                        \nonumber\\
&+\int_{\frZ_1}J_t^{\xi^{\varepsilon}*}u_t^{\varepsilon}\,\nu_1(d\frz)dt
+\int_{\frZ_1}I_t^{\xi^{\varepsilon}*} u_t^{\varepsilon}\,\tilde N_1(d\frz,dt),
\quad
\text{with $u_0^{\varepsilon}=\psi^{(\varepsilon)}$,}                                                           \label{equdZn}
\end{align}
where  
$$
\cM_t^{\varepsilon k} = \rho^{(\varepsilon) ik}_t D_i + B_t^{(\varepsilon)k},\quad k=1,\dots,d',
$$
$$
\tilde\cL_t^{\varepsilon}=a_t^{\varepsilon, ij}D_{ij}
+b_t^{(\varepsilon)i}D_i+
\beta^k_t\cM^{\varepsilon k}_t, 
\quad
\beta_t=B(t,X_t,Y_t),  
$$
$$
a_t^{\varepsilon, ij}
:=\tfrac{1}{2}\sum_k (\sigma_t^{(\varepsilon)ik}\sigma_t^{(\varepsilon)jk}
+\rho_t^{(\varepsilon)ik}\rho_t^{(\varepsilon)jk}), 
\quad i,j=1,2,...,d, 
$$
the operators $J_t^{\eta^{\varepsilon}}$ and $J_t^{\xi^{\varepsilon}}$ 
are defined as 
$J^{\xi}_t$ in \eqref{IJ}  with $\eta^{(\varepsilon)}_t$ and 
$\xi^{(\varepsilon)}_t$ in place of $\xi_t$,  and the operator 
 $I_t^{\xi^{\varepsilon}}$ is  
defined as $I_t^{\xi}$ in \eqref{IJ} with $\xi_t^{(\varepsilon)}$ in place of 
$\xi_t$. 
(Remember that $v^{(\varepsilon)}$ 
denotes the convolution of functions $v$ in $x\in\bR^d$, with the kernel 
$\kappa_{\varepsilon}$ described above.) We define the 
$L_p$-solution $(u^{\varepsilon}_t)_{t\in[0,T]}$ to \eqref{equdZn} 
in the sense of Definition \ref{def Lp solution}. 
Define now for each $\omega\in\Omega$, $t\geq0$ 
and $\frz_i\in\frZ_i$ the functions 
\begin{equation}
\label{12.7.22.1}
\tau^{\eta_t^{\varepsilon}}(x)=x + \eta_t^{(\varepsilon)}(x),
\quad
\tau^{\xi_t^{\varepsilon}}(x)=x + \xi_t^{(\varepsilon)}(x),
\quad x\in\bR^d,  
\end{equation}
where, and later on, we suppress the variables $\frz_i$, $i=0,1$. 

We recall that for $p\geq1$, $\bL_p$ denotes the space of $L_p$-valued 
$\cF_0$-measurable random variables $Z$ such that $\E|Z|_{L_p}^p<\infty$, as well as that 
for $p,q\geq 1$ the notation $\bL_{p,q}$ stands for the space of $L_p$-valued 
$\cF_t$-optional processes $v=(v_t)_{t\in[0,T]}$ such that   
$$
|v|_{\bL_{p,q}}^p:=\E\left(\int_0^T|v_t|^q_{L_p}\,dt\right)^{p/q}<\infty.
$$
Let $\bB_0$ denote the set of $\cF_0\otimes\cB(\bR^d)$-measurable 
real-valued bounded functions $\psi$ on $\Omega\times\bR^d$ such that 
and $\psi(x)=0$ for $|x|\geq R$ 
for some constant $R>0$ depending on $\psi$.
It is easy to see that $\bB_0$ is a dense 
subspace of $\bL_p$ for every $p\in[1,\infty)$. 

\begin{lemma}                                                                                            \label{lemma compact}
Let Assumptions \ref{assumption SDE},  
\ref{assumption p} and \ref{assumption estimates} 
hold with  $K_1=0$. 
Assume that the following ``support condition" holds: There is 
some $R>0$ such that 
\begin{equation}                                                                                            \label{supp_condition}
\big(b_t(x),B_t(x),\sigma_t(x),\rho_t(x), \eta_t(x,\frz_0),\xi_t(x,\frz_1)\big)=0
\end{equation}
for $\omega\in\Omega$, $t\geq0$, $\frz_0\in\frZ_0$, $\frz_1\in\frZ_1$ and  
$x\in\bR^{d}$ such that 
$|x|\geq R$.
Let $\psi\in\bB_0$ such that $\psi(x)=0$ if $|x|\geq R$.
Then there exists an $\varepsilon_0>0$ 
and a constant $\bar R=\bar R(R,K,K_0,K_\xi,K_\eta)$ 
such that the following statements hold for all $m\geq 1$ 
and even integers $p \geq 2$.\newline
(i) For every $\varepsilon\in (0,\varepsilon_0)$ 
there is an $L_p$-solution $u^\varepsilon = (u^\varepsilon_t)_{t\in [0,T]}$ 
to \eqref{equdZn}, which is a $W^m_p$-valued weakly cadlag process. 
Moreover, it satisfies
\begin{equation}
\label{12.7.22.2}
\E\sup_{t\in [0,T]}|u^\varepsilon_t|_{W^m_p}^p<\infty
\quad
\text{and}\quad u^\varepsilon_t(x) = 0,
\quad
\text{for $dx$-a.e. $x\in \{x\in\bR^d:|x|\geq \bar R\}$,} 
\end{equation}
almost surely for all $t\in [0,T]$.
\newline
(ii) There exists a unique 
$L_p$-solution $u=(u_t)_{t\in[0,T]}$ to equation \eqref{equdZ} 
with initial condition $u_0=\psi$, such that 
almost surely $u_t(x)=0$ for $dx$-almost every  
$x\in\{x\in\bR^d:|x|\geq \bar R\}$ for every $t\in[0,T]$ 
and 
\begin{equation}                                                                                                   \label{30.9.21.5}
\E\sup_{t\in [0,T]}|u_t|_{L_p}^p\leq N\E|\psi|_{L_p}^p
\end{equation}
with a constant 
$N=N(d,p,T,K, K_{\xi}, K_{\eta}, L,\lambda,|\bar{\xi}|_{L_2},|\bar{\eta}|_{L_2})$.
\newline
(iii) If $(\varepsilon_n)_{n=1}^\infty\subset (0,\varepsilon_0)$ such that 
$\varepsilon\to0$ then we have
$$
u^{\varepsilon_n}\rightarrow u
\quad
\text{weakly in $\bL_{p,q}$, for every integer $q\geq 1$}.
$$
\end{lemma}
\begin{proof}
We may assume $\E|\psi|^p_{L_p}<\infty$.
To prove (i), we look for a $W^m_p$-valued 
weakly cadlag $\cF_t$-adapted process $(u^{\varepsilon}_t)_{t\in[0,T]}$ such that 
for each $\varphi\in C_0^{\infty}$ almost surely 
\begin{align}                                                                       
(u_t^{\varepsilon},\varphi)=&(\psi^{(\varepsilon)},\varphi)+
\int_0^t(\tilde\cL_s^{\varepsilon\ast}u_s^{\varepsilon},\varphi)\,ds
+\int_0^t(\cM_s^{\varepsilon k\ast}u_s^{\varepsilon},\varphi)\,dV^k_s
+\int_0^t\int_{\frZ_0}
(J_s^{\eta^{\varepsilon}\ast}u_s^{\varepsilon},\varphi)\,\nu_0(d\frz)\,ds                        \nonumber\\
&+\int_0^t\int_{\frZ_1}
(J_s^{\xi^{\varepsilon}*}u_s^{\varepsilon},\varphi)\,\nu_1(d\frz)\,ds
+\int_0^t\int_{\frZ_1}
(I_s^{\xi^{\varepsilon}*}u_s^{\varepsilon},\varphi)\,\tilde N_1(d\frz,ds),   \label{equdZ*}
\end{align}
holds for all $t\in[0,T]$, where by virtue of \eqref{*} 
\begin{equation*}                                                                     
I_s^{\xi^{\varepsilon}\ast}=I^{\xi_s^{\varepsilon\ast}}
+\frc^{\xi_s^{\varepsilon}}T^{\xi_s^{\varepsilon\ast}},
\quad
J^{\xi^{\varepsilon}\ast}_s=J^{\xi^{\varepsilon\ast}}_s
+\frc^{\xi_s^{\varepsilon}}I^{\xi_s^{\varepsilon\ast}}
+(\xi_s^{\varepsilon\ast i}+\xi_s^{(\varepsilon) i})D_i 
+\bar\frc^{\xi_t^{\varepsilon}} +D_i(\xi_s^{\varepsilon\ast i}+\xi_s^{(\varepsilon) i}),   
\end{equation*}
\begin{equation}
J^{\eta^{\varepsilon}\ast}_s=J^{\eta^{\varepsilon\ast}}_s
+\frc^{\eta_t^{\varepsilon}}I^{\eta_s^{\varepsilon\ast}}
+(\eta_s^{\varepsilon\ast i}+\eta_s^{(\varepsilon) i})D_i 
+\bar\frc^{\eta_s^{\varepsilon}} +D_i(\eta_s^{\varepsilon\ast i}
+\eta_s^{(\varepsilon) i}), 
\end{equation}
with the functions
$$
\eta_t^{\varepsilon\ast}(x)=-x + (\tau^{\eta^{\varepsilon}_t})^{-1}(x),  
\quad                          
\xi_t^{\varepsilon\ast}(x)=-x + (\tau^{\xi_t^{\varepsilon}})^{-1}(x), 
$$
\begin{equation*}
\frc^{\xi_t^{\varepsilon}}(x)=|\det D(\tau^{\xi_t^{\varepsilon}})^{-1}(x)|-1,
\quad 
\frc^{\eta_t^{\varepsilon}}(x)=|\det D(\tau^{\eta_t^{\varepsilon}})^{-1}(x)|-1,
\end{equation*}
\begin{align}                                                                                      
\bar\frc^{\xi_t^{\varepsilon}}(x)=&|\det D(\tau^{\xi_t^{\varepsilon}})^{-1}(x)|-1
-D_i\xi_t^{\varepsilon\ast i}(x),                                                                        \nonumber\\
\quad
\bar\frc^{\eta_t^{\varepsilon}}(x)=&|\det D(\tau^{\eta_t^{\varepsilon}})^{-1}(x)|
-1-D_i\eta_t^{\varepsilon\ast i}(x)
\quad x\in\bR^d,                                                                      \label{coefficients smoothed}
\end{align}
and clearly, 
$$
\cM_s^{\varepsilon * k}\phi = -D_i(\rho^{(\varepsilon) ik}_s \phi ) + B_s^{(\varepsilon)k}\phi,
$$
$$
\tilde\cL_s^{\varepsilon\ast}\phi=D_{ij}(a_s^{\varepsilon, ij}\phi)
-D_i(b^{(\varepsilon)i}_s\phi)
-\beta^k_sD_i(\rho_s^{(\varepsilon)ik}\phi)
+\beta^k_sB^{(\varepsilon)k}_s\phi
\quad
\text{for $\phi\in W^m_p$}. 
$$
Note that by Assumption \ref{assumption SDE}(i) together with Assumption \ref{assumption estimates}(i) \& (ii), for each $\omega\in\Omega$, $ t\in [0,T]$ and $\frz_i\in\frZ_i$, $i=0,1$, the mappings 
$$
\tau^\eta(x)=x+\eta_t(x,\frz_0),\quad\text{and}\quad
\tau^\xi(x) = x+\xi_t(x,\frz_1)
$$
are biLipschitz and continuously differentiable as functions of $x\in\bR^d$. Hence, as biLipschitz functions admit Lipschitz continuous inverses, it is easy to see that for each $\omega\in\Omega$, $ t\in [0,T]$ and $\frz_i\in\frZ_i$, $i=0,1$,
$$
\lambda'\leq\inf_{x\in\bR^d} |\det D\tau^\eta(x)|,\quad\text{and}\quad
\lambda'\leq\inf_{x\in\bR^d} |\det D\tau^\xi(x)|
$$
for some $\lambda'=\lambda'(d,\lambda,L,K_\eta,K_\xi)$.
Due to Assumption \ref{assumption estimates} (i) by virtue of 
Corollary \ref{corollary 1.13.3.22}  there is $\varepsilon_0\in(0,1)$  
such that 
for $\varepsilon\in(0,\varepsilon_0)$ the functions $\tau^{\eta_t^{\varepsilon}}$ 
and $\tau^{\xi_t^{\varepsilon}}$, defined in \eqref{12.7.22.1}, 
are $C^{\infty}$-diffeomorphisms on $\bR^d$ 
for all $\omega\in\Omega$, $t\in[0,T]$ and $\frz_i$, $i=0,1$.  
Moreover, the functions defined in \eqref{coefficients smoothed} 
are infinitely differentiable functions in $x\in\bR^d$,  for all
$t\in[0,T]$ and $\frz_i\in\frZ_i$, $i=0,1$. 

 Hence we can easily 
verify that for each $\varepsilon\in(0,\varepsilon_0)$ equation  \eqref{equdZn} 
satisfies the conditions of the existence and uniqueness theorem, 
Theorem 2.1 in \cite{GW2020}. Hence  \eqref{equdZn} 
has a unique $L_p$-solution $(u^{\varepsilon}_t)_{t\in[0,T]}$ 
which is weakly cadlag as $W^m_p$-valued process and satisfies the first equation in \eqref{12.7.22.2}, 
for every $m\geq1$. 
Due to the support condition \eqref{supp_condition} and that 
$|\xi|\leq K_0K_{\xi}$,  $|\eta|\leq  K_0K_{\eta}$, there is a constant 
$\bar R=\bar R(R, K_0,K,K_{\xi}, K_{\eta})$ such that for 
$\varepsilon\in(0,\varepsilon_0)$ 
and $s\in[0,T]$ we have 
$$
\tilde{\cL}^{\varepsilon}_s\varphi=\cM_s^{\varepsilon k}\varphi
=I_s^{\xi^{\varepsilon}}\varphi=J_s^{\xi^{\varepsilon}}\varphi
=J_s^{\eta^{\varepsilon}}\varphi=0, \quad k=1,2,..., d', 
$$
for all $\varphi\in C_0^{\infty}$ such that $\varphi(x)=0$ for $|x|\leq \bar R$. Thus 
from equation \eqref{equdZ*} we get that almost surely
$$
(u_t^{\varepsilon},\varphi)=0
\quad
\text{for all $\varphi\in C_0^{\infty}$ such that $\varphi(x)=0$ for $|x|\leq \bar R$}
$$
for all $t\in[0,T]$, which implies  
\begin{equation}                                                                             \label{4.24.3.22}
u_t^{\varepsilon}=0
\quad 
\text{for $dx$-almost every $x\in\{x\in\bR^d, |x|\geq \bar R$\} 
for all $t\in[0,T]$}
\end{equation} 
for each $\varepsilon\in(0,\varepsilon_0)$. 
To prove (ii) and (iii), note first that 
$$
\sup_{t\in[0,T]}|u^{\varepsilon}_t|_{L_1}
\leq {\bar R}^{d(p-1)/p}
\sup_{t\in[0,T]}|u^{\varepsilon}_t|_{L_p}<\infty \,\rm{(a.s.)}. 
$$
It is not difficult to see that $\sigma^{(\varepsilon)}_t$, 
$\rho^{(\varepsilon)}_t$, 
$b^{(\varepsilon)}_t$ and $B^{(\varepsilon)}_t$ 
are bounded and Lipschitz continuous in $x\in\bR^d$, 
uniformly in $\omega\in\Omega$,  
$t\in[0,T]$ and $\varepsilon\in(0,\varepsilon_0)$. 
Moreover, for $\varepsilon\in(0,\varepsilon_0)$ 
$$
|\eta^{(\varepsilon)}_t(x,\frz_0)|
\leq K_0\bar{\xi}(\frz_0),
\quad 
|\xi^{(\varepsilon)}_t(x,\frz_1)|
\leq K_0\bar{\xi}(\frz_1), 
$$
$$
|\eta^{(\varepsilon)}_t(x,\frz_0)-\eta^{(\varepsilon)}_t(y,\frz_0)|
\leq \bar{\eta}(\frz_0)|x-y|,
\quad 
|\xi^{(\varepsilon)}_t(x,\frz_1)-\xi^{(\varepsilon)}_t(y,\frz_1)|
\leq \bar{\xi}(\frz_1)|x-y|, 
$$
for all $x,y\in\bR^d$, $\omega\in\Omega$, 
$t\in[0,T]$, $\frz_i\in\frZ_i$, $i=0,1$. 
Hence by Lemma \ref{lemma usup} for 
$\varepsilon\in(0,\varepsilon_0)$ we have 
\begin{equation}                                                                         \label{1.24.3.22}
\E|u^{\varepsilon}_T|_{L_p}^p
+\E\Big(\int_0^T|u_t^{\varepsilon}|_{L_p}^{q}\,dt \Big)^{p/q}
\leq \E|u^{\varepsilon}_T|_{L_p}^p+T^{p/q}\E\sup_{t\in[0,T]}
|u_t^{\varepsilon}|_{L_p}^p\leq N\E|\psi|_{L_p}^p
\end{equation}
for all $q\geq1$ with a constant 
$N=N(d,p,T,K, K_{\xi}, K_{\eta},R,|\bar\eta|_{L_2},|\bar\xi|_{L_2})$. 
By virtue of \eqref{1.24.3.22}  there exists 
a sequence $\varepsilon_n\downarrow0$  
such that
$u^{\varepsilon_n}$ converges weakly 
in $\bL_{p,q}$ to some $\bar u\in\bL_{p,q}$ 
for every integer $q\geq1$ and $u^{\varepsilon_n}_T$ 
converges weakly to some 
$g$ in $\bL_p(\cF_T)$, 
the space of $L_p$-valued $\cF_T$-measurable 
random variables $Z$ with the norm $(\E|Z|^p_{L_p})^{1/p}<\infty$.   
From \eqref{1.24.3.22} we get  
\begin{equation}                                                                                               \label{2.24.3.22}
\E|g|_{L_p}^p+|\bar u|_{\bL_{p,q}}^p\leq N\E|\psi|_{L_p}^p
\quad
\text{for every integer  $q\geq 1$}
\end{equation}
with the constant $N$ in \eqref{1.24.3.22}. 
Taking $\varepsilon_n$ in place of 
$\varepsilon$ in equation \eqref{equdZ*} 
then multiplying both sides of the equation with 
an $\cF_t$-optional bounded process $\phi$ and integrating over 
$\Omega\times [0,T]$ against $P\otimes dt$ we obtain 
\begin{equation}                                                                                       \label{3.24.3.2022}
F(u^{\varepsilon_n})=F(\psi^{\varepsilon_n}) 
+ \sum_{i=1}^5 F_i^{\varepsilon_n}(u^{\varepsilon_n}), 
\end{equation}
where $F$ and $F_i^{\varepsilon}$ are linear functionals over $\bL_{p,q}$, 
defined by 
$$
F(v):=\E\int_0^T\phi_t(v_t,\vp)\,dt,
\quad 
F_1^{\varepsilon} (v)
:=\E\int_0^T\phi_t\int_0^t(v_s,\tilde\cL_s^{\varepsilon}\varphi)\,ds\,dt,
$$
$$
F_2^{\varepsilon}(v)
:=\E\int_0^T\phi_t\int_0^t(v_s,\cM_s^{\varepsilon k}\varphi)\,dV^k_s\,dt,
\quad 
F_3^{\varepsilon}(v):=\E\int_0^T\phi_t\int_0^t\int_{\frZ_0}
(v_s,J_s^{\eta^{\varepsilon}}\varphi)\,\nu_0(d\frz)ds\,dt,
$$
\begin{equation*}                                                                                     \label{equ F operators}
F_4^{\varepsilon}(v)
:=\E\int_0^T\phi_t\int_0^t\int_{\frZ_1}
(v_s,J_s^{\xi^{\varepsilon}}\varphi)\,\nu_1(d\frz)ds\,dt,
\end{equation*}
$$
 F_5^{\varepsilon}(v)
 :=\E\int_0^T\phi_t\int_0^t\int_{\frZ_1}
 (v_s,I_s^{\xi^{\varepsilon}}\varphi)\,\tilde N_1(d\frz,ds) \,dt
$$
for a fixed $\varphi\in C_0^{\infty}$. Define also $F_i$ as $F_i^{\varepsilon}$ 
for $i=1,2,...,5$, 
with $\tilde\cL_s$, $\cM_s^k$, $J_s^{\eta}$, $J^{\xi}_s$ and $I_s^{\xi}$ in place of 
$\tilde\cL^{\varepsilon}_s$, $\cM_s^{\varepsilon k}$, 
$J_s^{\eta^{\varepsilon}}$, $J_s^{\xi^{\varepsilon}}$ 
and $I_s^{\xi^{\varepsilon}}$, respectively.  It is an easy exercise to show that 
 $F$ and $F_i$  and $F_i^{\varepsilon}$, $i=1,2,3,4,5$ are 
continuous linear functionals on $\bL_{p,q}$ for all $q>1$ 
such that 
$$
\lim_{\varepsilon\downarrow0}
\sup_{|v|_{\bL_{p,q}} = 1}|F_i(v)-F_i^{\varepsilon}(v)|=0
\quad
\text{for every $q>1$}.
$$
Since $u^{\varepsilon_n}$ converges weakly to $\bar u$ in $\bL_{p,q}$,  
and $F^{\varepsilon_n}_i$ converges strongly to $F_i$ in $\bL^{\ast}_{p,q}$, 
the dual of $\bL_{p,q}$,  
we get that $F_i^{\varepsilon_n}(u^{\varepsilon_n})$ converges to 
$F_i(\bar u)$  for $i=1,2,3,4,5$. 
Therefore letting $\varepsilon\downarrow0$ in \eqref{3.24.3.2022} we obtain
$$
\E\int_0^T\phi_t(\bar u_t,\vp)\,dt=\E\int_0^T\phi_t(\psi,\vp)\,dt 
+  \E\int_0^T\phi_t\int_0^t(\bar u_{s},\tilde\cL_s\varphi)\,ds\,dt
$$
\begin{equation}
\label{equ weak Zakai}
+ \E\int_0^T\phi_t\int_0^t(\bar u_{s},\cM_s^{k}\varphi)\,dV^k_s\,dt
+ \E\int_0^T\phi_t\int_0^t\int_{\frZ_0}(\bar u_{s},J_s^{\eta}\varphi)\,\nu_0(d\frz)ds\,dt  
\end{equation}
$$
+\E\int_0^T\phi_t\int_0^t\int_{\frZ_1}(\bar u_{s},J_s^{\xi}\varphi)\,\nu_1(d\frz)ds\,dt+\E\int_0^T\phi_t\int_0^t\int_{\frZ_1}(\bar u_{s},I_s^{\xi}\varphi)\,\tilde N_1(d\frz,ds) \,dt.
$$
Since this equation holds for all bounded $\cF_t$-optional processes 
$\phi=(\phi_{t})_{t\in[0,T]}$ and functions $\varphi\in C_0^{\infty}$ 
we conclude that $\bar u$ is a $\bV_p$-solution to \eqref{equdZ}. 
Letting $n\to \infty$ in equation \eqref{equdZ*} 
after taking $\varepsilon_n$ in place of 
$\varepsilon$, $T$ in place of $t$, multiplying both sides of the equation with 
an arbitrary $\cF_T$-measurable bounded random variable 
$\rho$ and taking expectation 
we get 
$$
\E\rho(g,\vp)=\E\rho(\psi,\vp)
+  \E\rho\int_0^T(\bar u_{s},\tilde\cL_s\varphi)\,ds
$$
\begin{equation*}                                                                
+ \E\rho\int_0^T(\bar u_{s},\cM_s^{k}\varphi)\,dV^k_s
+ \E\rho\int_0^t\int_{\frZ_0}(\bar u_{s},J_s^{\eta}\varphi)\,\nu_0(d\frz)ds 
\end{equation*}
$$
+\E\rho\int_0^T\int_{\frZ_1}(\bar u_{s},J_s^{\xi}\varphi)\,\nu_1(d\frz)ds
+\E\rho\int_0^T\int_{\frZ_1}(\bar u_{s},I_s^{\xi}\varphi)\,\tilde N_1(d\frz,ds), 
$$
which implies that almost surely
$$
(g,\vp)=(\psi,\vp)
+ \int_0^T(\bar u_{s},\tilde\cL_s\varphi)\,ds                                                         
+ \int_0^T(\bar u_{s},\cM_s^{k}\varphi)\,dV^k_s
$$
$$
+\int_0^T\int_{\frZ_0}(\bar u_{s},J_s^{\eta}\varphi)\,\nu_0(d\frz)ds 
+\int_0^T\int_{\frZ_1}(\bar u_{s},J_s^{\xi}\varphi)\,\nu_1(d\frz)ds
+\int_0^T\int_{\frZ_1}(\bar u_{s},I_s^{\xi}\varphi)\,\tilde N_1(d\frz,ds).  
$$
Letting $q\to\infty$ in \eqref{2.24.3.22} we get 
$$
\E\essup_{t\in[0,T]}|\bar u_t|_{L_p}^p\leq N\E|\psi|_{L_p}^p<\infty.   
$$
Consequently, by virtue of Lemma \ref{lemma weakly cadlag} we get the existence 
of a $P\otimes dt$-modification $u$ of $\bar u$, which is an $L_p$-solution 
to \eqref{equdZ}, and hence  
\begin{equation}                                                                             \label{30.9.21.2}
\E\sup_{t\in[0,T]}|u_t|_{L_p}^p\leq N\E|\psi|_{L_p}^p. 
\end{equation}
By \eqref{4.24.3.22} almost surely $u_t=0$ for $dx$-almost 
every $x\in\{x\in\bR^d: |x|\geq\bar R\}$,    
for all $t\in[0,T]$, which due to \eqref{30.9.21.2} 
by H\"older's inequality implies 
 $$
\E\sup_{t\in [0,T]}|u_t|_{L_1}\leq N\bar R^{d(p-1)/p}\E|\psi|_{L_p}<\infty.
$$ 
Hence by \eqref{eusupbecsles} in Lemma \ref{lemma usup}
 the uniqueness 
of the $L_p$-solution follows, 
which completes the proof of the lemma. 
\end{proof}
\begin{corollary}                                                                               \label{corollary 1.3.4.22}
Let Assumptions \ref{assumption SDE},  
\ref{assumption p} and \ref{assumption estimates} 
hold with  $K_1=0$. 
Assume, moreover, that  the``support condition" \eqref{supp_condition}
holds for some $R>0$. 
Then for every $p\geq2$ there is a linear operator $\bS$ defined 
on $\bL_p$ 
such that $\bS\psi$ admits a $P\otimes dt$-modification 
$u=(u_t)_{t\in[0,T]}$ which is an $L_p$-solution 
to equation \eqref{equdZ} for every $\psi\in\bL_p$, with initial condition $u_0=\psi$,  
and 
\begin{equation}                                                                                                   \label{30.9.21.5}
\E\sup_{t\in [0,T]}|u_t|_{L_p}^p\leq N\E|\psi|_{L_p}^p
\end{equation}
with a constant 
$N=
N(d,p,T,K, K_{\xi}, K_{\eta}, L,\lambda,|\bar{\xi}|_{L_2},|\bar{\eta}|_{L_2})$. 
Moreover, if $\psi\in\bL_p$ such that almost surely $\psi(x)=0$ for $|x|\geq R$,  
then there is a constant 
$\bar{R}=\bar{R}(R,K,K_0, K_{\xi}, K_{\eta})$ such that 
almost surely $u_t(x)=0$ for $dx$-a.e. $x\in\{x\in\bR^d:|x|\geq \bar R\}$ 
for all $t\in[0,T]$.
\end{corollary}
\begin{proof}
If $p$ is an even integer, then the corollary follows from 
Lemma \ref{lemma compact}. 
Assume $p$ is not an even integer. Then let $p_0$ be 
the greatest even integer such that $p_0\leq p$ and let $p_1$ 
be the smallest even integer such that $p\leq p_1$. 
By Lemma \ref{lemma compact}
there are linear operators $\bS$ and $\bS_T$ 
defined on $\bB_0$ such that 
$\bS\psi:=(u_t)_{t\in[0,T]}$ is the 
unique $L_{p_i}$-solution of equation \eqref{equdZ} with initial condition 
$u_0=\psi\in\bB_0$ and $\bS_T\psi=u_T$. 
for $i=0,1$. 
Moreover, by \eqref{30.9.21.5} 
we have 
$$
|\bS_T\psi|_{\bL_{p_i}}+|\bS\psi|_{\bL_{p_i,q}}
\leq N|\psi|_{\bL_{p_i}}\quad\text{for $i=0,1$}
$$
for every $q\in[1,\infty)$ with a constant 
$N=
N(d,p,T,K, K_{\xi}, K_{\eta}, L,\lambda,|\bar{\xi}|_{L_2},|\bar{\eta}|_{L_2})$. 
Hence by a well-known generalization of the Riesz-Thorin interpolation theorem 
we have 
\begin{equation}                                                                       \label{2.3.4.22}
|\bS_T\psi|_{\bL_{p}}\leq N|\psi|_{\bL_{p}},
\quad
 |\bS\psi|_{\bL_{p,q}}\leq N|\psi|_{\bL_{p}}
\quad
\text{for every $q\in[1,\infty)$,}
\end{equation}
for $\psi\in\bB_0$ with a constant 
$N=
N(d,p,T,K, K_{\xi}, K_{\eta}, L,\lambda,|\bar{\xi}|_{L_2},|\bar{\eta}|_{L_2})$. 
Assume $\psi\in\bL_p$. Then there is a sequence 
$(\psi^n)_{n=1}^{\infty}\subset \bB_0$ such that 
$\psi^n\to \psi$ in $\bL_p$ and $u^n=\bS\psi^n$  
has a $P\otimes dt$-modification, again denoted by 
$u^{n}=(u_t^{n})_{t\in[0,T]}$ which is an $L_p$-solution for every $n$ 
with initial condition $u^n_0=\psi^n$. 
In particular, for each $\varphi\in C_0^{\infty}$ 
almost surely 
\begin{align}                                                                       
(u_t^{n},\varphi)=&(\psi^{n},\varphi)+
\int_0^t(u_s^{n},\tilde\cL_s\varphi)\,ds
+\int_0^t(u_s^{n},\cM_s^{k}\varphi)\,dV^k_s
+\int_0^t\int_{\frZ_0}
(u_s^n,J_s^{\eta}\varphi)\,\nu_0(d\frz)\,ds                        \nonumber\\
&+\int_0^t\int_{\frZ_1}
(u_s^{n},J_s^{\xi}\varphi)\,\nu_1(d\frz)\,ds
+\int_0^t\int_{\frZ_1}(u_s^n,I_s^{\xi}\varphi)\,\tilde N_1(d\frz,ds), \label{3.3.4.22}  
\end{align}
holds for all $t\in[0,T]$. By virtue of \eqref{2.3.4.22} $u^n$ converges in $\bL_{p,q}$ 
to some $\bar u\in \bL_{p,q}$ for every $q>1$, and $u_T^n$ converges in $\bL_p$ 
to some $g\in\bL_p$. Hence, letting $n\to\infty$ 
in equation \eqref{3.3.4.22} (after multiplying both sides 
of it  with any bounded $\cF_t$-optional process $\phi=(\phi_t)_{t\in[0,T]}$ 
and integrating it over $\Omega\times[0,T]$ against $P\otimes dt$) we can 
see that $\bar u$ is a $\bV_p$-solution such that \eqref{2.3.4.22} holds. 
Letting $n\to\infty$ in equation \eqref{3.3.4.22} with $t:=T$ (after 
multiplying both sides with an arbitrary $\cF_T$-measurable bounded 
random variable $\rho$ and 
taking expectation) we get 
$$                                                                    
\E\rho(g,\varphi)=\E\rho(\psi,\varphi)+
\E\rho\int_0^T(\bar{u}_s,\tilde\cL_s\varphi)\,ds
+\E\rho\int_0^T(\bar{u}_s,\cM_s^{k}\varphi)\,dV^k_s
$$
$$
+\E\rho\int_0^T\int_{\frZ_0}
(\bar{u}_s,J_s^{\eta}\varphi)\,\nu_0(d\frz)\,ds                     
+\E\rho\int_0^T\int_{\frZ_1}
(\bar{u}_s,J_s^{\xi}\varphi)\,\nu_1(d\frz)\,ds
$$
$$
+\E\rho\int_0^T\int_{\frZ_1}(\bar{u}_s,I_s^{\xi}\varphi)\,\tilde N_1(d\frz,ds), 
$$
which implies
$$                                                                    
(g,\varphi)=(\psi,\varphi)+
\int_0^T(\bar{u}_s,\tilde\cL_s\varphi)\,ds
+\int_0^T(\bar{u}_s,\cM_s^{k}\varphi)\,dV^k_s
$$
$$
+\int_0^T\int_{\frZ_0}
(\bar{u}_s,J_s^{\eta}\varphi)\,\nu_0(d\frz)\,ds                     
+\int_0^T\int_{\frZ_1}
(\bar{u}_s,J_s^{\xi}\varphi)\,\nu_1(d\frz)\,ds
$$
$$
+\int_0^T\int_{\frZ_1}(\bar{u}_s,I_s^{\xi}\varphi)\,\tilde N_1(d\frz,ds) \quad(a.s.).  
$$
Letting $q\to\infty$ in \eqref{2.3.4.22} we get 
$$
\E\esssup_{t\in[0,T]}|\bar{u}_t|^p_{L_p}\leq N|\psi|_{L_p}^p. 
$$
Hence by virtue of Lemma \ref{lemma weakly cadlag} the process $\bar u$ has 
a $P\otimes dt$ modification $u=(u_t)_{t\in[0,T]}$ which is an $L_p$-solution 
to equation \eqref{equdZ} and \eqref{30.9.21.5} holds. Finally the last statement 
of the corollary about the compact support of $u$ can be proved in the same way 
as it was shown for $u^{\varepsilon}$ in the proof of Lemma \ref{lemma compact}.
\end{proof}

\mysection{Proof of Theorem \ref{theorem 1}}

To prove Theorem \ref{theorem 1} we want to show that for $p\geq2$ 
equation \eqref{equdZ} has an $L_p$-solution which we can identify as the 
unnormalised conditional density of the conditional distribution of $X_t$ given 
the observation $\{Y_s:s\leq t\}$. To this end we need some lemmas. 
To formulate 
the first one, we recall that 
$\bW^m_p$ denotes the space of $W^m_p$-valued $\cF_0$-measurable random variables 
$Z$ such that 
$$
|Z|^p_{\bW^m_p}=\E|Z|^p_{W^m_p}<\infty.
$$
\begin{lemma}                                                                              \label{lemma 13.10.2021}
Let $(X,Y)$ be an $\cF_0$-measurable $\bR^{d+d'}$-valued 
random variable such that the conditional density 
$\pi=P(X\in dx|Y)/dx$ exists. Assume $(\Omega,\cF_0, P)$ is ``rich enough"
to carry an $\bR^d$-valued random variable $\zeta$ 
which is independent of $(X,Y)$ 
and has a smooth probability density $g$ supported 
in the unit ball centred at the origin.
Then there exists a sequence of $\cF_0$-measurable random variables 
$(X_n)_{n=1}^\infty$ such that the conditional density 
$\pi_n=P(X_n\in dx|Y)/dx$ exists, almost surely $\pi_n(x)=0$ 
for $|x|\geq n+1$ for each $n$, 
$$
\lim_{n\to\infty}X_n=X\quad \text{for every $\omega\in\Omega$}, 
$$
and, if $\pi\in\bW^m_p$ for some $p\geq1$, $m\geq0$, 
then $\pi_n\in\bW^m_p$ for every $n\geq1$, and  
$$
\lim_{n\to\infty}|\pi_n-\pi|_{\bW^m_p}=0.
$$
Moreover, for every $n\geq1$ we have 
$$
\E|X_n|^q\leq N(1+\E|X|^q)
\quad 
\text{for every $q\in(0,\infty)$}
$$
with a constant $N$ depending only on $q$.
\end{lemma}
\begin{proof}
For $\varepsilon\in(0,1)$ define
$$
X^{\varepsilon} _{k}:=X{\bf1}_{|X|\leq k}+\varepsilon\zeta
\quad
\text{for integers $n\geq1$}.    
$$
Let $g_{\varepsilon}$ denote the density function of $\varepsilon\zeta$,   
and let  $\mu_{k}$ be the regular conditional distribution 
of $Z_{k}:=X{\bf1}_{|X|\leq k}$ given $Y$. Then 
$$
\mu_{k}^{(\varepsilon)}(x)=\int_{\bR^d}g_\varepsilon(x-y)\,\mu_{k}(dy)
\quad
\text{and}
\quad 
\pi^{(\varepsilon)}(x)=\int_{\bR^d}g_{\varepsilon}(x-y)\pi(y)dy,
\quad
x\in\bR^d,
$$
are the conditional density functions of 
$X^{\varepsilon }_{k}$ and $X+\varepsilon\zeta$, 
given $Y$, respectively.  Clearly, if $\pi\in\bW^m_p$, then 
$\mu^{(\varepsilon)}_{k}$ and $\pi^{\ep}$ belong to $\bW^m_p$ for every $k$ 
and $\varepsilon$. Moreover, 
by Fubini's theorem, for each multi-index $\alpha=(\alpha_1,\dots,\alpha_d)$, such that $0\leq |\alpha|\leq m$ we have
\begin{equation}                                                                           \label{26.10.21.1}
|D^\alpha\mu^{(\varepsilon)}_{k}-D^\alpha\pi^{\ep}|^p_{\bL_p}
=\E\int_{\bR^d}
\Big|\int_{\bR^d}D^\alpha g_\varepsilon(x-y)\mu_{k}(dy)
 - \int_{\bR^d}D^\alpha g_\varepsilon(x-y)\pi(y)dy\Big|^pdx
\end{equation}
\begin{equation*}
=\int_{\bR^d}
\E\Big|\int_{\bR^d}D^\alpha g_\varepsilon(x-y)\mu_{k}(dy)
 - \int_{\bR^d}D^\alpha g_\varepsilon(x-y)\pi(y)dy\Big|^pdx
\end{equation*}
\begin{equation*}
=\int_{\bR^d}
\E\big|\E(D^\alpha g_\varepsilon(x-Z_k)-D^\alpha g_\varepsilon(x-X)|Y)\big|^p\,dx
 \end{equation*}
 \begin{equation*}
 \leq 
 \int_{\bR^d}\E\big|D^\alpha g_\varepsilon(x-Z_{k})-D^\alpha g_\varepsilon(x-X)\big|^p\,dx 
=
\E \int_{\bR^d}\big|D^\alpha g_\varepsilon(x-Z_{k})-D^\alpha g_\varepsilon(x-X)\big|^p\,dx, 
\end{equation*}
where the inequality is obtained by an application of Jensen's inequality.  
Clearly, for every $0\leq |\alpha|\leq m$,
$$
 \int_{\bR^d}\big|D^\alpha g_\varepsilon(x-Z_{k})-D^\alpha g_\varepsilon(x-X)\big|^p\,dx
 \leq 2^{p} |g_{\varepsilon}|^p_{W^m_p}<\infty 
 \quad
 \text{for every $\omega\in\Omega$ and $k\geq1$.} 
$$
Hence by Lebesgue's theorem on dominated convergence, for each $0\leq |\alpha|\leq m$,
$$
\lim_{k\to\infty}
\E\int_{\bR^d}\big|D^\alpha g_\varepsilon(x-Z_k)-D^\alpha g_\varepsilon(x-X)\big|^p\,dx
=\E\lim_{k\to\infty}\int_{\bR^d}
\big|D^\alpha g_\varepsilon(x-Z_{k})-D^\alpha g_\varepsilon(x-X)\big|^p\,dx=0.  
$$
Consequently, by virtue of \eqref{26.10.21.1} we have 
$\lim_{k\to\infty}
|\mu^{(\varepsilon)}_{k}-\pi^{(\varepsilon)}|_{\bW^m_p}=0$ 
for every 
$\varepsilon\in(0,1)$. Since almost surely 
$|\pi^{(\varepsilon)}-\pi|_{W^m_p}\to0$  as 
$\varepsilon\downarrow0$, and 
$|\pi^{(\varepsilon)}-\pi|_{W^m_p}\leq 2|\pi|_{W^m_p}$ for every $\omega\in\Omega$,  
we have $\lim_{\varepsilon\downarrow0}|\pi^{\ep}-\pi|_{\bW^m_p}=0$ 
by Lebesgue's theorem on 
dominated convergence.  Hence there is a sequence of 
positive integers $k_n\uparrow\infty$ such that 
for $\pi_n:=\mu_{k_n}^{ (1/n)}$ we have $\lim_{n\to\infty}|\pi_n-\pi|_{\bW^m_p}=0$. 
Clearly, for $X_n:=X^{\varepsilon_n}_{ k_n}$ with $\varepsilon_n=1/n$ 
we have $\lim_{n\to\infty}X_n=X$ 
for every $\omega\in\Omega$. Moreover, for every integer $n\geq1$ 
$$
\E|X_{n}|^q
\leq N\big(\E|X{\bf1}_{|X|\leq k_n}|^q + \varepsilon_n^q\E|\zeta|^q\big)
\leq N(\E|X|^q+1) 
\quad
\text{for  $q\in(0,\infty)$}   
$$
with a constant $N=N(q)$, which completes the proof of the lemma. 
\end{proof}

To formulate our next lemma let $\chi$ be a smooth function on $\bR$ 
such that $\chi(r)=1$ for $r\in[-1,1]$, $\chi(r)=0$ for $|r|\geq2$, 
$\chi(r)\in[0,1]$ and $\chi'(r)=\tfrac{d}{dr}\chi(r)\in[-2,2]$ for all $r\in\bR$. 
\begin{lemma}                                                                          \label{lemma 8.7}
Let $b=(b^i)$ be an $\bR^d$-valued function on $\bR^m$ such that for a constant $L$ 
\begin{equation}                                                                        \label{8.5.2}
|b(v)-b(z)|\leq L|v-z| \quad\text{for all $v,z\in\bR^m$}. 
\end{equation}                                                                           
Then for $b_n(z)=\chi(|z|/n)b(z)$, $z\in\bR^m$,  
for integers $n\geq1$ we have 
\begin{equation}                                                                          \label{8.5.3}
|b_n(z)|\leq 2nL+|b(0)|,
\quad 
|b_n(v)-b_n(z)|\leq (5L+2|b(0)|)|v-z| \quad\text{for all $v,z\in\bR^m$}. 
\end{equation}
\end{lemma}

\begin{proof}
We leave the proof as an easy exercise for the reader.
\end{proof}

We will  truncate the functions $\xi$ and $\eta$ 
in equation \eqref{system_1} by the help 
of the following lemma, in which for each 
fixed $R>0$ and $\epsilon>0$ we use a
function 
$\kappa^R_\epsilon$ defined on $\bR^d$ by 
\begin{equation}                                                                             \label{Sandy function}
\kappa^R_\epsilon(x)=\int_{\bR^d}\phi^R_{\epsilon}(x-y)k(y)\,dy, 
\quad
\phi^R_{\epsilon}(x)=\begin{cases}
1,& |x|\leq R+1,\\
1+\epsilon\log\big(\tfrac{R+1}{|x|}\big),& R+1<|x|< (R+1)e^{1/\epsilon},\\
0,& |x|\geq (R+1)e^{1/\epsilon}, 
\end{cases}
\end{equation}
where $k$ is a nonnegative $C^{\infty}$ 
mapping on $\bR^d$ with support in $\{x\in\bR^d:|x|\leq 1\}$. Notice that 
$\kappa^R_\epsilon\in C^{\infty}_0$ for each 
$R,\epsilon>0$,  such that  
if $x,y\in\bR^d$ and $|y|\leq |x|$, then 
$$
|\phi^{R}_{\epsilon}(x)-\phi^{R}_{\epsilon}(y)|\leq\frac{\epsilon|x-y|}{\max(R,|y|)},  
$$
and hence 
\begin{equation}                                                                         \label{Sandy}
|\kappa^R_\epsilon(x)-\kappa^R_\epsilon(y)|
\leq \int_{\bR^d}|\phi^{R}_{\epsilon}(x-u)-\phi^{R}_{\epsilon}(y-u)|k(u)\,du
\leq \frac{\epsilon|x-y|}{\max(R,|y|-1)}.
\end{equation}

\begin{lemma}                                                                               \label{lemma Sandy}
Let $\xi:\bR^d\mapsto\bR^d$ be such that for a constant $L\geq 1$ 
and for every $\theta\in [0,1]$ the function 
$\tau_\theta(x)=x+\theta\xi(x)$ is $L$-biLipschitz, i.e.,
\begin{equation}
\label{tth}L^{-1}|x-y|\leq|\tau_\theta(x)-\tau_\theta(y)|\leq L|x-y|
\end{equation}
for all $x,y\in\R^d$. Then for any $M>L$ and any $R>0$ there is an $\epsilon=\epsilon(L,M,R,|\xi(0)|)>0$ such that 
with $\kappa^R:=\kappa^R_\epsilon$ the function $\xi^R:=\kappa^R\xi$ vanishes for $|x|\geq \bar R$ for a constant 
$\bar{R}=\bar R(L,M,R,|\xi(0)|)>R$, $|\xi^R|$ is bounded by a constant 
$N=N(L,M,R, |\xi(0)|)$,  and for every $\theta\in [0,1]$ the mapping
$$
\tau^R_\theta (x) = x + \theta\xi^R(x) , \quad x\in\bR^d
$$
is $M$-biLipschitz.
\end{lemma}

\begin{proof}
To show $\tau^R_\theta$ is $M$-biLipschitz, we first note that if $x,y\in\R^d$ 
with $|x|\geq|y|$ then $\tau^R_\theta(x)-\tau^R_\theta(y)=A+B$ where
$A=\tau_{\theta\kappa^R(x)}(x)-\tau_{\theta\kappa^R(x)}(y)$  
and $B=\xi(y)(\kappa^R(x)-\kappa^R(y))$. 
The biLipschitz hypothesis (\ref{tth}), with $\theta$ replaced
by $\theta\kappa^R(x)$, implies $L^{-1}|x-y|\leq|A|\leq L|x-y|$. 
Due to \eqref{Sandy} and that $\xi$ has linear growth, 
we can choose a sufficiently small 
$\epsilon=\epsilon(L,M,R,|\xi(0)|)$ to get 
$|B|<(L^{-1}-M^{-1})|x-y|$ and hence
\[M^{-1}|x-y|\leq|\tau^R_\theta(x)-\tau^R_\theta(y)|\leq M|x-y|\]
as required.
Finally the boundedness of $|\xi^R|$ follows from the fact that it vanishes for $|x|>Re^{1/\epsilon}$ and that $\xi$ has linear growth.
\end{proof}

\begin{remark}
Note that if $\tau$ is a continuously differentiable $L$-biLipschitz function on $\bR^d$ then 
$$
L^{-d}\leq|\det(D\tau(x))|\leq L^d\quad\text{for $x\in\bR^d$}.
$$
\end{remark}
\begin{proof} 
This remark must be well-known, since for $d=1$ it is obvious, 
and for $d>1$ it can be easily shown by using the singular value decomposition 
for the matrices $D\tau$, $D\tau^{-1}$, or by applying Hadamard's inequality 
to their determinants. 
\end{proof}

\begin{proof}[Proof of Theorem \ref{theorem 1}.] 
The proof is structured into three steps. 
First we prove the theorem for the case where $p=2$. 
As second step we prove the results for all $p\geq 2$ 
for compactly supported coefficients and compactly supported 
initial conditional densities. The third step then involves an approximation procedure 
to obtain the desired results for coefficients and 
initial conditional densities with unbounded support.
\newline
\textbf{Step I:} Let Assumptions \ref{assumption SDE},  
\ref{assumption p} 
and \ref{assumption estimates} hold. 
Then by Theorem \ref{theorem Z1}, the process $(P_t)_{t\in[0,T]}$ 
of the regular conditional distribution $P_t$ of $X_t$ given $\cF^Y_t$, 
and $\mu=(\mu_t)_{t\in[0,T]}=(P_t(^o\!\gamma_t)^{-1})_{t\in[0,T]}$,  
the ``unnormalised" (regular) conditional 
distribution process, are measure-valued weakly cadlag 
processes, and $\mu$ is a measure-valued solution to equation \eqref{eqZ1}. 
(Recall that $(^o\!\gamma_t)_{t\in [0,T]}$ 
is the positive normalising process from Remark \ref{remark gamma}.)
Assume that $u_0:=P(X_0\in dx|Y_0)/dx$ exists almost surely 
such that $\E|u_0|_{L_p}^p<\infty$ 
for $p=2$.
In order to apply Lemma \ref{lemma eu1} if $K_1\neq 0$, 
we need to verify that
\begin{equation}                                                                                        \label{25.1.22.1}
G(\mu) = \sup_{t\in [0,T]}
\int_{\bR^d}|x|^2\,\mu_t(dx)<\infty \quad\text{almost surely.}
\end{equation}
For integers $k\geq 1$ let $\Omega_k:= \{|Y_0|\leq k\}\in \cF_0^Y$. 
Then $\Omega_k\uparrow\Omega$ as $k\to\infty$.
Taking $r>2$ from Assumption \ref{assumption nu}, by  
Doob's inequality, and by Jensen's inequality for optional projections we get
$$
\E\sup_{t\in [0,T]}
\big(\E(|X_t|^2{\bf1}_{\Omega_k}|\cF_t^Y)\big)^{r/2}
\leq \E\sup_{t\in [0,T]}
\big(\E(\sup_{s\in [0,T]}|X_s|^2{\bf1}_{\Omega_k}|\cF_t^Y)\big)^{r/2}
$$
$$
\leq N \E\big(\E(\sup_{s\in [0,T]}|X_s|^{2}{\bf1}_{\Omega_k}|\cF_T^Y)\big)^{r/2}
\leq N \E\sup_{s\in [0,T]}|X_s|^{r}{\bf1}_{\Omega_k},
$$
for all $k$ with a constant $N$ depending only on $r$.
Thus, by Fubini's theorem and H\"older's inequality, if $K_1\neq 0$, for all $k$
we have 
$$
G_k(\mu):=\E\sup_{t\in [0,T]}\int_{\bR^d}|x|^2\mu_t(dx){\bf1}_{\Omega_k} 
$$
$$
= \E\sup_{t\in [0,T]}\E(|X_t|^{2}|\cF_t^Y)({^o\!\gamma_t})^{-1}{\bf1}_{\Omega_k}
= \E\sup_{t\in [0,T]}\E(|X_t|^{2}{\bf1}_{\Omega_k}|\cF_t^Y)({^o\!\gamma})_t^{-1}
$$
$$
\leq \E\sup_{t\in [0,T]}\E(|X_t|^{2}{\bf1}_{\Omega_k}|\cF_t^Y)
\sup_{t\in [0,T]}({^o\!\gamma})_t^{-1}
\leq \big(\E\sup_{t\in [0,T]}\big(\E(|X_t|^2{\bf1}_{\Omega_k}|\cF_t^Y)\big)^{r/2} \big)^{2/r}
\big(\E \sup_{t\in [0,T]}({^o\!\gamma}_t)^{-r'}\big)^{1/r'}
$$
$$
\leq N \big( \E\sup_{t\in [0,T]}|X_t|^{r}{\bf1}_{\Omega_k}\big)^{2/r},
$$
where $2/r + 1/r'=1$, $N=N(r,d,C)$ is a constant,  and we use that 
by Jensen's inequality 
for optional projections and the boundedness of $|B|$ 
\begin{equation}
\label{24.10.21.3}
\E\sup_{t\in [0,T]}({^o\!\gamma}_t)^{-r'}
\leq \E\sup_{t\in [0,T]}\gamma_t^{-r'}:=C<\infty
\end{equation}
with a constant $C$ only depending on the bound in magnitude of $|B|$ and $r$. 
Hence, using 
\eqref{bound_Z} with $q=r$ we have 
$$
G_k(\mu)\leq N\big(1+\E\sup_{t\in [0,T]}|X_t|^{r}{\bf1}_{\Omega_k}\big)
\leq N'\big(k^{r}+\E|X_0|^{r} \big)<\infty,
$$
for constant $N=N(r,d,C)$ and 
$N' = N'(d,d',r,K,K_0,K_1,K_{\xi}, 
K_{\eta},   T,|\bar\xi|_{L_2},|\bar\eta|_{L_2})$. 
Since for all $k\geq 1$ we have that $G_k(\mu) <\infty$ 
we can conclude that \eqref{25.1.22.1} holds.
Hence, 
by Lemma \ref{lemma eu1}, almost surely  $d\mu_t/dx$ exists, and 
there is 
an $L_2$-valued weakly cadlag stochastic process $(u_t)_{t\in [0,T]}$ 
such that almost surely $u_t=d\mu_t/dx$ for all $t\in[0,T]$ and  
\begin{equation}
\label{2.11.2021.1}
\E\sup_{t\in [0,T]}|u_t|_{L_2}^2\leq N\E|\pi_0|_{L_2}^2
\end{equation}
for every $T$ with a constant 
$N=N(d, d',K,K_\xi, K_\eta, L, T,|\bar\xi|_{L_2}, |\bar\eta|_{L_2},\lambda)$. 
Thus $\pi_t=dP_t/dx=u_t{^o\!\gamma_t}$, $t\in [0,T]$, 
is an $L_2$-valued weakly cadlag 
process, which proves Theorem \ref{theorem 1} for $p=2$.\newline
\textbf{Step II.} Let the assumptions of Theorem \ref{theorem 1} 
hold with $K_1=0$ 
in Assumption \ref{assumption SDE}. 
Assume 
that $\pi_0=P(X_0\in dx|Y_0)/dx\in \bL_p$ for some $p>2$, such that 
almost surely $u_0(x)=0$  for $|x|\geq R$ for a constant $R$. 
Assume moreover, 
that the support condition \eqref{supp_condition} holds. 
Then by Corollary \ref{corollary 1.3.4.22} there is 
an $L_p$-solution $(v_t)_{t\in [0,T]}$ to \eqref{equdZ} with initial condition 
$v_0=\pi_0$ such that 
\begin{equation}                                                                             \label{24.10.21.1}
\E\sup_{t\in[0,T]}|v_t|^p_{L_p}\leq N\E|\psi|^p_{L_p}
\end{equation}
with a constant 
$N=N(d, d',K, L, K_{\xi}, K_{\eta}, T, p,\lambda,|\bar\xi|_{L_2}, |\bar\eta|_{L_2})$, 
and almost surely
$$
v_t(x)=0 \quad\text{for $dx$-a.e. $x\in\{x\in\bR^d:|x|\geq\bar R\}$ for all $t\in[0,T]$} 
$$
with a constant $\bar R=\bar R(R,K,K_0,K_\xi,K_\eta)$. 
Hence $(v_t)_{t\in[0,T]}$ is also an $L_2$-solution to equation \eqref{equdZ}, and clearly, 
$$
\sup_{t\in [0,T]}|v_t|_{L_1}
\leq \bar R^{d(p-1)/p}\sup_{t\in[0,T]}|v_t|_{L_p}<\infty. 
$$
Since in particular $\E|\pi_0|^2_{L_2}<\infty$, by Step I there is an $L_2$-solution 
$(u_t)_{t\in[0,T]}$ to equation \eqref{equdZ} such that almost surely 
$u_t=d\mu_t/dx$ for all $t\in[0,T]$, 
where $\mu_t=P_t(^o\!\,\gamma_t)^{-1}$ 
is the unnormalised (regular) conditional distribution of 
$X_t$ given $\cF^Y_t$. Clearly, 
$$
\sup_{t\in[0,T]}|u_t|_{L_1}=\sup_{t\in[0,T]}(^o\!\gamma_t)^{-1}<\infty\,(\text{a.s.}). 
$$
Hence by virtue of \eqref{eusupbecsles} in Lemma \ref{lemma usup} we obtain 
$\sup_{t\in[0,T]}|u_t-v_t|_{L_2}=0$ (a.s.), which completes 
the proof of Theorem \ref{theorem 1} under the additional assumptions of Step II. \newline
\textbf{Step III.} 
Finally, we dispense with the assumption that the 
coefficients and the initial condition are compactly supported, and that $K_1=0$  
in Assumption \ref{assumption SDE}. 
Define the functions $b^n=(b^{ni}(t,z))$, $B^n = (B^{nj}(t,z))$, 
$\sigma^n=(\sigma^{nij}(t,z))$, 
$\eta^n=(\eta^{ni}(t,z,\frz_0))$ and $\xi^n=(\xi^{ni}(t,z,\frz_1))$ by 
$$
(b^n, B^n,\sigma^n, \rho^n) = (b, B,\sigma, \rho)\chi_n,\quad 
(\eta^n,\xi^n)=(\eta,\xi)\bar\chi_n
$$
for every integer $n\geq1$, 
where $\chi_n(z)=\chi(|z|/n)$ and 
$\bar\chi_n(x,y)=\kappa^n(x)\chi(|y|/n)$, with $\chi$ 
defined before Lemma \ref{lemma 8.7} 
and with $\kappa^n$ stemming from Lemma \ref{lemma Sandy} 
applied to $\xi$ and $\eta$ as functions of $x\in\bR^d$.
By Lemma \ref{lemma 8.7},
Assumptions \ref{assumption SDE}  
and \ref{assumption p} hold for $b^n$, $B^n$, $\sigma^n$, 
$\rho^n$, $\eta^n$ and $\xi^n$, in place of 
 $b$, $\sigma$, $\rho$, $\eta$ and $\xi$, respectively, with $K_1=0$ 
 and with appropriate constants $K_0'=K_0'(n,K,K_0,K_1,K_\eta,K_\xi,L))$ and
 $L'=L'(K,K_0, K_1, L,K_{\xi},K_{\eta})$ in place of $K_0$ and $L$. Moreover, by Lemma \ref{lemma Sandy}, Assumption \ref{assumption estimates} is satisfied  with a constant $\lambda'=\lambda'(K_0,K_1,K_\xi,K_\eta,\lambda)$ in place of $\lambda$.
Since 
$\pi_0=P(X_0\in dx|Y_0)/dx\in\bL_p$ for $p>2$ 
by assumption (the case $p=2$ was proved in Step I) and clearly  $\pi_0\in\bL_1$, 
by H\"older's inequality we have 
$$
|\pi_0|_{\bL_2}\leq |\pi_0|^{1-\theta}_{\bL_1}|\pi_0|^{\theta}_{\bL_p}<\infty
\quad \text{with $\theta=\tfrac{p}{2(p-1)}\in(0,1)$}. 
$$
Thus by Lemma \ref{lemma 13.10.2021} 
 there exists a sequence $(X_0^n)_{n=1}^{\infty}$ 
 of $\cF_0$-measurable random variables 
 such that the conditional density 
 $\pi_0^n = P(X_0^n\in dx|\cF_0^Y)/dx$ exists, 
 $\pi_0^n(x)=0$ for $|x|\geq n+1$ for 
 every $n$, $\lim_{n\to\infty}X_0^n=X_0$ 
 for every $\omega\in\Omega$, 
 \begin{equation}                                                                          \label{1.5.4.22}
 \lim_{n\to\infty}|\pi_0^n-\pi_0|_{\bL_r}=0
 \quad 
 \text{for $r=2,p$}, 
 \end{equation} 
 and 
\begin{equation*}
\E|X^n_0|^{q}\leq N(1+\E|X_0|^{q}) \quad\text{for any $q>0$} 
\end{equation*}
with a constant $N=N(q)$. 
Let $(X^n_t,Y^n_t)_{t\in [0,T]}$ denote the solution of equation \eqref{system_1} 
with initial value $(X_0^n,Y_0)$ and with 
$b^n$, $\sigma^n$, $\rho^n$, $\xi^n$, $\eta^n$ and $B^n$ in place 
of $b$, $\sigma$, $\rho$, $\xi$, $\eta$ and $B$.  Define the random fields,
$$
b^n_t(x)=b^n(t,x,Y^n_{t-}), 
\quad
\sigma^n_t(x)=\sigma^n(t,x,Y^n_{t-}), 
\quad
\rho^n_t(x)=\sigma^n(t,x,Y^n_{t-}),
\quad
B^n_t(x)=B^n(t,x,Y^n_{t-})
$$
\begin{equation}
\label{truncated coeffs}
\eta^n_t(x,\frz_0)=\eta^n(t,x,Y^n_{t-}, \frz_0),
\quad
\xi^n_t(x,\frz_1)=\xi^n(t,x,Y^n_{t-},\frz_1), 
\quad
\beta^n_t=B^n(t,X^n_t,Y_{t-}^n)
\end{equation}
for $\omega\in\Omega$, $t\geq0$, $x\in\bR^d$, $\frz_i\in\frZ_i$, $i=0,1$. 
Consider the equation 
\begin{align}                                                                       
du_t^{n}=&\tilde\cL_t^{n\ast}u_t^{n}\,dt
+\cM_t^{n k\ast}u_t^{n}\,dV^k_t
+\int_{\frZ_0}J_t^{\eta^{n}\ast}u_t^{n}\,\nu_0(d\frz)dt                        \nonumber\\
&+\int_{\frZ_1}J_t^{\xi^{n}*}u_t^{n}\,\nu_1(d\frz)dt
+\int_{\frZ_1}I_t^{\xi^{n}*}u_t^{n}\,\tilde N_1(d\frz,dt),
\quad
\text{with $u_0^{n}=\pi_0^{n}$,}                                                           \label{2.5.4.22}
\end{align}
where  for each fixed $n$ and $k=1,2,...,d'$
$$
\tilde\cL_t^{n}:=a_t^{n ij}D_{ij}
+b_t^{ni}D_i+
\beta^{nk}_t\rho_t^{nik}D_i
+\beta_t^{nk}B_t^{nk}, 
\quad
\cM^{nk}_t:=\rho_t^{nik}D_i+B^{nk}_t, 
$$
$$
a_t^{n ij}
:=\tfrac{1}{2}\sum_k\sigma_t^{nik}\sigma_t^{njk}
+\tfrac{1}{2}\sum_k \rho_t^{nik}\rho_t^{njk}, 
\quad
\beta^n_t:=B^n(t,X^n_{t-},Y^n_{t-}),  
\quad i,j=1,2,...,d, 
$$
the operators $J_t^{\eta^{n}}$ and $J_t^{\xi^{n}}$ 
are defined as 
$J^{\xi}_t$ in \eqref{IJ}  with $\eta^{n}_t$ and 
$\xi^{n}_t$ in place of $\eta_t$ and $\xi_t$, respectively,  and the operator 
 $I_t^{\xi^{n}}$ is  
defined as $I_t^{\xi}$ in \eqref{IJ} with $\xi_t^{n}$ in place of 
$\xi_t$. For each $n$ let $\gamma^n$ denote the solution to 
$d\gamma^n_t=-\gamma^n_t\beta^n_t\,dV_t$, $\gamma^n_0=1$. 
By virtue of Step II  \eqref{2.5.4.22} 
has an $L_p$-solution 
$u^n=(u^n_t)_{t\in[0,T]}$, which is also its unique $L_2$-solution, i.e., 
for each $\varphi\in C_0^{\infty}$ almost surely 
\begin{align}                                                                       
(u_t^{n},\varphi)=&(\pi^n_0,\varphi)+\int_0^t(u_s^{n},\tilde\cL_s^{n}\varphi)\,ds
+\int_0^t(u_s^{n},\cM_s^{n k}\varphi)\,dV^k_s
+\int_0^t\int_{\frZ_0}(u_s^{n},J_s^{\eta^{n}}\varphi)\,\nu_0(d\frz)ds                       \nonumber\\
&+\int_0^t\int_{\frZ_1}(u_s^{n},J_s^{\xi^{n}}\varphi)\,\nu_1(d\frz)ds
+\int_0^t\int_{\frZ_1}(u_s^{n},I_s^{\xi^{n}}\varphi)\,\tilde N_1(d\frz,ds)              \label{4.5.4.22}
\end{align}
for all $t\in[0,T]$. Moreover,  almost surely 
$u^n_t = d\mu^n_t/dx$ for all $t\in[0,T]$, 
where $\mu^n_t=P^n_t(^o\!\gamma^n_t)^{-1}$ 
is the unnormalised conditional distribution, $P^n_t$ is the regular conditional 
distribution of $X_t^n$ given $\cF^{Y^n}_t$, and $^o\!\gamma^n$ denotes 
the $\cF_t^{Y^n}$-optional projection of $\gamma^n$ under $P$.  
Furthermore, for sufficiently large $n$ we have 
\begin{equation}                                                                                            \label{8.5.4.22}
\E\sup_{t\in[0,T]}|u^n_t|_{L_r}^r\leq N|\pi_0^n|_{\bL_r}\quad\text{for $r=p,2$}
\end{equation}
with a constant 
$N=N(d, d',p,K, K_{\xi}, K_{\eta}, L, T,|\bar\xi|_{L_2}, |\bar\eta|_{L_2},\lambda)$, 
which together with \eqref{1.5.4.22} implies 
$$
\sup_{n\geq1}(|u^n_T|_{\bL_r}+|u^n|_{\bL_{r,q}})<\infty
\quad
\text{for $r=2,p$ and every $q>1$}.
$$
Hence there exist a subsequence, denoted again by $(u^n)_{n=1}^\infty$, 
$\bar u\in\bigcap_{q=2}^{\infty}\bL_{r,q}$ and $g\in\bL_r$ for $r=2,p$ such that 
\begin{equation}                                                                                  \label{25.10.21.1}
u^n\rightarrow \bar u
\quad
\text{weakly in $\bL_{r,q}$ for $r=p,2$ and all integers $q>1$}, 
\end{equation}
and
\begin{equation}                                                                                  \label{25.10.21.1}
u^n_T\rightarrow g\quad\text{weakly in $\bL_{r}$ 
for $r=p,2$}.
\end{equation}
One knows, see e.g. \cite{GK1980}, that $(X^n_t,Y^n_t)_{t\geq0}$ 
converges to $(X_t,Y_t)_{t\geq0}$ in probability, uniformly in $t$ in finite intervals. 
Hence it is not difficult to show (see Lemma 3.8 in \cite{GW2021}) that
there is a subsequence of $Y^n$, denoted for simplicity also by $Y^n$, 
and there is an $\cF_t$-adapted cadlag process $(U_t)_{t\in[0,T]}$, such that 
almost surely $|Y^n_t|+|Y_t|\leq U_t$ for every $t\in[0,T]$ and integers  $n\geq1$. 
For every integer $m\geq1$ define the stopping time 
$$
\tau_m=\inf\{t\in[0,T]: U_t\geq m\}
$$
To show that $\bar u$ is a $\bV_r$-solution for $r=p,2$ 
 to \eqref{equdZ} with initial condition $u_0=\pi_0$,   
 we pass to the limit $u^n\to \bar u$ in equation \eqref{4.5.4.22} 
 in a similar way to that as we passed to the limit $u^{\varepsilon_n}\to \bar u$ 
 in equation \eqref{equdZ*} in the proof of Lemma \ref{lemma compact}. 
 We fix an integer $m\geq1$ and multiply both sides of \eqref{4.5.4.22} with 
 $(\phi_t{\bf1}_{t\leq\tau_m})_{t\in[0,T]}$, where $(\phi_t)_{t\in[0,T]}$ is 
 an arbitrary bounded $\cF_t$-optional process $\phi=(\phi_t)_{t\in[0,T]}$. 
 Then we integrate both sides of the equation we obtained  
 over $\Omega\times[0,T]$ against $P\otimes dt$ to get 
\begin{equation}                                                                                    \label{14.10.21.4}
F(u^n)=F(\pi^n_0) + \sum_{i=1}^5 F_i^n(u^n),
\end{equation}
where $F$ and $F_i^{n}$ are linear functionals over $\bL_{r,q}$, 
defined by 
$$
F(v):=\E\int_0^{T\wedge\tau_m}\phi_t(v_t,\vp)\,dt,
\quad 
F_1^{n} (v)
:=\E\int_0^{T\wedge\tau_m}\phi_t\int_0^t(v_s,\tilde\cL_s^{n}\varphi)\,ds\,dt,
$$
$$
F_2^{n}(v)
:=\E\int_0^{T\wedge\tau_m}\phi_t\int_0^t(v_s,\cM_s^{n k}\varphi)\,dV^k_s\,dt,
\quad 
F_3^{n}(v):=\E\int_0^{T\wedge\tau_m}\phi_t\int_0^t\int_{\frZ_0}
(v_s,J_s^{\eta^{n}}\varphi)\,\nu_0(d\frz)ds\,dt,
$$
\begin{equation*}                                                                                     
F_4^{n}(v)
:=\E\int_0^{T\wedge\tau_m}\phi_t\int_0^t\int_{\frZ_1}
(v_s,J_s^{\xi^{n}}\varphi)\,\nu_1(d\frz)ds\,dt,
\end{equation*}
$$
 F_5^{n}(v)
 :=\E\int_0^{T\wedge\tau_m}\phi_t\int_0^t\int_{\frZ_1}
 (v_s,I_s^{\xi^{n}}\varphi)\,\tilde N_1(d\frz,ds) \,dt
$$
for a fixed $\varphi\in C_0^{\infty}$. Define also $F_i$ as $F_i^{n}$ 
for $i=1,2,...,5$, 
with $\tilde\cL_s$, $\cM_s^k$, $J_s^{\eta}$, $J^{\xi}_s$ and $I_s^{\xi}$ in place of 
$\tilde\cL^{n}_s$, $\cM_s^{n k}$, 
$J_s^{\eta^{n}}$, $J_s^{\xi^{n}}$ 
and $I_s^{\xi^{n}}$, respectively.  It is an easy exercise to show that 
hat $F$ and $F_i^{n}$, $i=1,2,3,4,5$, are 
continuous linear functionals on $\bL_{r,q}$ for $r=p,2$ and all $q>1$. 
We are going to show now that for $r=p,2$ 
\begin{equation}                                              \label{7.5.4.22}
\lim_{n\to\infty}
\sup_{|v|_{\bL_{r,q}} = 1}|F_i(v)-F_i^{n}(v)|=0
\quad
\text{for every $q>1$, for i=1,2,...,5}.
\end{equation}
Let $r'=r/(r-1)$, $q'=q/(q-1)$. Then for $v\in\bL_{r,q}$ 
by H\"older's inequality we have 
\begin{equation}                                           \label{1.6.4.22}                                 
|F_1(v)-F^n_1(v)|
\leq KT|v|_{\bL_{r,q}}
|(\tilde\cL-\tilde\cL^n)\varphi|_{\bL_{r',q'}}
\end{equation}
with 
$K=\sup_{\omega\in\Omega}\sup_{t\in[0,T]}|\phi_t|<\infty$.  
Clearly, 
$
\lim_{n\to\infty}(\tilde\cL_s-\tilde\cL^n_s)\varphi(x)=0
$
almost surely for all $s\in[0,T]$ and $x\in\bR^d$, and 
there is a constant $N$ independent of $n$ and $m$ such that 
\begin{equation}                                                                 \label{3.6.4.22}
|(\tilde\cL_s-\tilde\cL^n_s)\varphi(x)|
\leq N(1+|x|^2+2m^2){\bf1}_{|x|\leq R}
\end{equation}
for $\omega\in\Omega$, $s\in[0,T\wedge\tau_m]$ and  $x\in\bR^d$, 
where $R$ is the diameter of the support of $\varphi$.  
Hence a repeated application of Lebesgue's theorem 
on dominated convergence gives 
$$
\lim_{n\to\infty}|(\tilde\cL-\tilde\cL^{n})\varphi|_{\bL_{r',q'}}=0, 
$$
and 
by \eqref{1.6.4.22} proves \eqref{7.5.4.22} for $i=1$. 
By the Davis inequality and H\"older's inequality we have 
\begin{equation*}                            
|F_2(v)-F^n_2(v)|
\leq 3KT
\E\Big(
\int_0^{T\wedge\tau_m}\sum_k|(v_s,(\cM^k_s-\cM^{nk}_s)\varphi)|^2\,ds
\Big)^{1/2}
\leq C_n^{(2)}|v|_{\bL_{r,q}} 
\end{equation*}
with
$$
C_n^{(2)}=3KT\Big(\E\Big(
\int_0^{T\wedge\tau_m}
\big(
\sum_k
|(\cM^k_s-\cM^{nk}_s)\varphi|^{2}_{L_{r'}}\Big)^{q/(q-2)}\,ds
\Big)^{r'(q-2)/2q}\Big)^{1/r'}.
$$
Clearly, 
$\lim_{n\to\infty}(\cM^k_s-\cM^{nk}_s)\varphi(x)=0$, 
and with a constant $N$ independent of $n$ and $m$ we have 
$$
\sum_k|(\cM^k_s-\cM^{nk}_s)\varphi(x)|
\leq 
\sup_{s\in[0,T]}N(1+|x|+2m){\bf1}_{|x|\leq R}
$$
for all $\omega\in\Omega$, $s\in[0,T\wedge\tau_m]$ and 
$x\in\bR^d$. 
Thus repeating the above argument we obtain 
\eqref{7.5.4.22} for $i=2$. By H\"older's inequality we have 
$$
|F^3_n(v)-F^3(v)|\leq KT|v|_{\bL_{r,q}}C^{(3)}_n
$$
with 
$$
C^{(3)}_n=
\Big(
\E\Big(\int_0^{T\wedge\tau_m}\big|\int_{\frZ_0}
|(J_s^{\eta}-J_s^{\eta^n})\varphi|_{L_{r'}}\nu_0(d\frz)\big|^{q'}\,ds
\Big)^{r'/q'}
\Big)^{1/r'}, 
$$
where we have suppressed the variable $\frz\in\frZ_0$ in the integrand. 
Clearly, 
$$
\lim_{n\to\infty}(J^{\eta}-J^{\eta^n})\varphi(x)=0
\quad
\text{almost surely for all $s\in[0,T]$, $x\in\bR^d$ and $\frz\in\frZ_0$}.
$$
By Taylor's formula 
$$
|J_s^{\eta^n}\varphi(x)|
\leq \sup_{\theta\in[0,1]}|D^2\varphi(x+\theta\eta^n_s(x,\frz))||\eta_s(x,\frz)|^2,
$$
$$
|J_s^{\eta}\varphi(x)|
\leq \sup_{\theta\in[0,1]}|D^2\varphi(x+\theta\eta_s(x,\frz))||\eta_s(x,\frz)|^2,
$$
and by Lemma \ref{lemma Sandy}  with $\lambda'$ from above we have 
$$
\lambda'|x|\leq |x+\theta(\eta^n_s(x,\frz)-\eta^n_s(0,\frz))|
\leq |x+\theta\eta^n_s(x,\frz)|+|\eta^n_s(0,\frz)|,
$$
$$
\lambda'|x|\leq |x+\theta(\eta_s(x,\frz)-\eta_s(0,\frz))|
\leq |x+\theta\eta_s(x,\frz)|+|\eta_s(0,\frz)|
$$
for all $\theta\in[0,1]$, 
$\omega\in\Omega$, $s\in[0,T\wedge\tau_m]$. 
Hence, taking into account the the linear growth condition on $\eta$, 
see Assumption \eqref{assumption SDE} (ii), 
for any given $R>0$ we have a constant 
$\tilde{R}=\tilde R(R,K_0,K_1,K_\eta,m)>R$ such that 
$$
|x+\theta\eta_s(x,\frz)|\geq  R,  
\quad
|x+\theta\eta^n_s(x,\frz)|\geq R
\quad
\text{for $|x|\geq \tilde{R}$}, 
$$
for all $\theta\in[0,1]$, 
$\omega\in\Omega$, $s\in[0,T\wedge\tau_m]$. 
Taking $R$ such that $\varphi(x)=0$ for $|x|\geq R$ we have
$$
|J_s^{\eta^n}\varphi(x)-J_s^{\eta}\varphi(x)|
\leq |J_s^{\eta^n}\varphi(x)|+|J_s^{\eta}\varphi(x)|
\leq 2\sup_{x\in\bR^d}|D^2\varphi(x)|\bar \eta^2(\frz){\bf1}_{|x|\leq \tilde{R}} 
$$
for $x\in\bR$, 
$\omega\in\Omega$, $s\in[0,T\wedge\tau_m]$ and $\frz\in\frZ_0$. 
Hence by Lebesgue's theorem on dominated convergence 
$\lim_{n\to\infty}C_n^{(3)}=0$ which gives \eqref{7.5.4.22} 
for $i=3$. We get \eqref{7.5.4.22} for $i=4$ in the same way. 
By the Davis inequality and H\"older's inequality we have 
\begin{equation*}                            
|F_5(v)-F^n_5(v)|
\leq 3KT
\E\Big(
\int_0^{T\wedge\tau_m}\int_{\frZ_1}|(v_s,(I^{\xi^n}_s-I^{\xi}_s)\varphi)|^2\nu_1(d\frz)\,ds
\Big)^{1/2}
\leq C^{(5)}_n|v|_{\bL_{r,q}} 
\end{equation*}
with
$$
C^{(5)}_n
=
3KT\Big
(\E\Big(
\int_0^{T\wedge\tau_m}
\Big(
\int_{\frZ_1}
|(I^{\xi^n}_s-I^{\xi}_s)\varphi|^{2}_{L_{r'}}
\nu_1(d\frz)\Big)^{q/(q-2)}\,ds
\Big)^{r'(q-2)/2q}\Big)^{1/r'}.
$$
Clearly, $\lim_{n\to\infty}(I^{\xi^n}_s-I^{\xi}_s)\varphi(x)=0$ 
almost surely for all $s\in[0,T]$, $x\in\bR^d$ and $\frz\in\frZ_1$. 
By Taylor's formula 
$$
|I_s^{\xi^n}\varphi(x)|
\leq \sup_{\theta\in[0,1]}|D\varphi(x+\theta\xi^n_s(x,\frz))||\xi_s(x,\frz)|, 
$$
$$
|I_s^{\xi}\varphi(x)|
\leq \sup_{\theta\in[0,1]}|D\varphi(x+\theta\xi_s(x,\frz))||\xi_s(x,\frz)|. 
$$
Hence, using Assumptions 
\ref{assumption SDE}, 
\ref{assumption p} and 
\ref{assumption estimates} 
in the same way as above, 
we get a constant 
$\tilde{R}=\tilde R(R,K_0,K_1,K_\eta,m)$ such that 
$$
|I_s^{\xi^n}\varphi(x)-I_s^{\xi}\varphi(x)|
\leq |I_s^{\xi^n}\varphi(x)|+|I_s^{\xi}\varphi(x)|
\leq 2\sup_{x\in\bR^d}|D\varphi(x)|\bar \xi(\frz){\bf1}_{|x|\leq \tilde R}
$$
for $x\in\bR$, 
$\omega\in\Omega$, $s\in[0,T\wedge\tau_m]$ and $\frz\in\frZ_0$. Consequently, 
by Lebesgue's theorem on dominated convergence we obtain 
\eqref{7.5.4.22} for $i=5$, which completes the proof of \eqref{7.5.4.22}. 
Since $u^{n}$ converges weakly to $\bar u$ in $\bL_{r,q}$,  
and $F^{n}_i$ converges strongly to $F_i$ in $\bL^{\ast}_{p,q}$, 
the dual of $\bL_{p,q}$,  
we get that $F_i^{n}(u^{n})$ converges to 
$F_i(\bar u)$ for for $i=1,2,3,4,5$. 
Therefore letting $n\to\infty$  in  \eqref{14.10.21.4} we obtain
$$
\E\int_0^{T\wedge\tau_m}\phi_t(\bar u_t,\vp)\,dt
=\E\int_0^{T\wedge\tau_m}\phi_t(\psi,\vp)\,dt 
+  \E\int_0^{T\wedge\tau_m}
\phi_t\int_0^t(\bar u_{s},\tilde\cL_s\varphi)\,ds\,dt
$$
\begin{equation*}                                                                   
+ \E\int_0^{T\wedge\tau_m}
\phi_t\int_0^t(\bar u_{s},\cM_s^{k}\varphi)\,dV^k_s\,dt
+ \E\int_0^{T\wedge\tau_m}
\phi_t\int_0^t\int_{\frZ_0}(\bar u_{s},J_s^{\eta}\varphi)\,\nu_0(d\frz)ds\,dt  
\end{equation*}
$$
+\E\int_0^{T\wedge\tau_m}
\phi_t\int_0^t\int_{\frZ_1}(\bar u_{s},J_s^{\xi}\varphi)\,\nu_1(d\frz)ds\,dt
+\E\int_0^{T\wedge\tau_m}
\phi_t\int_0^t\int_{\frZ_1}(\bar u_{s},I_s^{\xi}\varphi)\,\tilde N_1(d\frz,ds) \,dt.
$$
Since this equation holds for all bounded $\cF_t$-optional processes 
$\phi=(\phi_{t})_{t\in[0,T]}$, we get 
$$
{\bf1}_{t\leq\tau_m}(\bar u_t,\vp)
={\bf1}_{t\leq\tau_m}\left((\psi,\vp)
+\int_0^t(\bar u_{s},\tilde\cL_s\varphi)\,ds
+\int_0^t(\bar u_{s},\cM_s^{k}\varphi)\,dV^k_s\right)
$$
$$
+{\bf1}_{t\leq\tau_m}\left(\int_0^t\int_{\frZ_0}(\bar u_{s},J_s^{\eta}\varphi)\,\nu_0(d\frz)ds 
+\int_0^t\int_{\frZ_1}(\bar u_{s},J_s^{\xi}\varphi)\,\nu_1(d\frz)ds
+\int_0^t\int_{\frZ_1}(\bar u_{s},I_s^{\xi}\varphi)\,\tilde N_1(d\frz,ds)\right)
$$
for $P\otimes dt$-almost every $(t,\omega)\in[0,T]\times\Omega$ 
for every $\varphi\in C_0^{\infty}$ and integer $m\geq1$, which implies 
that $\bar u$ is a $\bV_r$-solution to \eqref{equdZ} for $r=2,p$. 
In the same way as in the proof  Lemma \ref{lemma compact} we can show 
first that almost surely
\begin{equation}                                                                        \label{9.5.4.22}
{\bf1}_{\tau_m>T}(g,\vp)={\bf1}_{\tau_m>T}\left((\psi,\vp) 
+\int_0^T(\bar u_{s},\tilde\cL_s\varphi)\,ds
+ \int_0^T(\bar u_{s},\cM_s^{k}\varphi)\,dV^k_s\right)
\end{equation}
$$
+{\bf1}_{\tau_m>T}\left( \int_0^T\int_{\frZ_0}(\bar u_{s},J_s^{\eta}\varphi)\,\nu_0(d\frz)ds  
+\int_0^T\int_{\frZ_1}(\bar u_{s},J_s^{\xi}\varphi)\,\nu_1(d\frz)ds
+\int_0^T\int_{\frZ_1}(\bar u_{s},I_s^{\xi}\varphi)\,\tilde N_1(d\frz,ds)
\right) 
$$
for every $m\geq1$. Hence taking into account 
$P(\cup_{m=1}^{\infty}\{\tau_m>T\})=1$, 
we get that equation \eqref{9.5.4.22} remains valid if we omit ${\bf1}_{\tau_m>T}$ everywhere in it. 
From \eqref{8.5.4.22} we get that for sufficiently large $n$ 
$$
|u^n|_{\bL_{r,q}}\leq N|\pi^n_0|_{\bL_r} \quad\text{for $r=2,p$, for integers $q>1$}
$$
with a constant 
$N=N(d, d',p,K,  K_{\xi}, K_{\eta}, L, T,|\bar\xi|_{L_2}, |\bar\eta|_{L_2},\lambda)$. 
Letting here $n\to\infty$ and taking into account \eqref{1.5.4.22} and \eqref{25.10.21.1} 
we obtain 
$$
|\bar u|_{\bL_{r,q}}\leq \liminf_{n\to\infty}|u^n|_{\bL_{r,q}}
\leq N\lim_{n\to\infty}|\pi^n_0|_{L_r}\leq N|\pi_0|_{\bL_r}
 \quad
 \text{for $r=2,p$, and for integers $q>1$}. 
$$
Letting here $q\to\infty$ we get 
$$
\E\essup_{t\in [0,T]}|\bar u_t|_{\bL_r}^r\leq N^r\E|\pi_0|^r_{\bL_r},\quad \text{for $r=2,p$}. 
$$
Hence, taking into account \eqref{9.5.4.22}, by Lemma \ref{lemma weakly cadlag} 
we get a $P\otimes dt$-modification $u$ of $\bar u$, which is an $L_r$-solution 
for $r=2,p$ to equation \eqref{equdZ} with initial condition $u_0=\pi_0$. As the limit of $P\otimes dt\otimes dx$-almost everywhere nonnegative functions, $u$ is also $P\otimes dt\otimes dx$ almost everywhere nonnegative.
We now show that $u$ satisfies 
\begin{equation}
\label{24.10.21.4}
G(u):=\sup_{t\in [0,T]}\int_{\bR^d}|x|^2u_t(dx)<\infty \,(a.s.).
\end{equation}
To show this recall that for each $n$ and $\vp\in C_b^2$, 
by Theorem \ref{theorem Z1}, Remark \ref{remark gamma} and by what we have proven above,
$$
\mu^n_t(\vp) = P_t^n(\vp)\mu^n_t({\bf1}) 
=\E(\vp(X_t^n)|\cF_t^{Y^n})({^o\!\gamma}^n_t)^{-1},
$$
where $\mu^n_t(dx) = u^n_t(x)dx$, $P^n_t(dx) 
= \pi^n_t(x)dx$ and ${^o\!\gamma}^n$ denotes the 
$\cF_t^{Y^n}$-optional projection of $(\gamma_t^n)_{t\in[0,T]}$. 
Further, for integers $m\geq 1$ let again 
$\Omega_m:= \{|Y_0|\leq m\} \in \cF_0^Y$. 
Thus by Doob's inequality and Jensen's inequality for optional projections, 
for $r>1$ we have, in the same way as in Step I, 
$$
G_{m}(u^n)
:=\E\sup_{t\in [0,T]}\int_{\bR^d}|x|^{2}u^n_t(x)\,dx{\bf1}_{\Omega_m} 
= \E\sup_{t\in [0,T]}
\E(|X_t^n|^{2}|\cF_t^{Y^n})({^o\!\gamma}^n_t)^{-1}{\bf1}_{\Omega_m}
$$
$$
\leq N\big(\E\sup_{t\in [0,T]}|X_t^n|^{r}{\bf1}_{\Omega_m} \big)^{2/r} 
\quad
\text{for $t\in[0,T]$}  
$$
with a constant $N = N(r,C)$, where $C$ is the constant 
from \eqref{24.10.21.3}, which  depends only on $K$, $r$ and $T$.
Taking $r$ from Assumption \ref{assumption nu}, by Young's inequality, 
\eqref{bound_Z} 
for all $m$ and $n$ we have 
\begin{equation}                                                               \label{21.10.2021.2}
G_{m}(u^n)\leq N\big(m^{r}+\E|X^n_0|^{r})\big)
\leq N\big(m^{r}+\sup_n\E|X^n_0|^{r}\big)=:N'(m)<\infty. 
\end{equation}
By Mazur's theorem there exists a sequence of convex linear 
combinations $v^k = \sum_{i=1}^{k}c_{i,k}u^i$ converging to $u$ 
(strongly) in $\bL_{p,q}$ as $k\to\infty$. 
Thus there exists a subsequence, also denoted 
by $(v^k)_{k=1}^{\infty}$ which converges to $u$ 
for $P\otimes dt\otimes dx$-almost every $(\omega,t,x)$. 
Then, by Fatou's lemma and \eqref{21.10.2021.2},
$$
G_m(u) = 
\E\sup_{t\in [0,T]}\int_{\bR^d}|x|^{2}
\liminf_{k\to\infty} v_t^{k}(x)\,dx
{\bf1}_{\Omega_m}
\leq \liminf_{k\to\infty}G_m(v^{k})
$$
$$
 = \liminf_{k\to\infty}\sum_{i=1}^{k} c_{k,i}G_{m}(u^i)
 \leq N'(m)
 \quad
 \text{for each integer $m\geq1$},
$$
which proves \eqref{24.10.21.4}. Next, due to 
Lemma \ref{lemma E sup L1}, using $|B^n|\leq|B|\leq K$, we have 
\begin{equation}
\sup_{n\in\bN}\E\sup_{t\in[0,T]}|u_t^n|_{L_1}\leq N,
\label{equ L1 sup}
\end{equation}
for a constant $N=N(d,K,T)$.
The estimate above implies that $u^n \in \bL_{1,q}$ for all $q\geq 1$.
Returning to the sequence $(v^k)_{k\in\bN}\subset\bL_{1,q}\cap\bL_{p,q}$ 
converging point-wise to $u$ for 
$P\otimes dt\otimes dx$-almost every $(\omega,t,x)$, 
we can compute by use of Fatou's lemma
$$
\E\essup_{t\in [0,T]}|u_t|_{L_1} 
= \E\essup_{t\in [0,T]}|\liminf_{k\to\infty}v^k_t|_{L_1}
\leq \liminf_{k\to\infty}\E\sup_{t\in [0,T]}|v^k_t|_{L_1}
$$
\begin{equation}
\label{6.11.2021.2}
\leq\liminf_{k\to\infty}\sum_{i=1}^kc_{i,k}\E\sup_{t\in [0,T]}|u^i_t|_{L_1}\leq N,
\end{equation}
with the constant $N$ from \eqref{equ L1 sup}. 
As also \eqref{24.10.21.4} holds and since $u$ is 
in particular an $L_2$-solution to \eqref{equdZ} 
we can apply Lemma \ref{lemma usup} to see 
that indeed for all $t\in [0,T]$, $u_t = d\mu_t/dx$ 
almost surely and thus $\pi_t = u_t{^o\!\gamma}_t$. 
This finishes the proof.
\end{proof}

Inspecting the proof above, together with Lemma \ref{lemma 13.10.2021}, 
it is useful to make the following observations.

\begin{corollary}
\label{lemma convergence 2}
Let the conditions of Theorem \ref{theorem 1} hold. 
Assume the initial conditional density 
$\pi_0 = P(X_0\in dx|\cF^Y_0)$ additionally satisfies 
$\E|\pi_0|_{W^m_p}^p<\infty$ for some integer $m\geq 0$.
Then there exist sequences
$$
(X_0^n)_{n=1}^\infty, ((X_t^n,Y_t^n)_{t\in [0,T]})_{n=1}^\infty,\quad\text{as well as}\quad(\pi_0^n)_{n=1}^\infty\quad \text{and}\quad ((\pi_t^n)_{t\in [0,T]})_{n=1}^\infty
$$
such that the following are satisfied:\newline
(i) For each $n\geq 1$ the coefficients 
$b^n,B^n,\sigma^n,\rho^n,\xi^n$ and $\eta^n$, 
defined in \eqref{truncated coeffs}, 
satisfy Assumptions \ref{assumption SDE} 
and \ref{assumption p} with $K_1=0$ and constants 
$K_0'=K'_0(n,L, K, K_0, K_1,K_{\xi},K_{\eta})$ 
 and $L'=L'(K,K_0, K_1, L,K_{\xi},K_{\eta})$ 
 in place of $K_0$ and $L$, as well as
 Assumption \ref{assumption estimates} 
 with $\lambda'=\lambda'(\lambda,K_0,K_1,K_\xi,K_\eta)$ in place of $\lambda$. 
 Moreover, for each $n\geq 1$ they satisfy the support 
 condition \eqref{supp_condition} 
 of Lemma \ref{lemma compact} with some $R>0$ depending only on $n$.\newline
(ii) For each $n\geq 1$ the random variable 
$X_0^n$ is $\cF_0$-measurable and such that 
$$
\lim_{n\to\infty} X_0^n = X_0\,\, ,
\omega \in \Omega,\quad\text{and}\quad \E|X_0^n|^r 
\leq N(1+\E|X_0|^r)
$$
with a constant $N$ independent of $n$.
\newline
(iii) $Z_t^n=(X_t^n,Y_t^n)$ is the solution to \eqref{system_1} 
with the coefficients $b^n,B^n,\sigma^n,\rho^n,\xi^n$ and $\eta^n$ 
in place of $b,B,\sigma,\rho,\xi$ and $\eta$, respectively, 
and with initial condition $Z_0^n = (X_0^n,Y_0)$.
\newline
(iv) For each $n\geq 1$ we have 
$\pi_0^n = P(X_0^n\in dx|\cF^Y_0)/dx$, $\pi_0^n(x)=0$ for $|x|\geq n+1$ and 
$$
\lim_{n\to\infty}|\pi^n_0-\pi_0|_{\bW^m_p}=0.
$$\newline
(v) For each $n\geq 1$ there exists an $L_r$-solution $u^n$ to \eqref{equdZ}, $r=2,p$, with initial condition $\pi^n_0$, such that $u^n$ is the unnormalised conditional density of $X^n$ given $Y^n$, almost surely
$$
u_t^n(x)=0 \quad\text{for $dx$-a.e. $x\in\{x\in\bR^d:|x|\geq\bar R\}$ for all $t\in[0,T]$} 
$$
with a constant $\bar R=\bar R(n,K,K_0,K_\xi,K_\eta)$ and 
\begin{equation}
\label{28.5.22}
\E\sup_{t\in [0,T]}|u^n_t|_{L_p}^p\leq N\E|\pi_0^n|^p_{L_p}
\end{equation} 
with a constant 
$N=N(d, d',K, L, K_{\xi}, K_{\eta}, T, p,\lambda,|\bar\xi|_{L_2}, |\bar\eta|_{L_2})$. Moreover,
$$
u^n\rightarrow  u
\quad
\text{weakly in $\bL_{r,q}$ for $r=p,2$ and all integers $q>1$},
$$
where $u$ is the unnormalised conditional density of $X$ given $Y$, satisfying \eqref{28.5.22} with the same constant $N$ and $u$ in place of $u^n$. 
\newline
(vi) Consequently, for each $n\geq 1$ and $t\in [0,T]$ we have 
$$
\pi^n_t = P(X_t^n\in dx|\cF^{Y^n}_t)/dx = u^n_t(x){^o\!\gamma_t^n},\quad\text{almost surely},
$$
as well as
$$
\pi_t = P(X_t\in dx|\cF^{Y}_t)/dx = u_t(x){^o\!\gamma_t},\quad\text{almost surely},
$$
where ${^o\!\gamma^n}$ and ${^o\!\gamma}$  are cadlag positive normalising processes.
\end{corollary}

{\bf Acknowledgements.} The authors thank two anonymous referees for their useful comments and suggestions, which improved the quality of the exposition.\newline

{\bf Conflict of interest.} 
No funding was received to assist with the preparation of this manuscript. 
The authors declare they have no financial interests. 
Authors Alexander Davie and Fabian Germ have no other conflicts of interests to declare. 
Author Istv\'an Gy\"ongy is a member of the editorial board of this journal.

\end{document}